\documentclass[10pt]{amsart}

\usepackage[latin1]{inputenc}
\usepackage{amssymb, amscd}

\usepackage{color}
\usepackage{graphics}
\newcommand\x{\xi}

\newcommand\matr[4]{\left( {\hfill #1\@@atop\hfill #3}{\hfill
#2\@@atop\hfill #4}\right)}
\newcommand\matl[4]{\left( { #1\@@atop #3}{ #2\@@atop\hfill #4}\right)}
\makeatother

\newcommand\SL{{\mathrm{SL}}}

\newcommand\SO{{\mathrm{SO}}}
\newcommand\so{{\mathfrak{so}}}

\newcommand\End{\operatorname{End}}
\newcommand\liek{{\mathfrak k}}

\newcommand\SU{{\mathrm{SU}}}
\newcommand\lieg{{\mathfrak g}}
\newcommand\U{{\mathrm{U}}}
\newcommand\GL{{\mathrm{GL}}}

\newcommand\tr{\operatorname{tr}}

\newcommand\lieh{{\mathfrak h}}

\newcommand\CC{{\mathbb C}}

\newcommand\RR{{\mathbb R}}
\newcommand\ZZ{{\mathbb Z}}
\newcommand\z{{\mathbb Z}}
\newcommand\NN{{\mathbb N}}

\newcommand\vz[1]{\mathchoice{\left\{ #1 \right\}}{\left\{ #1
\right\}}{\{ #1 \}}{\{ #1 \}}}
\newcommand\vzm[2]{\mathchoice{\left\{\, #1 : #2 \,\right\}}{\{\, #1
:\allowbreak #2 \,\}}{\{ #1 :\allowbreak #2 \}}{\{ #1 :\allowbreak #2
\}}}
\newcommand\lw[1]{\hbox{}_{#1}\!}

\theoremstyle{plain}
\newtheorem{thm}{Theorem}[section]
\newtheorem{lem}[thm]{Lemma}

\newtheorem{prop}[thm]{Proposition}
\newtheorem{cor}[thm]{Corollary}

\newtheorem*{propD3}{Proposition \ref{D3}}
\newtheorem*{propE3}{Proposition \ref{E3}}
\theoremstyle{definition}
\newtheorem{definition}[thm]{Definition}

\theoremstyle{remark}
\newtheorem{remark}[thm]{Remark}
\newtheorem*{remarks}{Remark}

\title[Spherical functions on the $3$-sphere]{SPHERICAL FUNCTIONS ASSOCIATED TO THE THREE DIMENSIONAL SPHERE}

\author{In\'es Pacharoni}
\author{Juan Tirao}
\author {Ignacio Zurri\'an}
\address{CIEM-FaMAF, Universidad Nacional de C\'ordoba, 5000 C\'ordoba, Argentina.}
\email{\newline pacharon@famaf.unc.edu.ar, tirao@famaf.unc.edu.ar, zurrian@famaf.unc.edu.ar}
\thanks{This paper was partially supported by CONICET, PIP 112-200801-01533, and SeCyT-UNC. }

\begin{document}
\begin{abstract}
In this paper, we determine all irreducible spherical functions
$\Phi$ of any $K $-type associated to the pair
$(G,K)=(\SO(4),\SO(3))$. This is accomplished by associating to
$\Phi$ a vector valued function $H=H(u)$ of a real variable $u$,
which is analytic at $u=0$ and whose components are solutions of two
coupled systems of ordinary differential equations. By an
appropriate conjugation involving Hahn polynomials we uncouple one
of the systems. Then this is taken to an uncoupled system of hypergeometric equations,
leading to a vector valued solution $P=P(u)$, whose entries are Gegenbauer's polynomials.
Afterward, we identify those simultaneous solutions and use the representation theory of $\SO(4)$ to characterize all irreducible spherical functions.
The functions $P=P(u)$ corresponding to the irreducible spherical functions of a fixed $K$-type $\pi_\ell$ are appropriately packaged into a sequence of matrix valued polynomials $(P_w)_{w\ge0}$ of size $(\ell+1)\times(\ell+1)$. Finally we prove that $\widetilde P_w={P_0}^{-1}P_w$ is a sequence of matrix orthogonal polynomials with respect to a weight matrix $W$. Moreover, we show  that $W$ admits a second order symmetric hypergeometric operator $\widetilde D$ and a first order symmetric differential operator $\widetilde E$.
\end{abstract}

\maketitle

\section{Introduction}\label{statements}
The theory of spherical functions dates back to the classical papers of \'E. Cartan and H. Weyl; they showed that spherical harmonics arise in a natural way from the study of functions on the $n$-dimensional sphere $S^n=\SO(n+1)/\SO(n)$. The first general results in this direction were obtained in 1950 by Gelfand, who considered zonal spherical functions of a Riemannian symmetric space $G/K$. In  this case we have a decomposition $G=KAK$. When the abelian subgroup $A$ is one dimensional, the restrictions of zonal spherical functions to $A$ can be identified with hypergeometric functions, providing a deep and fruitful connection between group representation theory and special functions. In particular when $G$ is compact this gives a one to one correspondence between all zonal spherical functions of the symmetric pair $(G,K)$ and a sequence of orthogonal polynomials.

In light of this remarkable background it is reasonable to look for an extension of the above results,
 by considering matrix valued spherical functions on $G$ of a general $K$-type. This was accomplished for the first time in the case of the
 complex projective plane $P_2(\CC)=\SU(3)/\U(2)$ in \cite{GPT1}. This seminal work gave rise to a series of papers including
 \cite{GPT, GPT03, GPT05, PT1, PT2, PR}, where one considers matrix valued spherical functions associated to a compact symmetric pair $(G,K)$, arriving at sequences of matrix valued orthogonal polynomials of one real variable satisfying an explicit three-term recursion relation, which are also eigenfunctions of a second order matrix differential operator (bispectral property).

The theory of matrix valued orthogonal polynomials without any consideration of differential equations  was started by M. G. Krein in \cite{K1} and \cite{K2}. After that the theory was revived by  A. Dur\'an in \cite{D}, who posed the problem of finding matrix weight functions $W$ with symmetric matrix second order  differential operators $D$. But the existence of such ``classical pairs''
$(W,D)$ was first established in \cite{G} and \cite{GPT03b} as a byproduct of what was obtained in \cite{GPT1}.
In fact, in \cite{GPT1} for any $K$-type $\pi$ a matrix weight function $W$ of size $m=\deg \pi$, a symmetric second order differential operator $D$ and a sequence of matrix polynomials $\{P_n\}_{n\ge0}$ was constructed from the spherical functions of the pair  $(\SU(3),\U(2))$.
Such a sequence has   the following properties:
$\deg P_n=n+m$, $\det P_0\not \equiv 0$,
$P_n$ and $P_{n'}$ are orthogonal with respect to $W$ for all $n\neq n'$,  $DP_n=P_n\Lambda_n$ and the sequence  $\{P_n\}_{n\ge0}$ satisfies a three-term recursion relation.   Yet, the sequence $\{P_n\}_{n\geq 0}$ does not fit directly into the existing theory of matrix valued orthogonal polynomials as given in \cite{D}. In \cite{G} and \cite{GPT03b} we established such a connection  defining the matrix valued function $Q_n$ by means of $Q_n=P_0^{-1}P_n$.
It is worth to point out that whenever we are under these hypothesis one can prove that $\{Q_n\}_{n\geq 0}$ is a sequence of matrix valued orthogonal polynomials with respect to $\widetilde W= P_{0}^*WP_{0}$, and that
$\widetilde D= P_{0}^{-1}DP_{0}$ is symmetric.
Related results can be also found in \cite{GPT03b, PT11, GT, RT, RT2}.

A different approach to find examples of classical matrix orthogonal polynomials can be found for example in \cite{DG}.

The irreducible spherical
functions associated to the complex projective space $P_n(\CC)=\SU(n+1)/\U(n)$ of a given $K$-type are encoded in a
sequence of matrix valued orthogonal polynomials, which are given in terms of the matrix hypergeometric function.
The semi infinite matrix  corresponding to the three-term recursion relation turns out to be stochastic. This unexpected result leads to the study of the random walk with this transition probability matrix, see \cite{GPT11}.

The present paper is an outgrowth of \cite{Z} and we are currently working on the extension of these  results to the $n$-dimensional sphere and the $n$-dimensional real projective space. The starting point is to describe the irreducible spherical functions as simultaneous eigenfunctions of two commutative differential operators, one of order two and the other of order one, and then  the irreducible spherical functions of the same $K$-type are  encoded in a sequence of matrix valued orthogonal polynomials.

More recently in \cite{KPR12a}  the authors studied the irreducible spherical functions of the pair $(G,K)=(\SU(2)\times\SU(2), \SU(2))$ ($\SU(2)$ embedded diagonally) as projections onto $K$-isotypic components of irreducible representations of $G$. This  approach is comparable with the construction of vector valued polynomials given in \cite{K85}. Also in \cite{KPR12b} the authors come back to the same subject but starting with the construction of the matrix orthogonal polynomials using a recursion relation and the orthogonality relations, and by ending with the differential operators.
\color{black}
The group $\SU(2)\times\SU(2)$ is the universal covering group
\color{black} of $\SO(4)$ and the image of $\SU(2)$ under the covering homomorphism is $\SO(3)$. Thus, the pairs $(\SU(2)\times\SU(2), \SU(2))$ and
$(\SO(4),\SO(3))$ are very closely related. The results of this paper and those %in \cite{KPR12a} and
in \cite{KPR12b} are in agreement,
\color{black}
%for example: our essential function $\Psi$ defined in \eqref{psi} is, up to a diagonal matrix, equal to the transposed of the function $L$
%in Section 2 of \cite{KPR12b}, the weight functions coincide and our differential operators $\widetilde D$ and $\widetilde E$
%given in Theorem ?? of \cite{KPR12a}
%are closely  related to some differential operators in \cite{KPR12b}.
see Remark \ref{comparacion} at the end of this paper for details.
%$D$ and $E$ given in Theorem ?? by \begin{align*} \widetilde D\\ \widetilde E.\end{align*}
However, the treatments are very different. %  and the works were developed independently.
%For the benefit of the reader, and by suggestion of an anonymous referee, we observe some comparisons with \cite{KPR12b}: the equation \eqref{pes} which involves Gegenbauer polynomials is comparable with the result given in Theorem 6.2, the function $L$ in Section 2 of that work is, up to a diagonal matrix, equal to the transposed of our function $\Psi$ (see \eqref{psi}), and our equation \eqref{LDU} gives a LDU-decomposition of the weight $W$ comparable with the decomposition made in that work.
%

\smallskip
\color{black}

Briefly the main results of this paper are the following. After some preliminary  developed along the first sections, in Section \ref{DE}
we are able to explicitly describe the irreducible spherical functions of the symmetric pair $(\SO(4),\SO(3))$ of a fixed $K$-type,
by a vector valued function $P=P(u)$, whose entries are certain Gegenbauer polynomials in a suitable variable $u$.
This is accomplished by uncoupling a system of second order linear differential equations using a constant matrix of Hahn polynomials, see Proposition \ref{Hahn}.

In Section \ref{depafi} it is established which are those vector valued polynomials $P=P(u)$ that correspond to irreducible spherical functions, and it is shown how to reconstruct the spherical functions out of them.

The aim of the last two sections is to build classical sequences of matrix valued orthogonal polynomials from our previous work.
In Section \ref{hyper} we define a sequence of polynomial matrices $P_w$, $w\ge0$, whose columns are the vector valued polynomials $P=P(u)$ corresponding to some specific irreducible spherical functions of the same $K$-type.
 In Section \ref{Matrix Orthogonal Polynomials} we consider the sequence $\widetilde P_w={P_0}^{-1}P_w$ and we prove that $(\widetilde P_w)_{w\ge0}$ is a sequence of matrix orthogonal polynomials with respect to a weight function $W$ explicitly given in (\ref{LDU}). Moreover, the matrix differential operators $\widetilde D$ and $\widetilde E$ given in Theorem \ref{hyp} satisfy $\widetilde D\widetilde P_w=\widetilde P_w\Lambda_w$ and $\widetilde E\widetilde P_w=\widetilde P_w M_w$, where the eigenvalues $\Lambda_w$ and $M_w$ are real diagonal matrices. Thus, $\widetilde D$ and $\widetilde E$ are symmetric with respect to $W$.

\medskip\color{black}
\noindent {\bf Acknowledgements.}  We would like to thank the referee for many useful comments
and suggestions that helped us to improve a first version of this paper. In particular she or he pointed out that  %\eqref{columnortog}
 equation \color{black} \eqref{LDU} \color{black} gives an LDU-decomposition of the matrix weight $W$; provided the more elegant proof of Lemma \ref{referee}; asked for an explicit expression of the sequence defined in equation \eqref{ttr} in terms of known discrete orthogonal  polynomials, see Proposition \ref{ajwkRacah};
\color{black}
and pointed out that our essential function $\Psi$ defined in equation \eqref{psi} is, up to a diagonal matrix, equal to the transposed of the function $L$ in Section 2 of \cite{KPR12b}.
\color{black}

\section{Preliminaries}\label{sec:prelim}

\subsection{Spherical functions }

\

Let $G$ be a locally compact unimodular group and let
$K$ be a compact subgroup of $G$. Let $\hat K$ denote
the set of all
equivalence classes of complex finite dimensional
irreducible representations of $K$; for each
$\delta\in \hat K$, let
$\xi_\delta$ denote the character of $\delta$,
$d(\delta)$ the degree of $\delta$, i.e. the dimension
of any representation in
the class $\delta$, and
$\chi_\delta=d(\delta)\x_\delta$. We shall choose once
and for all the Haar measure $dk$ on
$K$ normalized by $\int_K dk=1$.

We shall denote by $V$ a finite dimensional vector
space over the field $\CC$ of complex numbers and by
$\End(V)$ the space
of all linear transformations of $V$ into $V$.
Whenever we refer to a topology on such a vector space
we shall be talking
about the unique Hausdorff linear topology on it.

\begin{definition} \label{def} A spherical function $\Phi$ on $G$ of type
$\delta\in \hat K$ is a continuous function on $G$ with values in
$\End(V)$ such that
\begin{enumerate} \item[i)] $\Phi(e)=I$ ($I$= identity
transformation).

\item[ii)] $\Phi(x)\Phi(y)=\int_K
\chi_{\delta}(k^{-1})\Phi(xky)\, dk$, for all $x,y\in
G$.
\end{enumerate}
\end{definition}

The reader can find a number of general results in \cite{T} and
\cite{GV}. For our purpose it is appropriate to recall the following facts.

\begin{prop}[Proposition 1.2 in \cite{T}]\label{propesf}
If $\Phi:G\longrightarrow \End(V)$ is a spherical function
of type $\delta$
then:
\begin{enumerate}
\item[i)] $\Phi(k_1gk_2)=\Phi(k_1)\Phi(g)\Phi(k_2)$, for all
$k_1,k_2\in K$, $g\in G$.
\item[ii)] $k\mapsto \Phi(k)$ is a representation of
$K$ such that any irreducible subrepresentation
belongs to $\delta$.
\end{enumerate}
\end{prop}

Concerning the definition, let us point out that the
spherical function $\Phi$ determines its type
univocally
(Proposition \ref{propesf}) and let us say that the
number of times that $\delta$ occurs in the
representation
$k\mapsto \Phi(k)$ is called the {\em height} of $\Phi$.

A spherical function $\Phi : G \longrightarrow \End(V )$ is called {\em irreducible} if $V$ has no
proper subspace invariant by $\Phi(g)$ for all $g\in G$.

\smallskip
If $G$ is a connected Lie group, it is
not difficult to prove that any spherical function $\Phi:G\longrightarrow
\End(V)$ is
differentiable ($C^\infty$), and moreover that it is analytic. Let $D(G)$
denote the algebra of all left invariant differential operators on $G$ and let $D(G)^K$ denote the
subalgebra of all operators  in $D(G)$ which are invariant under all right
translations by elements in $K$.

\smallskip
In the following proposition $(V,\pi)$ will be a finite dimensional
representation of $K$ such that any irreducible
subrepresentation belongs to the same class $\delta\in\hat K$.

\begin{prop}\label{defeq}%(\cite{T},\cite{GV})
A function
$\Phi:G\longrightarrow \End(V)$ is a spherical function of type $\delta$ if
and only if
\begin{enumerate}
\item[i)] $\Phi$ is analytic.
\item[ii)] $\Phi(k_1gk_2)=\pi(k_1)\Phi(g)\pi(k_2)$, for all $k_1,k_2\in K$,
$g\in G$, and
$\Phi(e)=I$.
\item[iii)] $[D\Phi ](g)=\Phi(g)[D\Phi](e)$, for all $D\in D(G)^K$, $g\in G$.
\end{enumerate}
\end{prop}
\begin{proof}
If $\Phi:G\longrightarrow\End(V)$ is a spherical function of type $\delta$ then $\Phi$ satisfies iii) (see Lemma 4.2 in \cite{T}) and $\Phi$ is analytic (see Proposition 4.3 in \cite{T}). Conversely, if $\Phi$ satisfies i), ii) and iii), then $D\mapsto [D\Phi](e)$ is a representation of $D(G)^K$ and therefore $\Phi$ satisfies the integral equation ii) in Definition \ref{def}, see Proposition 4.6 in \cite{T}.
\end{proof}

Moreover, we have that the eigenvalues $[D\Phi](e)$, $D\in D(G)^K$, characterize the spherical functions $\Phi$ as stated in the following proposition.

\begin{prop}[Remark 4.7 in \cite{T}]\label{unicidad}
 Let $\Phi,\Psi:G\longrightarrow \End(V)$ be
two
spherical functions on a connected
Lie group $G$ of the same type $\delta\in K$. Then $\Phi=\Psi$ if and
only if
$(D\Phi)(e)=(D\Psi)(e)$ for all $D\in D(G)^K$.
\end{prop}

Let us observe that if $\Phi:G\longrightarrow \End(V)$ is a spherical
function, then $\Phi:D\mapsto
[D\Phi](e)$ maps $D(G)^K$ into $\End_K(V)$ ($\End_K(V)$ denotes the space of
all linear maps of $V$ into $V$ which commutes with $\pi(k)$ for all $k\in K$) defining a finite
dimensional representation of the associative algebra $D(G)^K$. Moreover, the spherical
function is irreducible if and only if the representation $\Phi: D(G)^K\longrightarrow \End_K (V)$
is irreducible. In fact, if $W<V$ is $\Phi(G)$-invariant, then clearly $W$ is invariant as a $(D(G)^K,K)$-module. Therefore, if $\Phi:D(G)^K\longrightarrow\End_K(V)$
is irreducible then the spherical function $\Phi$ is irreducible. Conversely, if $\Phi:D(G)^K\longrightarrow\End_K(V)$ is not irreducible, then there exists a proper $(D(G)^K,K)$-invariant subspace $W<V$. Let $P:V\longrightarrow W$ be any $K$-projection. Let us now consider the following functions: $P\Phi P$ and $\Phi P$. Both are analytic and $(P\Phi P)(k_1gk_2)=\pi(k_1)(P\Phi P)(g)\pi(k_2)$ for all $g\in G$ and $k_1,k_2\in K$. Moreover, if $D\in D(G)^K$ then $[D(P\Phi P)](e)=P[D(\Phi)](e)P=[D(\Phi)](e)P=[D(\Phi P)](e)$. Therefore, using Remark 4.7 in \cite{T} it follows that $P\Phi P=\Phi P$. This implies that $W$ is $\Phi(G)$-invariant. Hence, if $\Phi$ is an irreducible spherical function, then $\Phi:D(G)^K\longrightarrow\End_K(V)$ is an irreducible representation.

As a consequence of this we have:
\begin{prop}\label{height1} Let $G$ be a connected reductive linear Lie group. Then the following properties are equivalent:
\begin{enumerate}
\item[i)] $D(G)^K$ is commutative.
\item[ii)] Every irreducible spherical function of $(G,K)$ is of height
one. \end{enumerate}
\end{prop}

\begin{lem}\label{separation} Let $G$ be a linear Lie group. Given $D\ne0$, $D\in D(G)$, there exists a finite dimensional representation $U$ of $G$ such that $[DU](e)\ne0$.
\end{lem}
\begin{proof} We may assume that $G$ is a Lie subgroup of $\SL(E)$ for certain real finite dimensional vector space $E$. The identity representation of $G$ extends in the usual way to a representation $U_s$ of $G$ on $E_s=\otimes^s E$.  Let $U_s$ also denote the corresponding representation of the universal enveloping algebra $\mathcal U(\lieg)$ of the Lie algebra $\lieg$ of $G$. Then as Harish-Chandra showed (see $\S$2.3.2 of \cite{W}) there exists $s\in \NN$ such that $U_s(D)\ne0$. Finally, by using the canonical isomorphism $\mathcal U(\lieg)\simeq D(G)$, we obtain $[DU_s](e)=U_s(D)\ne0$.
\end{proof}

\begin{proof}[Proof of Proposition \ref{height1}]
 i) $\Rightarrow$ ii). If $\Phi$ is an irreducible spherical function then  $\Phi:D(G)^K\longrightarrow \End_K(V)$ is an irreducible representation. Therefore, $\End_K(V)\simeq\CC$ which is equivalent to $\Phi$ being of height one.

ii) $\Rightarrow$ i). If $\Phi$ is a spherical function of height one and $D\in D(G)^K$, then $[D\Phi](e)=\lambda I$ with $\lambda \in \CC$. Hence, if $D_1, D_2 \in D(G)^K$ we have
$$[(D_1D_2)\Phi](e)=[D_1\Phi](e)[D_2\Phi](e)=[(D_2D_1)\Phi](e).$$

On the other hand, we have that the irreducible spherical functions of $(G,K)$ separate the elements of $D(G)^K$. In fact, if $D\ne0$, $D\in D(G)^K$, by Lemma \ref{separation} there exists a finite dimensional representation $U$ of $G$ such that $[DU](e)\ne0$. By hypothesis we may assume that $U$ is irreducible. Let $U=\oplus_{\delta\in \hat K}U_\delta$ be the decomposition of $U$ into $K$-isotypic components and let $P_\delta$ be the corresponding projections of $U$ onto $U_\delta$. Then, there exits $\delta\in \hat K$ such that $[D(P_\delta UP_\delta)](e)\ne0$. Thus, the corresponding spherical function $\Phi_\delta$ is irreducible and $[D\Phi_\delta](e)\ne0$. Therefore $D_1D_2=D_2D_1$.
\end{proof}

In this paper the pair $(G,K)$ is $(\SO(4),\SO(3))$.
Then, it is known that $D(G)^K$ is an abelian algebra; moreover, $D(G)^K$ is isomorphic to
$D(G)^G\otimes D(K)^K$ (See Theorem 10.1 in \cite{Kp2} or \cite{Co}), where $D(G)^G$ (resp. $D(K)^K$) denotes
the subalgebra of all
operators in $D(G)$ (resp. $D(K)$) which are invariant under all right
translations by elements in $G$ (resp.
$K$).

Now, in our case we have that $D(G)^G$ is a polynomial
algebra in two algebraically independent generators. This is because the Lie algebra of $G$ is $\mathfrak{so}(4)
\simeq \mathfrak{so}(3)\oplus\mathfrak{so}(3)$, hence, if $\Delta_1$ and $\Delta_2$ are the Casimirs of
the corresponding  $\mathfrak{so}(3)$, we have that  $\Delta_1$ and
$\Delta_2$ generate $D(G)^G$.

The first consequence of this is that all irreducible spherical
functions of our pair $(G,K)$
are of height one.
The second consequence is that to find all irreducible spherical functions of type
$\delta\in \hat K$ is
equivalent to take any irreducible representation $(V,\pi)$ of $K$ in the
class $\delta$ and to
determine all analytic functions
$\Phi:G\longrightarrow \End(V)$ such that
\begin{enumerate}
\item[i)] $\Phi(k_1gk_2)=\pi(k_1)\Phi(g)\pi(k_2)$, for
all $k_1,k_2\in K$, $g\in G$, and $\Phi(e)=I$.
\item[ii)] $[\Delta_1\Phi ](g)=\tilde\lambda\Phi(g)$, $[\Delta_2\Phi ](g)=\tilde\mu\Phi(g)$ for all
$g\in G$ and for some $\tilde\lambda,\tilde\mu\in \CC$.
\end{enumerate}

\noindent A particular choice of these
operators $\Delta_1$ and $\Delta_2$ is given in (\ref{deltas}).

\

Spherical functions of type $\delta$ (see Section 3 in \cite{T}) arise in a natural way upon
considering representations of $G$. If $g\mapsto U(g)$ is a
continuous representation of $G$, say on a finite dimensional vector
space $E$, then $$P_\delta=\int_K \chi_\delta(k^{-1})U(k)\, dk$$ is
a projection of $E$ onto $P_\delta E=E(\delta)$. If $P_\delta\ne0$ the function
$\Phi:G\longrightarrow \End(E(\delta))$ defined by
$$\Phi(g)a=P_\delta U(g)a,\quad g\in G,\; a\in E(\delta),$$
is a spherical function of type $\delta$. In fact, if $a\in E(\delta)$ we have
\begin{align*}
\Phi(x)\Phi(y)a&= P_\delta U(x)P_\delta U(y)a=\int_K \chi_\delta(k^{-1})
P_\delta U(x)U(k)U(y)a\, dk\\
&=\left(\int_K\chi_\delta(k^{-1})\Phi(xky)\, dk\right) a.
\end{align*}

If the representation $g\mapsto U(g)$ is irreducible then the associated
spherical function $\Phi$ is also irreducible. Conversely, any irreducible
spherical function on a compact group $G$ arises in this way from a finite dimensional irreducible representation of $G$.

\subsection{The groups $G$ and $K$}\label{GyK}
\

The three dimensional sphere $S^3$ can be realized as the homogeneous space $G/K$, where $G=\SO(4)$ and $K={\SO(3)}$,
where as usual we identify $\SO(3)$ as a subgroup of $\SO(4)$: for every $k$ in $K$, let $k=\left(\begin{smallmatrix} k &0\\ 0 & 1 \end{smallmatrix}\right) \in G $.

Also, we have a decomposition $G=KAK$, where $A$ is the Lie subgroup of $G$ of all elements of
the form
$$a(\theta)= \left(\begin{matrix} \cos \theta&0& 0&
\sin \theta\\ 0&1&0&0	\\ 0&0&1&0\\ -\sin \theta&0& 0&\cos
\theta\end{matrix}\right)\, ,	\qquad \theta\in \RR.$$

\smallskip
It is well known that there exists a double covering Lie homomorphism $\SO(4)\longrightarrow \SO(3)\times \SO(3)$, in particular $\mathfrak{so}(4)\simeq \mathfrak{so}(3)\oplus \mathfrak{so}(3)$. Explicitly it is obtained in the following way:
 Let  $q:\SO(4)\longrightarrow \GL\big(\Lambda^2(\RR^4)\big)$ be the Lie homomorphism defined by
 $$q(g)(e_i \wedge e_j)=g(e_i)\wedge g(e_j)\,,\qquad g\in \SO(4)\,,\qquad 1\leq i<j\leq4,$$
 where $\{e_j\}_{j =1}^4$ is the canonical basis of $\RR^4$.
Let $\dot{q}:\so(4)\longrightarrow \mathfrak{gl}\big(\Lambda^2(\RR^4)\big)$
denote the corresponding differential homomorphism.

We observe that $\Lambda^2(\RR^4)$ is reducible as  $G$-module. In fact we have the following decomposition into irreducible $G$-modules,
 $\Lambda^2(\RR^4)=V_1 \oplus V_2$, where
  $$V_1=\text{span}\{e_1 \wedge e_4+e_2 \wedge e_3 , e_1 \wedge e_3-e_2 \wedge e_4 , -e_1 \wedge e_2-e_3 \wedge e_4\},$$
$$V_2=\text{span}\{e_1 \wedge e_4-e_2 \wedge e_3 , e_1 \wedge e_3+e_2 \wedge e_4 , -e_1 \wedge e_2+e_3 \wedge e_4\}.$$
Let $P_{1}$ and $P_{2}$ be the canonical  projections onto the subspaces $V_1$ and $V_2$, respectively.
The functions defined by
$$a(g)=P_{1}\,q(g)_{\mid_{V_1}}, \qquad b(g)=P_{2}\,q(g)_{|{V_2}}, $$
are  Lie homomorphisms from $\mathrm{SO}(4)$ onto $\mathrm{SO}(V_1)\simeq \mathrm{SO}(3)$ and $\mathrm{SO}(V_2)\simeq \mathrm{SO}(3)$, respectively.
 Therefore, in an appropriate basis we have for each $g\in\SO(4)$ and for all $X\in\so(4)$
\begin{equation}\label{funcionq}
q(g)=\left(\begin{matrix}  a(g) &0\\ 0 & b(g) \end{matrix}\right), \qquad
\dot{q}(X)=\left(\begin{matrix} \dot a(X) & 0\\ 0 & \dot b(X) \end{matrix}\right).
\end{equation}

Hence, we can consider $q$ as a   homomorphism from $\mathrm{SO}(4)$ onto $\mathrm{SO}(3)\times \mathrm{SO}(3)$
with kernel $\{I,-I\}$.
%Besides, it can be proved that $a(g)=b(g)$ if and only if $g\in K$.

\subsection{The Lie algebra structure}
\
\color{black}

The cartan involution of $G$ (respectively of $\mathfrak g$) is $\Theta (g)=I_{3,1}gI_{3,1}$ (respectively $\theta (X)=I_{3,1}XI_{3,1}$), where $$I_{3,1}=\left(\begin{smallmatrix}
1&0&0&0\\0&1&0&0\\0&0&1&0\\0&0&0&-1\\
\end{smallmatrix}\right).$$
The subgroup $K$ is the connected component of the fixed points of $\Theta$ in $G$. The corresponding Cartan decomposition of the Lie algebra of $G$ is $\mathfrak g = \mathfrak k \oplus \mathfrak p $, $\mathfrak k$ being the Lie algebra of $K$ and $\mathfrak p $ being the $(-1)$-eigenspace of $\theta$.
\color{black}

A basis of $\lieg=\mathfrak{so}(4)$ over $\RR$ is given by
\begin{align*}
 Y_1=&\left( \begin{smallmatrix} 0&1&0&0 \\-1&0&0&0 \\0&0&0&0 \\0&0&0&0 \end{smallmatrix}\right),&
Y_2=&\left( \begin{smallmatrix} 0&0&1&0\\0&0&0&0 \\-1&0&0&0 \\0&0&0&0 \end{smallmatrix}\right),&
Y_3=&\left( \begin{smallmatrix} 0&0&0&0 \\0&0&1&0  \\0&-1&0&0 \\0&0&0&0 \end{smallmatrix}\right), \\
 Y_4=&\left( \begin{smallmatrix} 0&0&0&1 \\0&0&0&0 \\0&0&0&0 \\-1&0&0&0 \end{smallmatrix}\right),&
Y_5=&\left( \begin{smallmatrix} 0&0&0&0 \\0&0&0&1 \\0&0&0&0 \\0&-1&0&0 \end{smallmatrix}\right),&
Y_6=&\left( \begin{smallmatrix} 0&0&0&0 \\0&0&0&0  \\0&0&0&1 \\0&0&-1&0 \end{smallmatrix}\right).
\end{align*}
\color{black}
Observe that $\{Y_1,Y_2,Y_3\}$ is a basis of $\mathfrak k$ and $\{Y_4,Y_5,Y_6\}$ is a basis of $\mathfrak p$.

\color{black}
Consider the following vectors
$$
Z_1=\frac{1}{2}\left(Y_3+Y_4\right), \quad
Z_2=\frac{1}{2}\left(Y_2-Y_5\right), \quad
Z_3=\frac{1}{2}\left(Y_1+Y_6\right), \quad   $$
$$
Z_4=\frac{1}{2}\left(Y_3-Y_4\right), \quad
Z_5=\frac{1}{2}\left(Y_2+Y_5\right), \quad
Z_6=\frac{1}{2}\left(Y_1-Y_6\right). \quad   $$
It can be proved that these vectors define a basis of $\so(4)$ adapted to the decomposition $\so(4)\simeq \so(3)\oplus\so(3)$, i.e.
$\{Z_4,Z_5,Z_6 \}$ is a basis of the first summand and $\{Z_1,Z_2,Z_3 \}$ is a basis of the second one.

\
\smallskip
The algebra $D(G)^G$ is generated by the algebraically independent elements
\begin{equation}\label{deltas}
 \Delta_1=-Z_4^2-Z_5^2-Z_6^2\, , \qquad \qquad
 \Delta_2= -Z_1^2-Z_2^2-Z_3^2,
\end{equation}
which are the Casimirs of the first and the second $\mathfrak{so}(3)$ respectively. The Casimir of $K$ will be denoted by $\Delta_K$, and it is given by $-Y_1^2-Y_2^2-Y_3^2$.

The complexification of $\liek$  is isomorphic to $\mathfrak{sl}(2,\CC)$. If we define
\begin{equation}\label{efh} e= \left( \begin{smallmatrix} 0&i&-1 \\-i&0&0 \\1&0&0 \end{smallmatrix}\right),  \qquad
  f= \left( \begin{smallmatrix} 0&i&1 \\-i&0&0 \\-1&0&0 \end{smallmatrix}\right) , \qquad
  h= \left( \begin{smallmatrix} 0&0&0 \\0&0&-2i \\0&2i&0 \end{smallmatrix}\right)
\end{equation}
then we have that $\{e,f,h\}$ is an ${s}$-triple in $\liek_\CC$, i.e. $$[e,f]=h,\qquad[h,e]=2e,\qquad[h,f]=-2f .$$

We take as a Cartan subalgebra $\mathfrak h_\CC$ of $\mathfrak{so}(4,\CC)$ the complexification of the maximal abelian subalgebra of $\so(4)$ of all matrices of the form
\begin{align*}
H=  \begin{pmatrix}
    0   &x_1  &0&0
\\ -x_1&0     & 0    &0
\\ 0    &0     &0     &x_2
\\ 0    &0     &-x_2 &0
  \end{pmatrix}.
\end{align*}

\noindent Let $\varepsilon_j\in \mathfrak h_\CC^*$  be given by $\varepsilon_j(H)=-i x_j$
for $j=1,2$. Then
$$\Delta(\mathfrak g_\CC,\mathfrak h_\CC)=\vzm{\pm(\varepsilon_1\pm \varepsilon_2)}{\varepsilon_1, \varepsilon_2\in \mathfrak h_\CC^*},$$
and we choose as positive roots those in the set $\Delta^+(\mathfrak g_\CC,\mathfrak h_\CC)=\{\varepsilon_1-\varepsilon_2,\varepsilon_1+\varepsilon_2\}.$

We define
\begin{equation*}
\begin{split}
X_{\varepsilon_1+\varepsilon_2}=  \begin{pmatrix}
   0 & 0  & 1 & -i
\\ 0 & 0  & -i & -1
\\ -1 & i  & 0 & 0
\\ i & 1  & 0 & 0
  \end{pmatrix},\,\,\,
X_{\varepsilon_1-\varepsilon_2}&=  \begin{pmatrix}
   0 & 0  & 1 & i
\\ 0 & 0  & -i & 1
\\ -1 & i  & 0 & 0
\\ -i & -1  & 0 & 0
  \end{pmatrix},\\
X_{-\varepsilon_1+\varepsilon_2}=  \begin{pmatrix}
   0 & 0  & 1 & -i
\\ 0 & 0  & i & 1
\\ -1 & -i  & 0 & 0
\\ i & -1  & 0 & 0
  \end{pmatrix},
X_{-\varepsilon_1-\varepsilon_2}&=  \begin{pmatrix}
   0 & 0  & 1 & i
\\ 0 & 0  & i & -1
\\ -1 & -i  & 0 & 0
\\ -i & 1  & 0 & 0
  \end{pmatrix}.
\end{split}
\end{equation*}

\noindent Then, for every $H$ in $\mathfrak{h}_\CC$ we get that
 $$ [H,X_{\pm(\varepsilon_1\pm\varepsilon_2)}]=\pm(\varepsilon_1\pm\varepsilon_2)(H)X_{\pm(\varepsilon_1\pm\varepsilon_2)}.$$
Hence, each $X_{\pm(\varepsilon_1\pm\varepsilon_2)}$ belongs to the root-space $\mathfrak g_{\pm(\varepsilon_1\pm\varepsilon_2)}$.

Then, in terms of the root structure of $\mathfrak{so}(4,\CC)$, $\Delta_1$ and $\Delta_2$ become
\begin{equation} \label{Delta23}
\begin{split}
\Delta_1&=-Z_6^2+iZ_6-(Z_5+iZ_4)(Z_5-iZ_4), \\
\Delta_2&=-Z_3^2+iZ_3-(Z_2+iZ_1)(Z_2-iZ_1) .
\end{split}
\end{equation}
We observe that $(Z_5-iZ_4)=X_{\varepsilon_1-\varepsilon_2}\in \mathfrak g_{\varepsilon_1-\varepsilon_2}$ and
 $(Z_2-iZ_1)=X_{\varepsilon_1+\varepsilon_2}\in \mathfrak g_{\varepsilon_1+\varepsilon_2}$
and $ Z_3,  Z_6 \in \mathfrak h _\CC$.

\subsection{Irreducible representations of $G$ and $K$.}\label{representations}
\

Let us first consider $\SU (2)$. It is well known that the irreducible finite dimensional representations of $\SU(2)$ are, up to equivalence, $(\pi_\ell,V_\ell)_{\ell\geq0}$, where $V_\ell$ is the
complex vector space of all polynomial functions in two complex variables $z_1$ and $z_2$ homogeneous of degree $\ell$, and $\pi_\ell$ is defined by
$$\pi_\ell\left(\begin{matrix}a &b \\c & d\end{matrix}\right)\, P\left(\begin{matrix} z_1\\z_2\end{matrix}\right)=
P\left(\left(\begin{matrix}a &b \\c & d\end{matrix}\right)^{-1}\left(\begin{matrix} z_1\\z_2\end{matrix}\right)\right),
 \qquad \text{for } \left(\begin{matrix}a &b \\c & d\end{matrix}\right) \in \SU(2).$$
Hence, since there is a Lie homomorphism of $\SU(2)$ onto $\SO(3)$ with kernel $\{\pm I\}$, the irreducible representations of $\SO(3)$ correspond to those
representations $\pi_\ell$ of $\SU(2)$ with $\ell \in 2\NN_0$. Therefore, we have $\hat \SO(3)=\{[\pi_\ell]\}_{\ell\in2\NN_0}$,
 even more, if $\pi=\pi_\ell$ is any such irreducible representation of
$\mathrm{SO}(3)$,  it is well known (see \cite{Hu}, page 32) that there exists a basis $\mathcal B=\vz{v_j}_{j=0}^\ell$ of $V_\pi$
such that the corresponding representation $\dot\pi$ of the complexification of $\so(3)$ is given by
\begin{equation*}
\begin{split}
&\dot\pi(h)v_j=(\ell-2j)v_j,\\
&\dot\pi(e)v_j=(\ell-j+1)v_{j-1}, \quad (v_{-1}=0),\\
&\dot\pi(f)v_j=(j+1)v_{j+1}, \quad (v_{\ell+1}=0).
\end{split}
\end{equation*}

{ It is known (see \cite{vilenkin}, page 362) that an irreducible representation $\tau\in \hat\SO(4)$ has highest weight of the form $\eta = m_1\varepsilon_1+m_2\varepsilon_2$,}
 where  $m_1$ and $m_2$ are integers such that $m_1\geq |m_2|$. Moreover, the  representation $\tau=\tau_{(m_1,m_2)}$, restricted to $\SO(3)$, contains the representation $\pi_\ell$  if and only if
 $$m_1\geq\tfrac\ell2\geq |m_2|.$$

\subsection{$K$-orbits in $G/K$. }
\

The group $G=\SO(4)$ acts in a natural way in the sphere $S^3$.
This action is transitive and $K$ is the isotropy subgroup of the north pole
$e_4=(0,0,0,1)\in S^3$. Therefore, $ S^3\simeq G/K.$
Moreover, the $G$-action on $S^3$
corresponds to the action induced by left multiplication on $G/K$.

In the north hemisphere of $S^3$
 $$(S^3)^{+}=\vzm{x=(x_1,x_2,x_3,x_4)\in S^3}{x_4>0},$$ we will consider the coordinate system
{ $p:(S^3)^{+}\longrightarrow \RR^3$ given by the central projection of the sphere onto its tangent plane at the north pole (see Figure \ref{fig1}):}
\begin{equation}\label{pfunction}
p(x)=\left(\frac{x_1}{x_4},\frac{x_2}{x_4},\frac{x_3}{x_4}\right)=(y_1,y_2,y_3).
 \end{equation}

\begin{figure}
 \centering
\includegraphics{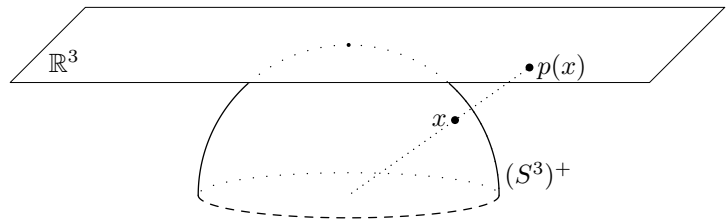}
 \caption{The central projection $p$.}
\label{fig1}
\end{figure}

{ Homogeneous coordinates were also used in the case of the complex projective plane, see \cite{GPT1}.}
The coordinate map $p$ carries the $K$-orbits in $(S^3)^+$  into the $K$-orbits in $\RR^3$, which are the spheres
 $$S_r=\vzm{(y_1,y_2,y_3)\in \RR^3}{|| y||^2=|y_1|^2+|y_2|^2+|y_3|^2=r^2}, \qquad 0\leq r <\infty.$$
Then, the interval $[0,\infty)$ parameterizes the set of $K$-orbits of $\RR^3$.

\subsection{The auxiliary function $\Phi_\pi$}\label{auxiliar}
\

As in \cite{GPT1}, to determine all irreducible spherical functions $\Phi$ of type $\pi=\pi_\ell\in \hat K$  an auxiliary function $\Phi_\pi: G\longrightarrow \End(V_\pi)$ is introduced.
In this case it is defined by
\begin{equation*}
\Phi_\pi(g)= \pi(a(g)), \qquad g\in G,
\end{equation*}
where $a$ is the Lie homomorphism from $\SO(4)$ to $\SO(3)$ given in \eqref{funcionq}.
It is clear that $\Phi_\pi$ is an irreducible representation of $\SO(4)$ and hence a spherical function of type $\pi$ (see Definition \ref{def}).

\section{The differential operators $D$ and $E$}
To determine all irreducible spherical functions on $G$ of type  $\pi\in \hat K$, it is equivalent to determine  all analytic functions
$\Phi:G\longrightarrow \End(V_\pi)$ such that
\begin{enumerate}
\item[i)] $\Phi(k_1gk_2)=\pi(k_1)\Phi(g)\pi(k_2)$, for
all $k_1,k_2\in K$, $g\in G$, and $\Phi(e)=I$.
\item[ii)] $[\Delta_1\Phi ](g)=\widetilde\lambda\Phi(g)$, $[\Delta_2\Phi ](g)=\widetilde\mu\Phi(g)$ for all
$g\in G$ and for some $\widetilde\lambda,\widetilde\mu\in \CC$.
\end{enumerate}

Instead of looking at an irreducible  spherical function
$\Phi$ of type $\pi$, we use the auxiliary function $\Phi_\pi$ to look at the
function $H:G\longrightarrow \End(V_\pi)$ defined by
\begin{equation}\label{defHg}
H(g)=\Phi(g)\Phi_\pi (g)^{-1}.
\end{equation}
We observe that $H$ is well defined on $G$ because $\Phi_\pi$ is a representation of $G$. This function $H$, associated to the spherical function $\Phi$, satisfies
\begin{enumerate}
\item [i)] $H(e)=I$.
\item [ii)] $ H(gk)=H(g)$, for all $g\in G, k\in K$.
\item [iii)] $H(kg)=\pi(k)H(g)\pi(k^{-1})$, for all
$g\in G, k\in K$.
\end{enumerate}

\smallskip
The fact that $\Phi$ is an eigenfunction of $\Delta_1 $ and $\Delta_2$ makes the function $H$ into an eigenfunction of certain differential operators $D$ and $E$ on $G$ to be determined now.
Let us define
\begin{align}
\label{Ddefuniv} D(H)&=Y_4^2(H)+Y_5^2(H)+Y_6^2(H),\\
\label{Edefuniv}
E(H)&=\big(-Y_4(H)Y_3(\Phi_\pi)+ Y_5(H)Y_2(\Phi_\pi)-Y_6(H)Y_1(\Phi_\pi) \big)\Phi_\pi^{-1}.
\end{align}

\smallskip

\begin{prop}\label{relacionautovalores}
 For any $H\in C^{\infty}(G)\otimes
\End(V_\pi)$ right invariant under $K$, the function $\Phi=H\Phi_\pi$
satisfies  $\Delta_1 \Phi=\widetilde\lambda \Phi$ and
$\Delta_2\Phi=\widetilde\mu \Phi$
 if and only if $H$ satisfies $DH=\lambda H$ and $E H=\mu H$, with
$$\lambda = -4\widetilde\lambda, \quad \mu = -\tfrac14\ell(\ell+2)+\widetilde\mu-\widetilde\lambda.$$
\end{prop}
\begin{proof}
We firstly observe that $Z_4(\Phi_\pi)=Z_5(\Phi_\pi)=Z_6(\Phi_\pi)=0$, because $\Phi_\pi$ is a representation of $G$ and $\dot a(Z_j)=0$ for $j=4,5,6$. In fact,
\begin{align*}
 Z_j(\Phi_\pi)(g)&
 =\left.\tfrac{d}{dt} \right|_{t=0} \left[\Phi_\pi(g) \Phi_\pi(\exp tZ_j) \right] = \Phi_\pi(g)  \dot\pi(\dot a (Z_j))
 =0.
\end{align*}

On the other hand, since $H$ is right invariant under $K$, we have that  $Y_1(H)=Y_2(H)=Y_3(H)=0$.
Since $[Y_3,Y_4]=0$,  $[Y_2,Y_5]=0$ and  $[Y_1,Y_6]=0$, we have that $Z_j^2(H)=\frac 14 Y_j^2(H)$, for $j=4,5,6$.
Therefore, we obtain
\begin{align*}
\Delta_1(H\Phi_\pi)&=-\sum_{j=4}^{6}Z_j^2(H)\,\Phi_\pi=-\frac{1}{4}\sum_{j=4}^{6}Y_j^2(H)\,\Phi_\pi = -\frac{1}{4}D(H) \Phi_\pi.
\end{align*}

On the other hand, we have
\begin{align*}
\Delta_2 (H\Phi_\pi)=-\sum_{j=1}^{3}\big(Z_j^2(H)\,\Phi_\pi+2Z_j(H)Z_j(\Phi_\pi)+HZ_j^2(\,\Phi_\pi)\big),
 \end{align*}
We observe that $Z_1(H)= \frac 12 Y_4(H)$. Since $Z_1=Y_3-Z_4$, we have
$Z_1(\Phi_\pi)=Y_3(\Phi_\pi)$ and $Z_1^2(\Phi_\pi)=Y_3^2(\Phi_\pi)$. Similar results hold for $Z_2$ and $Z_3$.
Therefore,
\begin{align*}
\Delta_2 (H\Phi_\pi)&=-\big(Z_1^2(H)\,\Phi_\pi+Y_4(H)Y_3(\Phi_\pi)+HY_3^2(\,\Phi_\pi)\big) \\
 &\quad -\big(Z_2^2(H)\,\Phi_\pi-Y_5(H)Y_2(\Phi_\pi)+HY_2^2(\,\Phi_\pi)\big) \\
 &\quad -\big(Z_3^2(H)\,\Phi_\pi+Y_6(H)Y_1(\Phi_\pi)+HY_1^2(\,\Phi_\pi)\big) \\
 &=-\frac14 D(H) \,\Phi_\pi +E(H)\,\Phi_\pi + H \Delta_K(\Phi_\pi)\\
 &=-\frac14 D(H)\,\Phi_\pi +E(H)\,\Phi_\pi + H \Phi_\pi\dot\pi(\Delta_K).
 \end{align*}

Since $\Delta_K\in D(G)^K$, Schur's Lemma tells us that $\dot\pi(\Delta_K)=cI$.
Now we have $\Delta_1(H\Phi_\pi)=\widetilde\lambda H\Phi_\pi$ and $\Delta_2(H\Phi_\pi)=\widetilde\mu H\Phi_\pi$ if and only if
$D(H)=\lambda H$ and $E(H)=\mu H$, where
$$\widetilde\lambda=-\frac{1}{4}\lambda  \qquad \text{ and } \qquad \widetilde\mu=c+\widetilde\lambda+\mu .  $$

In order to compute the constant $c$, we take a highest weight vector $v \in V_\pi$, and write $Y_1$, $Y_2$, $Y_3$ in terms of the basis $\{e,f,g\}$ introduced in \eqref{efh}. It follows that
\begin{align*}
\dot\pi (\Delta_K) v &=\dot\pi\left(-\big(\tfrac{-i}{2}(e+f)\big)^2-\big(\tfrac{-1}{2}(e-f)\big)^2-\big(\tfrac{i}{2}h\big)^2\right)  v \\
&= \frac{-1}{4}\dot\pi\left( - 2 e f - 2 f e -h^2 \right)v = \frac{1}{4}\dot\pi\left( 2(f e+h)+2 f e +h^2\right) v \\
&= \frac{1}{4}\left( 2\ell+\ell^2\right)v = \frac{\ell(\ell+2)}{4} \,v\,.
\end{align*}
Thus, $c={\ell(\ell+2)}/{4}$ completing the proof of the proposition.
\end{proof}

{ \begin{remark}\label{DyEconmutan}
We observe that the differential operators $D$ and $E$ commute. In fact, from the proof of Proposition \ref{relacionautovalores} we have that
  \begin{align*}
    D(H)& = -4\Delta_1(H\Phi_\pi)\Phi_\pi^{-1},\\
    E(H)& = \Delta_2(H\Phi_\pi) \Phi_\pi^{-1}+\tfrac 14 D(H)-\tfrac{\ell(\ell+2)}2 H,
  \end{align*}
and $\Delta_1$ and  $\Delta_2$ commute because they are in the center of the algebra $D(G)$.
\end{remark}
}

\subsection{Reduction to $G/K$}\label{redG/K}
\

The quotient $G/K$ is the sphere $S^3$; moreover, the canonical diffeomorphism is given by $gK\mapsto (g_{14},g_{24}, g_{34}, g_{44})\in S^3$.

The function $H$ associated to the spherical function $\Phi$ is right invariant under $K$;
then, it may be considered
 as a function on  $S^3$, which we also called $H$. The differential operators  $D$ and $E$ introduced in \eqref{Ddefuniv} and \eqref{Edefuniv},
 define differential operators on $S^3$.

\begin{lem}\label{referee} The differential operators $D$ and $E$ on $G$ define differential operators $D$
and $E$ acting on $C^\infty( S^3)\otimes\End(V_\pi)$.
\end{lem}
\begin{proof}
The only thing we need to prove is that $D$ and $E$ preserve the subspace  $C^{\infty}(G)^K\otimes\End(V_\pi)$.
\color{black}

Given an irreducible spherical function  $\Phi$ of type $\pi$ and the function $\Phi_\pi$ introduced in Subsection \ref{auxiliar}, let $H(g)=\Phi(g)\,\Phi_\pi^{-1}(g)$. Consider any $D\in D(G)^K $ and the right translation $r_k(g)=gk$. Then
\begin{multline*}
 r_k^*(DH)(g)=r_k^*(\Delta(H\Phi_\pi))(g)\pi(k)^{-1}\Phi_\pi^{-1}(g)=\\
r_k^*(\Delta)(r_k^*(H\Phi_\pi))(g)\pi(k)^{-1}\Phi_\pi^{-1}(g)=\Delta(H\Phi_\pi)(g)\Phi_\pi^{-1}(g)=DH(g),
\end{multline*}
showing that  $DH$ is right $K$-invariant.
\end{proof}
\color{black}
Now we give the expressions of the operators $D$ and $E$ in the coordinate system
{ $p:(S^3)^{+}\longrightarrow \RR^3$ introduced in \eqref{pfunction} and given by}
\begin{equation*}
p(x)=\left(\frac{x_1}{x_4},\frac{x_2}{x_4},\frac{x_3}{x_4}\right)=(y_1,y_2,y_3).
 \end{equation*}

\color{black}
 We need the following propositions which require simple but lengthy computations. The \color{black} outlines of the proofs appear in the Appendix.

\color{black}
\begin{prop}\label{D3}
For any $H\in C^\infty( \RR^3)\otimes
\End(V_\pi)$ we have
\begin{align*}
D(H)& (y) =  (1+\left\|y\right\|^2)\Big((y_1^2+1) H_{y_1y_1}+(y_2^2+1) H_{y_2y_2}+(y_3^2+1) H_{y_3y_3}\\
 &+ 2(y_1y_2 H_{y_1y_2}+y_2y_3 H_{y_2y_3}+y_1y_3 H_{y_1y_3}) + 2(y_1 H_{y_1}+y_2 H_{y_2}+y_3 H_{y_3})
\Big).
\end{align*}
\end{prop}

\begin{prop}\label {E3}
For any $H\in C^\infty( \RR^3)\otimes
\End(V_\pi)$
we have
\begin{align*}
&E(H)(y)=H_{y_1}\dot\pi \left(\begin{smallmatrix} 0 &-y_2-y_1y_3 &-y_3+y_1y_2 \\y_2+y_1y_3 &0 &-1-y_1^2 \\ y_3-y_1y_2&1+y_1^2 &0 \end{smallmatrix}\right)  \\
& \quad +H_{y_2}\dot\pi \left(\begin{smallmatrix} 0 &-y_2y_3+y_1 &1+y_2^2 \\y_2y_3-y_1 & 0 & -y_3-y_1 y_2 \\-1-y_2^2 & y_3+y_1y_2&0 \end{smallmatrix}\right) +H_{y_3}\dot\pi \left(\begin{smallmatrix} 0 &-1-y_3^2 &y_1+y_2y_3 \\1+y_3^2 & 0 & y_2-y_1 y_3 \\ -y_1-y_2y_3 & -y_2+y_1y_3 &0 \end{smallmatrix}\right).
\end{align*}
\end{prop}

\subsection{Reduction to one variable}\label{onevariable}
\

We are interested in considering the differential operators $D$ and $E$ given in Propositions \ref{D3} and \ref{E3}
applied to  functions $H\in C^\infty(\RR^3)\otimes
\End(V_\pi)$ such that $$H(ky)=\pi(k)H(y)\pi(k)^{-1}, \qquad  \quad \text {for all }k\in K , y\in\RR^3.$$
{ Hence, the function $H=H(y)$ is determined by its restriction to a section of the $K$-orbits in $\RR^3$. We recall that the $K$-orbits in $\RR^3$ are the spheres
$$S_r=\vzm{(y_1,y_2,y_3)\in \RR^3}{|| y||^2=|y_1|^2+|y_2|^2+|y_3|^2=r^2}, \qquad 0\leq r <\infty.$$
In each orbit $S_r$ we choose the point $(r,0,0) \in \RR^3$ as a representative.

This allows us to find ordinary differential operators $\widetilde
D$} and $\widetilde E$ defined on the interval
$(0,\infty)$ such that $$ (D\,H)(r,0,0)=(\widetilde D\widetilde H)(r),\qquad
(E\,H)(r,0,0)=(\widetilde E\widetilde H)(r),$$ where { $\widetilde
H(r)=H(r,0,0)$.
}

{
\begin{remark}\label{DyEconmutan2}
  We observe that the differential operators $\widetilde D$ and $\widetilde E$ commute because they are the restrictions of the commuting differential operators $D$ and $E$.
\end{remark}
}
In order to give the explicit expressions of the differential operators $\widetilde D$ and $\widetilde E$, and starting from Propositions \ref{D3} and \ref{E3}, we need to compute a number of second order partial derivatives of the function $H:\RR^3\longrightarrow \End(V_\pi)$ at the points $(r,0,0)$, with $r>0$.
Given $y=(y_1,y_2,y_3)\in \RR^3$ in a neighborhood of $(r,0,0)$, $r>0$, we need a smooth function onto $K=\SO(3)$  that carries the point $y$ to the meridian $\{(r,0,0)\, : r>0\}$. A good choice is the following function
\begin{equation}\label{matrixA}
A(y)=\frac{1}{\left\|y\right\|} \left( \begin{matrix} y_1 &-y_2 & -y_3 \\y_2 & \frac{-y_2^2}{\left\|y\right\|+y_1}+ \left\|y\right\| & \frac{-y_2 y_3}{\left\|y\right\|+y_1} \\y_3 & \frac{-y_2 y_3}{\left\|y\right\|+y_1} & \frac{-y_3^2}{\left\|y\right\|+y_1}+ \left\|y\right\| \end{matrix}\right).
\end{equation}
Then
$$y=A(y)(\left\|y\right\|,0,0)^t.$$

It is easy to verify that  $A(y)$ is a matrix in $\SO(3)$ and it is well defined in $\RR^3-\vzm{(y_1,0,0)\in\RR^3}{y_1\leq0}$.

{ The proofs of the following propositions are similar to those in the case of the complex projective plane considered in \cite{GPT1}, see Propositions 13.2 and 13.3 in that paper.}
Let us consider the following elements in $\liek$
\begin{equation}\label{As}
A_1=E_{21}-E_{12},\qquad A_2=E_{31}-E_{13},\qquad A_3=E_{32}-E_{23}.
\end{equation}

\begin{prop}\label{derx1}
For $r>0$ we have
\begin{align*}
\frac{\partial H}{\partial y_1}(r,0,0)=&\frac{d\widetilde H}{dr}(r) ,\\
\frac{\partial H}{\partial y_2}(r,0,0)=& \frac1r\left[\dot \pi \left(A_1\right) , \widetilde H(r) \right], \qquad
\frac{\partial H}{\partial y_3}(r,0,0)= \frac1r\left[\dot \pi \left(A_2\right) , \widetilde H(r) \right].
\end{align*}
\end{prop}
\medskip
\begin{prop}\label{derx1x1}
 For $r>0$ we have
\begin{align*}
\frac{\partial² H}{\partial y_1²}(r,0,0) =&\frac{d^2\widetilde H}{dr^2}(r),\displaybreak[0]\\
\frac{\partial² H}{\partial y_2²}(r,0,0)=&\frac{1}{r²}\bigg(\ r\frac{d\widetilde H}{dr}+\dot \pi \left(A_1\right)^2 \widetilde H(r)+\widetilde H(r) \dot \pi \left(A_1\right)^2
-2  \dot \pi \left(A_1\right) \widetilde H(r)\dot \pi \left(A_1\right)\bigg),\\
\frac{\partial² H}{\partial y_3²}(r,0,0)=&\frac{1}{r²} \bigg(r\frac{d\widetilde H}{dr}+\dot \pi \left(A_2\right)^2 \widetilde H(r)+\widetilde H(r) \dot \pi \left(A_2\right)^2
-2  \dot \pi \left(A_2\right) \widetilde H(r)\dot \pi \left(A_2\right)\bigg).
\end{align*}
 \end{prop}

\noindent Now we can obtain the explicit expressions of the differential operators $\tilde D$ and $\tilde E$.

\begin{thm}\label{D} For  $r>0$ we have
\begin{align*}\widetilde D(\widetilde H)(r)&=(1+r^2)^2\frac{d^2 \widetilde H}{dr^2}
+2\frac{(1+r^2)^2}{r} \frac{d\widetilde H}{dr}\displaybreak[0] \\
 & \quad +\frac{(1+r^2)}{r^2} \left( \dot\pi(A_1)^2 \,\widetilde H(r)\,+\widetilde H(r)\dot\pi(A_1)^2 -2\dot\pi(A_1)\widetilde H(r)\dot\pi(A_1)\right)\\
 & \quad +\frac{(1+r^2)}{r^2} \left( \dot\pi(A_2)^2 \,\widetilde H(r)\,+\widetilde H(r)\dot\pi(A_2)^2 -2\dot\pi(A_2)\widetilde H(r)\dot\pi(A_2)\right).
\displaybreak[0]
\end{align*}
\end{thm}

\begin{proof}
  Since $\widetilde D( \widetilde H) (r)=D(H) (r,0,0)$, from Proposition \ref{D3} we have
 \begin{align*}
\widetilde D (\widetilde H )(r) =  (1+r^2)\Big((1+r^2)H_{y_1y_1}+H_{y_2y_2}+H_{y_3y_3}+ 2rH_{y_1} \Big).	
\end{align*}
Using Propositions \ref{derx1} and \ref{derx1x1} we get
\begin{align*}
\widetilde D (\widetilde H )(r) =&  (1+r^2)\Bigg[(1+r^2)\frac{d^2\widetilde H}{dr^2}(r)+2r\frac{d\widetilde H}{dr}(r) +\frac2r\frac{d\widetilde H}{dr}\\
&+\frac{1}{r²}\bigg(\dot \pi \left(A_1\right)^2 \widetilde H(r)+\widetilde H(r) \dot \pi \left(A_1\right)^2 -2  \dot \pi \left(A_1\right) \widetilde H(r)\dot \pi \left(A_1\right)\bigg)\\
&+\frac{1}{r²} \bigg(\dot \pi \left(A_2\right)^2 \widetilde H(r)+\widetilde H(r) \dot \pi \left(A_2\right)^2-2  \dot \pi \left(A_2\right) \widetilde H(r)\dot \pi \left(A_2\right)\bigg)
\Bigg].
\end{align*}
Now the theorem follows easily.

\end{proof}

\begin{thm}\label{E}    For  $r>0$ we have
\begin{align*}
\widetilde E(\widetilde H)(r)& =
\frac{d\widetilde H}{dr}(1+r²)\dot\pi(A_3)
-\frac1r\left[\dot \pi (A_1) , \widetilde H(r) \right ]  \dot\pi\left(rA_1+A_2\right)
\\&\quad +\frac1r\left[\dot \pi (A_2) , \widetilde H(r) \right ]   \dot\pi (A_1-rA_2).
\end{align*}
\end{thm}
\begin{proof}

Since $\widetilde E( \widetilde H) (r)=E(H) (r,0,0)$, from Proposition \ref{E3} we have

\begin{align*}
&\widetilde E(\widetilde H)(r)=H_{y_1}\dot\pi \left(\begin{smallmatrix} 0 &0&0 \\0 &0 &-1-r^2 \\ 0&1+r^2 &0 \end{smallmatrix}\right) +H_{y_2}\dot\pi \left(\begin{smallmatrix} 0 &r &1 \\-r & 0 & 0 \\-1 & 0&0 \end{smallmatrix}\right)
+H_{y_3}\dot\pi \left(\begin{smallmatrix} 0 &-1 &r \\1 & 0 & 0 \\ -r & 0 &0 \end{smallmatrix}\right).
\end{align*}
Now, from Proposition \ref{derx1} we get

\begin{align*}
\widetilde E(H)(r)=&\frac{d\widetilde H}{dr}(1+r^2)\dot\pi \left(A_3\right)
-\frac1r\left[\dot \pi \left(A_1\right) , \widetilde H(r) \right ]\dot\pi \left(rA_1+A_2\right)
\\&+\frac1r\left[\dot \pi \left(A_2\right) , \widetilde H(r) \right ]\dot\pi \left(A_1-rA_2\right),
\end{align*}
\noindent which is the statement of the theorem.
\end{proof}

\medskip

Theorems \ref{D} and \ref{E} are given in terms of linear transformations. Now we will give the corresponding statements in terms of matrices by choosing an appropriate basis of $V_\pi$.
 We take the $\mathfrak{sl}(2)$-triple $\{e,f,h\}$ in $\mathfrak k_\CC \simeq \mathfrak{sl}(2,\CC)$ introduced in (\ref{efh}). \\
 If $\pi=\pi_\ell$ is the only irreducible representation of
$\mathrm{SO}(3)$ with highest weight $\ell/2$, we recalled in Subsection \ref{representations} that there exists a basis $\mathcal B=\vz{v_j}_{j=0}^\ell$ of $V_\pi$
  such that
\begin{equation}\label{basedeVpi}
\begin{split}
&\dot\pi(h)v_j=(\ell-2j)v_j,\\
&\dot\pi(e)v_j=(\ell-j+1)v_{j-1}, \quad (v_{-1}=0),\\
&\dot\pi(f)v_j=(j+1)v_{j+1}, \quad (v_{\ell+1}=0).
\end{split}
\end{equation}

\begin{prop}\label{Hdiagonal}
The function $\widetilde H$ associated to an irreducible spherical function $\Phi$ of type $\pi\in\hat K$ simultaneously diagonalizes
 in the basis $\mathcal B=\vz{v_j}_{j=0}^\ell$ of $V_\pi$.
\end{prop}
\begin{proof}
Let us consider the subgroup $M=\vzm{m_\theta}{\theta\in \RR}$ of $K$, where
\begin{equation}\label{Msubgrupo}
  m_\theta=\left(\begin{matrix} 1& 0&0  \\0&\cos\theta & \sin\theta\\0& -\sin\theta & \cos\theta  \\
\end{matrix}\right).
\end{equation}
Then, $M$ is isomorphic to $\SO(2)$ and fixes the points $(r,0,0)$ in $\RR^3$. Also, since the function $H$ satisfies $H(kg)=\pi(k)  H(g)$ $\pi(k^{-1})$ for all $k\in K$,
 we have that
\begin{align*}
\widetilde H(r)&= H(r,0,0)=H(m_\theta(r,0,0)^t)=\pi(m_\theta)H(r,0,0)\pi(m_\theta^{-1})\\
&=\pi(m_\theta)\widetilde H(r)\pi(m_\theta^{-1}).
 \end{align*}
Hence, $\widetilde H(r)$ and $\pi(m_\theta)$ commute for every $r$ in $\RR$ and every $m_\theta$ in $M$.

On the other hand, notice that $m_\theta=\exp (\theta\frac i2h)$ and then $\pi(m_\theta)=\exp (\dot\pi(\theta\frac i2h))$, but
from (\ref{basedeVpi}) we know that $\dot\pi(h)$ diagonalizes and that its eigenvalues have multiplicity
one. Therefore, the function $\widetilde H(r)$ simultaneously diagonalizes in the basis $\mathcal B=\vz{v_j}_{j=0}^\ell$ of $V_\pi$.
\end{proof}

Now we introduce the coordinate functions $\widetilde h_j(r)$ by means of
\begin{equation}\label{hj}
 \widetilde H(r)v_j= \widetilde h_j(r)v_j,
\end{equation}
and we identify $\widetilde H$ with  the column vector
\begin{equation}\label{Ht}
\widetilde H(r)=(\widetilde h_0(r),\dots,\widetilde h_\ell(r))^t.
\end{equation}

\begin{cor}\label{sistema1}
The functions $\widetilde H(r)$, $0<r<\infty$, satisfy
$(\widetilde D\widetilde H)(r)=\lambda \widetilde H(r)$ if and only if \begin{align*}
&\textstyle(1+r^2)^2\widetilde  h_j'' +2\textstyle\frac{(1+r^2)^2}{r}\widetilde h_j'+\textstyle\frac{1+r^2}{r^2} (j+1)(\ell-j)(\widetilde h_{j+1}-\widetilde h_j)
\\ & +\textstyle\frac{1+r^2}{r^2}j(\ell-j+1)(\widetilde h_{j-1}-\widetilde h_{j}) = \lambda \widetilde h_j ,
\end{align*}
for all $j=0,\dots, \ell$.
\end{cor}
\begin{proof}
Using the basis $\mathcal B=\vz{v_j}_{j=0}^\ell$ of $V_\pi$ (see (\ref{basedeVpi})) and writing the matrices $A_1$ and $A_2$ in terms
 of the $\mathfrak{sl}(2)$-triple $\{e,\,f,\,h\}$, see (\ref{efh}),
$$A_1=E_{21}-E_{12}=\, \frac i2({e+f})\; ,
\qquad A_2=E_{31}-E_{13}=\frac12({e-f}),$$
we have that Theorem \ref{D} says that $(\widetilde D\widetilde H)(r)=\lambda \widetilde H(r)$ if and only if

\begin{align*}
\lambda \widetilde H(r)&\,v_j=
(1+r^2)^2 \widetilde H''(r)\,v_j +2\frac{(1+r^2)^2}{r} \widetilde H'(r)\displaybreak[0]\,v_j \\
 & -\frac{(1+r^2)}{4r^2} \left( \dot\pi({e+f})^2 \,\widetilde H(r)\,+\widetilde H(r)\dot\pi({e+f})^2 -2\dot\pi({e+f})\widetilde H(r)\dot\pi({e+f})\right)\,v_j\\
 & +\frac{(1+r^2)}{4r^2} \left( \dot\pi({e-f}{})^2 \,\widetilde H(r)\,+\widetilde H(r)\dot\pi({e-f}{})^2 -2\dot\pi({e-f}{})\widetilde H(r)\dot\pi({e-f}{})\right)\,v_j,
\displaybreak[0]
\end{align*}
\noindent for $0\leq j\leq\ell$.

As $[e,f]=h$, we have that this is equivalent to
\begin{align*}
\lambda \widetilde H(r)v_j=\; &
(1+r^2)^2 \widetilde H''(r)\,v_j +2\frac{(1+r^2)^2}{r} \widetilde H'(r)\displaybreak[0]\,v_j &\\
 & -\frac{(1+r^2)}{2r^2} \bigg[ ( \dot\pi({h})+2\dot\pi({f})\dot\pi({e}) \, )\widetilde H(r)\,v_j\,+\widetilde H(r) (\dot\pi({h})+2\dot\pi({f})\dot\pi({e}) )\,v_j\\
&-2 (\dot\pi({e})\widetilde H(r)\dot\pi({f})+\dot\pi({f})\widetilde H(r)\dot\pi({e}) )\,v_j\bigg],
\end{align*}
\noindent for $0\leq j\leq\ell$.
Now, using (\ref{basedeVpi}), we obtain that $(\widetilde D\widetilde H)(r)=\lambda \widetilde H(r)$ if and only if
\begin{align*}
\lambda \widetilde h_j(r)\,v_j= &\,
(1+r^2)^2 \widetilde h_j''(r)\,v_j +2\frac{(1+r^2)^2}{r} \widetilde h_j'(r)\displaybreak[0]\,v_j \\
  -\frac{(1+r^2)}{2r^2} &\bigg[( (\ell-2j)+2j(\ell-j+1) \, )\widetilde h_j(r)\,v_j\,+\widetilde h_j(r) ((\ell-2j)+2j(\ell-j+1) )\,v_j      \\
           &  -2 ((\ell-j)\widetilde h_{j+1}(r)(j+1)+j\widetilde h_{j-1}(r)(\ell-j+1) )\,v_j\bigg],
\end{align*}
\noindent for $0\leq j\leq\ell$.

It can be easily checked that this is the required result. \end{proof}

\begin{cor}\label{sistema2}
The functions $\widetilde H(r)$, $0<r<\infty$, satisfy
$(\widetilde E\widetilde H)(r)=\mu \widetilde H(r)$ if and only if
\begin{align*}
&\textstyle -i(\ell-2j)\frac{1+r^2}{2}\widetilde h_{j}'+\frac{i}{2r}\biggr( (j+1)(\ell-j)(\widetilde h_{j+1}-\widetilde h_j)-j(\ell-j+1)(\widetilde h_{j-1}-\widetilde h_{j})\biggr) \\
&+\textstyle\frac{1}{2}\biggr( (j+1)(\ell-j)(\widetilde h_{j+1}-\widetilde h_j)+j(\ell-j+1)(\widetilde h_{j-1}-\widetilde h_{j})\biggr)
=\mu \widetilde h_j,
\end{align*}
for all $j=0,\dots, \ell$.
\end{cor}
\begin{proof}
We proceed in a way similar to the proof of Corollary \ref{sistema1}.
Using the $\mathfrak{sl}(2)$-triple $\{e,\,f,\,h\}$ and the matrices $A_1$, $A_2$ and $A_3$ (see \eqref{As}), from Theorem \ref{E} we have that $(\widetilde E\widetilde H)(r)=\mu \widetilde H(r)(r)$ if and only if

\begin{align*}
\mu \widetilde H(r)\,v_j=&(1+r²)H'(r)\dot\pi \left(A_3\right)\,v_j
-\frac1r\left[\dot \pi \left(A_1\right) , \widetilde H(r) \right ]\dot\pi \left(r A_1 +A_2\right)\,v_j
\\&+\frac1r\left[\dot \pi \left(A_2\right) , \widetilde H(r) \right ]\dot\pi \left(A_1 -rA_2\right)\,v_j,
\end{align*}
\noindent for every $v_j$ in $\mathcal B=\vz{v_j}_{j=0}^\ell$.

As in the proof of Theorem \ref{sistema1}, we write $A_1$, $A_2$ and $A_3$ in terms of $\{e,\,f,\,h\} $ (see (\ref{efh})),
$$A_1= \frac i2({e+f})\; ,\qquad A_2=\frac12{(e-f)}, \qquad A_3=-\frac i2h.$$
Hence, $(\widetilde D\widetilde H)(r)=\lambda \widetilde H(r)$ if and only if

\begin{align*}
\mu \widetilde H(r)v_j=&\frac 1{4r}\left[\dot \pi \left(e+f\right) , \widetilde H(r) \right ]\dot\pi \left(r (e+f) -i(e-f)\right)v_j \\
+&\frac1{4r}\left[\dot \pi \left(e-f\right) , \widetilde H(r) \right ]\dot\pi \left(i(e+f) -r(e-f)\right)v_j-i\frac{1+r²}{2}H'(r)\dot\pi \left(h\right)v_j
,
\end{align*}
\noindent for $0\leq j\leq\ell$. And that is equivalent to

\begin{align*}
\mu \widetilde H(r)\,v_j=&-i\frac{1+r²}{2}H'(r)\dot\pi \left(h\right)\,v_j
+\frac 1{2r}(r+i)\left[\dot \pi \left(e\right) , \widetilde H(r) \right ]  \dot\pi(f)\,v_j
\\&+\frac 1{2r}(r-i)\left[\dot \pi \left(f\right) , \widetilde H(r) \right ]  \dot\pi(e)\,v_j ,
\end{align*}
\noindent for $0\leq j\leq\ell$.

Finally, we use (\ref{basedeVpi}) to obtain

\begin{align*}
\mu \widetilde h_j\,v_j=&-i\frac{1+r²}{2}\widetilde h'_j(2\ell-j)\,v_j
+\frac 1{2r}(r+i)(\ell-j)(\widetilde h_{j+1}-\widetilde h_{j})  (j+1)\,v_j
\\&+\frac 1{2r}(r-i)j( \widetilde h_{j-1}-\widetilde h_{j}) (\ell-j+1)\,v_j ,
\end{align*}
\noindent for $0\leq j\leq\ell$.
Therefore, the corollary is proved.

\end{proof}

\medskip
In matrix notation, the differential operators $\widetilde D$ and $\widetilde E$ are given by
$$ \widetilde D \widetilde H=(1+r^2)^2 \widetilde H''+2\frac{(1+r^2)^2}{r} \widetilde H'+\frac{(1+r^2)}{r^2} (C_1+C_0)\widetilde H,$$
$$ \widetilde E \widetilde H= -i\frac{1+r^2}{2}A_0\widetilde H'+\frac{i}{2r}(C_1-C_0)\widetilde H+\frac{1}{2}(C_1+C_0)\widetilde H.$$
where the  matrices are given by
\begin{equation}\label{matricesAC}
\begin{split}
A_0&=\textstyle\sum_{j=0}^\ell(\ell-2j)E_{j,j},\\
C_0&=\textstyle\sum_{j=1}^\ell j(\ell-j+1)(E_{j,j-1}-E_{j,j}),\\
C_1&=\textstyle\sum_{j=0}^{\ell-1}(j+1)(\ell-j)(E_{j,j+1}-E_{j,j}).
\end{split}
\end{equation}

  When $\ell=0$, we are in the scalar case and the matrices $C_0$, $C_1$ and $A_0$ are zero. It is well known  that the zonal spherical functions on the sphere $S^3$  are given, in an appropriate variable $x$, in terms of Gegenbauer polynomials $C_n^\nu(x)$ with $\nu=1$ and $n=0,1,2,\dots$ (see \cite{AAR} page 302). Therefore, in some variable $x$, the functions $\widetilde H$ should satisfy a differential equation of the form
$$ (1-x^2)y'' - 3x y+n(n+2) y=0.$$
This suggests the following change of variable
\begin{equation}\label{variableu}
u=\tfrac {1} {\sqrt{1+r^2}}, \qquad u\in(0,1].
\end{equation}

\begin{remarks}
It is worth noticing that if $g=ka(\theta)k'$, with $k,k'\in K$, $a(\theta)\in A$ and $gK=(x_1,x_2,x_3,x_4)\in(S^3)^+$, then
$$u=\cos(\theta),$$
because
$$u(g)=\tfrac {1} {\sqrt{1+r^2}}=\tfrac {1} {\sqrt{1+y_1²+y_2²+y_3²}}=x_4=\cos(\theta).$$
\end{remarks}

We put
\begin{equation}\label{Hu}
H(u)=\widetilde H\left(\tfrac{\sqrt{1-u^2}}{u}\right) \text{ and } h_j(u)=\widetilde h_j\left(\tfrac{\sqrt{1-u^2}}{u}\right).
 \end{equation}

\noindent Under this change of variables, the differential operators
$\widetilde D$ and $\widetilde E$ are converted into two new differential operators $D$
and $E$.
We get the following expressions for them,
\begin{align}
\label{opDH}  DH&=(1-u^2) \frac{d^2H}{du^2} - 3u \frac{dH}{du}+\frac 1{1-u^2}(C_0+C_1)H,\\
\label{opEH}  EH& = \frac i2 \sqrt{1-u^2}A_0\frac{dH}{du}+\frac i2 \frac u{\sqrt{1-u^2}} (C_1-C_0)H+\frac 12 (C_0+C_1)H.
\end{align}

At this point there is a slight abuse of notation, since $D$ and $E$
were used earlier to denote operators on $\RR^3$.
{ \begin{remark}\label{DyEconmutan3}
 Clearly from Remark \ref{DyEconmutan2} we have that the differential operators $D$ and $E$ commute.
\end{remark}
}

 \section{Eigenfunctions of $D$}\label{DH}
We are interested in determining the functions
$H:(0,1)\longrightarrow \CC^{\ell+1}$ that are eigenfunctions of the differential operator
$$DH=(1-u^2) \frac{d^2H}{du^2} - 3u \frac{dH}{du}+\frac 1{1-u^2}(C_0+C_1)H,$$
$u\in (0,1)$.

It is well
known that such eigenfunctions are analytic functions on the  interval
$(0,1)$ and that the dimension of the corresponding eigenspace
is $2(\ell+1)$.

\medskip
The equation $DH=\lambda H$ is a coupled system of $\ell+1$ second
order differential equations in the components $(h_0, \dots , h_\ell)$ of
$H$, because the $(\ell+1)\times (\ell+1)$ matrix $C_0+C_1$ is not a diagonal
matrix. But fortunately the matrix $C_0+C_1$ is a symmetric one, thus
 diagonalizable. Now we quote from \cite{GPT2} the Proposition 5.1.

\begin{prop}\label{Hahn} The matrix $C_0+C_1$ is diagonalizable.
Moreover,
the eigenvalues are $-j(j+1)$ for $0\le j\le \ell$
and the corresponding eigenvectors are given by
$u_j=(U_{0,j},\dots,U_{\ell,j})$ where
\begin{equation*}
U_{k,j}= \lw{3}F_2\left( \begin{smallmatrix}
-j,\;-k,\;j+1 \\ 1,\;-\ell \end{smallmatrix}; 1 \right),
\end{equation*}
an instance of the Hahn orthogonal polynomials.
\end{prop}

Therefore, if we define $\check H(u)=U^{-1}H(u)$, we get that $DH=\lambda H$ is equivalent to
$$(1-u^2) \frac{d^2\check H}{du^2} - 3u \frac{d\check H}{du}-\frac 1{1-u^2}V_0\check H=\lambda \check H,$$
where $V_0=\sum_{j=0}^{\ell-1} j(j+1)E_{j,j}$.

In this way we obtain that $DH=\lambda H$ if and only if  the $j$-th component
$\check h_j(u)$ of $\check H(u)$, for  $0\leq j\leq \ell$, satisfies
\begin{equation}\label{eqtilde}
(1-u^2) \,\check h_j''(u)-3\,u \check
h_j'(u)-j(j+1)\frac{1}{(1-u^2)}\check h_j(u)-\lambda \check h_j(u)=0.
\end{equation}

If we write $\lambda=-n(n+2)$ with $n\in\CC$, and
$\check{h_j}(u)=(1-u^2)^{ j /2 }p_j(u)$. Then, for $0<j<\ell$, $p_j(u)$ satisfies
\begin{equation}\label{eqfj2}
(1-u^2) p_j''(u)-(2j+3)\,u p_j'(u)+ (n-j)(n+j+2) p_j(u) =0.
\end{equation}

Making a new change of variable, $s=(1-u)/2,\,s\in\left[0,\frac12 \right)$, and defining $\tilde p_j(s)=p_j(u)$ we have
\begin{equation}\label{eqfj3}
s(1-s) \tilde p_j''(s)+( j +\tfrac 32-(2j+3)\,s) \tilde p_j'(s)+ (n-j)(n+j+2) \tilde p_j
=0,
\end{equation}
 for $0<j<\ell$. This is a hypergeometric equation of parameters
$$a= -n+j\, , \qquad  b=n+j+2\, , \qquad c= j+\tfrac 32.$$

\noindent Hence, every solution $\tilde p_j(s)$ of \eqref{eqfj3} for $0<s<\tfrac12$ is
a linear combination of
$$  {}_2\!F_1\left(\begin{smallmatrix}-n+j,n+j+2\\  j+3/2 \end{smallmatrix}; s\right) \,  \qquad \text{ and } \qquad
s^{-j-1/2}\, {}_2\!F_1\left(\begin{smallmatrix} -n-1/2 ,n+ 3/2\\  -j+1/2 \end{smallmatrix}; s\right).  $$

  Therefore, for $0\leq j\leq \ell$,
any solution $\check h_j(u)$ of \eqref{eqtilde}, for $0<u<1$, is of the form
\begin{equation}\label{Hcheck}
\begin{split}
\check h_j(u) =& \,a_j (1-u^2)^{j/2} \, {}_2\!F_1\left(\begin{smallmatrix}-n+j,n+j+2\\  j+3/2 \end{smallmatrix}; \tfrac {1-u}2\right) \\& + b_j \, (1-u^2)^{-(j+1)/2}
 {}_2\!F_1\left(\begin{smallmatrix} -n-1/2 ,n+ 3/2\\  -j+1/2 \end{smallmatrix}; \tfrac{1-u}2\right),
 \end{split}
\end{equation}
for some $a_j,b_j\in \CC$.

\medskip

Therefore,  we have  proved the following theorem.

\begin{thm}\label{Dhyp0}
 Let  $H(u)$ be an eigenfunction of $D$ with eigenvalue $\lambda=-n(n+2)$, $n\in \CC$. Then, $H$ is of the form
  \begin{equation*}
    H(u)= UT(u)P(u) + US(u) Q(u)
  \end{equation*}
   where $U$ is the matrix defined in \eqref{Ucolumnas},
   $$  T(u)=\sum_{j=0}^\ell (1-u^2)^{j/2}E_{jj},\qquad  S(u)=\sum_{j=0}^\ell (1-u^2)^{-(j+1)/2}E_{jj}, $$
   $ P=(p_0, \dots, p_\ell)^t$  and $ Q=(q_0, \dots, q_\ell)^t$  are the vector valued functions given by
  \begin{align*}
  p_j(u)& = a_j \, {}_2\!F_1\left(\begin{smallmatrix}-n+j,n+j+2\\  j+3/2 \end{smallmatrix}; \tfrac{1-u}2\right),\\
  q_j(u)& = b_j \, {}_2\!F_1\left(\begin{smallmatrix}-n-1/2,n+3/2\\  -j+1/2 \end{smallmatrix};\tfrac{1-u}2\right),
\end{align*}
where  $a_j$ and $b_j$ are arbitrary complex numbers for $j=0,1,\ldots,\ell$.
\end{thm}

Going back to our problem of determining all irreducible spherical functions $\Phi$, we recall that  $\Phi(e)=I$; then, the associated function $H\in C^\infty(\RR^3)\otimes\End(V_\pi)$ satisfies $H(0,0,0)=I$.
In the variable $r\in \RR$, we have that $\lim_{r\to 0^+}\widetilde H(r)=I$.
Therefore, we are interested in those eigenfunctions of $D$ such that
$$\lim_{u\to 1^-} H(u)= (1,1,\dots, 1) \in \CC^{\ell+1}. $$

From Theorem \ref{Dhyp0} we observe that
$$\lim_{u\to 1^-}  P(u)=(a_0, a_1,\dots , a_\ell) \qquad \text{ and } \qquad \lim_{u\to 1^-}  Q(u)=(b_0, b_1,\dots , b_\ell).$$
Moreover, the matrix $T(u)$ has limit when $u\to 1^-$, while $S(u)$ does not. Therefore an eigenfunction $H$ of $D$ has limit when $u\to 1^-$ if and only if
the limit of $Q(u)$ when $u\to 1^-$ is $(0,\dots, 0)$. In such a case we have that
\begin{equation}\label{limiteH}
  \lim_{u\to 1^-} H(u)= \lim_{u\to 1^-} UT(u)P(u) =U \,(a_0,0, \dots, 0)^t= a_0 (1,\dots, 1)^t.
\end{equation}

In this way we have proved the following result.

\begin{cor}\label{Dhyp}
Let $H(u)$ be an eigenfunction of $D$ with eigenvalue $\lambda=-n(n+2)$, $n\in \CC$, such that $\lim_{u \to 1^- }{H(u)}$ exists.
 Then, $H$ is of the form
  \begin{equation*}
    H(u)= UT(u)P(u)
  \end{equation*}
   with $U$ the matrix defined in \eqref{Ucolumnas}, $  T(u)=\displaystyle \sum_{j=0}^\ell (1-u^2)^{j/2}E_{jj}, $ and
$ P=(p_0, \dots, p_\ell)^t$  is the vector valued function given by
  \begin{align*}
  p_j(u)& = a_j \, {}_2\!F_1\left(\begin{smallmatrix}-n+j,n+j+2\\  j+3/2 \end{smallmatrix}; \tfrac{1-u}{2}\right), \qquad  0\le j \le \ell,
\end{align*}
where  $a_j$ are arbitrary complex numbers for $j=1,2,\ldots,\ell$. We also have that $\lim_{u \to 1^- }{H(u)}=a_0(1,1,\ldots,1)^t$. Particularly, if $H(u)$ is associated to an irreducible spherical function, then $a_0=1$.
 \end{cor}

\section{Eigenfunctions of $D$ and $E$}\label{DE}
In this section  we shall study the simultaneous solutions of $DH(u)=\lambda H(u)$ and $EH(u)=\mu H(u)$, $0<u<1$.

We introduce a matrix function $P(u)$, defined from $H(u)$  by
\begin{equation}\label{HP}
  H(u)=U\, T(u)\,P(u),
\end{equation}
where  $U$ is the matrix defined in \eqref{Ucolumnas}
and $ T(u)=\sum_{j=0}^\ell (1-u^2)^{j/2}E_{jj} $.

The fact that $H$ is an eigenfunction of the differential operators $D$ and $E$
makes $P$ an eigenfunction of the differential operators
\begin{equation}\label{barras}
\bar D =\left(UT(u)\right)^{-1} D \left(UT(u)\right) \quad \text{ and } \quad \bar E =\left(UT(u)\right)^{-1} E \left(UT(u)\right),
\end{equation}
 with, respectively, the same eigenvalues $\lambda$ and $\mu$.

\smallskip
 The explicit expressions of $\bar D$ and $\bar E$ shall be given in Theorem \ref{puntos}, but first we recall some properties of the Hahn polynomials.

For real numbers $\alpha, \beta >-1$,  and for a positive integer $N$ the Hahn
polynomials $Q_n(x)=Q_n(x;\alpha, \beta, N)$ are defined by \color{black}
 $$Q_n(x)=\lw{3}F_2\left( \begin{smallmatrix}
-n,\;-x,\;n+\alpha+\beta+1 \\ \alpha+1,\;-N \end{smallmatrix}; 1 \right),
\qquad \text{ for } n=0,1,\dots , N.$$
Taking $\alpha=\beta=0$, $N=\ell$, $x=j$,  $n=k$,  we obtain
$$ U_{jk}= Q_k(j)= \lw{3}F_2\left( \begin{smallmatrix}
-k,\;-j,\;k+1 \\ 1,\;-\ell \end{smallmatrix}; 1 \right).$$
These Hahn polynomials are examples of orthogonal polynomials, see \cite{KS} equation (1.5.2):
\begin{equation}  \label{Hahnortogon}
    \sum_{r=0}^\ell Q_j(r)Q_k(r) = \delta_{jk} \frac{(-1)^j (j+1)_{\ell+1} j!}{(2j+1) (-\ell)_j}\ell!=\delta_{jk} \frac{(\ell+j+1)! (\ell-j)!}{(2j+1)\,\ell! \, \ell!}.
\end{equation}

Also these Hahn polynomials satisfy a three-term recursion relation in the variable $j$, see  \cite{AAR} equation (d) on page 346,
\color{black}
 \begin{equation}\label{Hahnrec_j}
\begin{split}
   \big ( j(\ell-j+1)+ &(j+1)(\ell-j)-k(k+1)\big)  U_{jk}\\ &= j(\ell-j+1) U_{j-1,k}+ (j+1)(\ell-j)U_{j+1,k}.
   \end{split}
 \end{equation}
Also, they satisfy a three-term recursion relation in the variable $k$, see  \cite{AAR} equation (c) on page 346,
\begin{equation}\label{Hahnrec_k}
(\ell-2j) U_{jk}
=  \tfrac{k(\ell+k+1)}{2k+1} U_{j,k-1} +\tfrac{(k+1)(\ell-k)}{2k+1}U_{j,k+1}.
\end{equation}

Karlin and McGregor in \cite{KMcG} also proved that the Hahn polynomials satisfy a first-order recurrence relation that combines the variables $j$ and $k$ (see also \cite{RS}, equation (36)):
\begin{equation}\label{Hahnrec_j_k}
\begin{split}
  \big( k(\ell-j)- & k(k+j+1)+ 2(j+1)(\ell-j) \big) U_{jk}\\&= 2(j+1)(\ell-j) U_{j+1,k} - k(k+\ell+1) U_{j,k-1}.
\end{split}
\end{equation}

\smallskip
We will need the following technical lemma.

\begin{lem}\label{Hahnproperties}
\noindent Let $U=\left(U_{jk}\right) $ be the matrix defined by
 \begin{equation}\label{Ucolumnas}
U_{jk}= \lw{3}F_2\left( \begin{smallmatrix} -k,\;-j,\;k+1 \\
1,\;-\ell \end{smallmatrix}; 1 \right),
\end{equation}
and let $A_0$, $C_0$ and $C_1$ be the matrices introduced in \eqref{matricesAC}.
Then,
   \begin{align}
\nonumber    U^{-1} A_0 U & =     Q_0+Q_1,\\
 \label{UU-1}  U^{-1} (C_1+C_0) U & = -V_0,\\
  \nonumber  U^{-1} (C_1-C_0) U & = Q_1 J-Q_0(J+1),
   \end{align}
where
\begin{align*}
V_0& = \sum_{j=0}^{\ell-1} j(j+1)E_{j,j} \,,
\qquad \qquad \quad
J=\sum_{j=0}^\ell j E_{jj} \,,\\
Q_0 & = \sum_{j=0}^{\ell-1}  \tfrac{(j+1)(\ell+j+2)}{2j+3} E_{j,j+1}\,,\quad Q_1= \sum_{j=1}^{\ell} \tfrac{j(\ell-j+1)}{2j-1}E_{j,j-1} \,.
\end{align*}
\end{lem}

\begin{proof}
 To prove that $U^{-1}A_0 U=Q_0+Q_1$ is equivalent to verifying that
\begin{equation*}
A_0 U=U(Q_0+Q_1).
\end{equation*}
By taking a look at the $jk$-entry for $j,k=0, \dots, \ell$, we obtain that
\begin{equation*}
(\ell-2j) U_{jk}
= U_{j,k-1} \tfrac{k(\ell+k+1)}{2k+1} +U_{j,k+1}
\tfrac{(k+1)(\ell-k)}{2k+1}.
\end{equation*}
This is the three-term recursion relation in the variable $k$ given in \eqref{Hahnrec_k}.

Observe that  $U^{-1} (C_1+C_0) U  = -V_0$ is a direct consequence of Proposition \ref{Hahn}. Also, it follows directly by considering every $jk$-entry of $(C_1+C_0)U=-UV_0$ and by using the recursion relation \eqref{Hahnrec_j}.

Now we have to prove that
\begin{equation*}\label{black}
U^{-1}(C_1-C_0)U=-Q_0 (J+1)+Q_1J.
\end{equation*}
By using  $ (C_0+C_1)U=-UV_0$, it is equivalent to prove that
\begin{equation*}
-2C_0U=U(-Q_0 (J+1)+Q_1J+V_0);
\end{equation*}
therefore, if we look at the $jk$-entry, what we need to verify is
\begin{align*}
-2{(C_0)}_{j,j}U_{jk}-2{(C_0)}_{j,j-1}U_{j-1,k}&\\=-U_{j,k-1}
{(Q_0)}_{k-1,k} &(J+1)_{k,k}+U_{j,k+1} {(Q_1)}_{k+1,k} J_{k,k}+U_{jk}
{(V_0)}_{k,k},
\end{align*}
or, equivalently, we have to prove that
\begin{equation}\label{T}
\begin{split}
2j& (\ell-j+1)U_{jk}-2j(\ell-j+1)U_{j-1,k}\\
&=-\tfrac{k(\ell+k+1)(k+1)}{2k+1} U_{j,k-1}  +
\tfrac{k(k+1)(\ell-k)}{2k+1} U_{j,k+1}+ k(k+1) U_{jk}.
\end{split}
\end{equation}

By using the recurrence relation \eqref{Hahnrec_k}, we can write $ U_{j,k+1}$ in terms of $U_{jk}$ and $U_{j,k-1}$. Therefore, the identity (\ref{T}) becomes
\begin{align*}
\big(2j(\ell-j+1)-k(\ell-2j)& -k(k+1)\big) U_{jk}\\ & =2j (\ell-j+1)U_{j-1,k}-k(\ell+k+1)U_{j,k-1},
\end{align*}
Finally, we use \eqref{Hahnrec_j} to write $ U_{j-1,k}$ in terms of $U_{j+1,k}$ and $U_{jk}$ and obtain
\begin{align*}\big( k(\ell-j)- & k(k+j+1)+ 2(j+1)(\ell-j) \big) U_{jk}\\&= 2(j+1)(\ell-j) U_{j+1,k} - k(k+\ell+1) U_{j,k-1},
\end{align*}
which is exactly the identity in \eqref{Hahnrec_j_k}, and this concludes the proof of the Lemma \ref{Hahnproperties}.
\end{proof}

 \begin{thm}\label{puntos}
 The operators $\bar D$ and $\bar E$ defined in \eqref{barras} are given by
 \begin{align*}
  \bar D P &= (1-u^2) P'' -u CP'-V P, \mbox{} \\
  \bar EP & = \tfrac i 2\left( (1-u^2)Q_0+Q_1 \right) P'-\tfrac i2 u MP-\tfrac 12 V_0 P,
\end{align*}
\noindent where
\begin{align*}
  \begin{alignedat}{2}
&C = \sum_{j=0}^\ell (2j+3)E_{jj},
& V &=  \sum_{j=0}^\ell j(j+2)E_{jj}, \\
& Q_0 = \sum_{j=0}^{\ell-1}  \tfrac{(j+1)(\ell+j+2)}{2j+3} E_{j,j+1},
& \qquad Q_1&= \sum_{j=1}^{\ell} \tfrac{j(\ell-j+1)}{2j-1}E_{j,j-1},\\
&M = \sum_{j=0}^{\ell-1} (j+1)(\ell+j+2) E_{j,j+1},
& V_0 &= \sum_{j=0}^{\ell-1} j(j+1)E_{jj}.
  \end{alignedat}
  \end{align*}
\end{thm}

\begin{proof}
Let $H=H(u)=U T(u) P(u)$. We start by  computing $D(H)$ for the differential operator $D$ introduced in \eqref{opDH}.
\begin{align*}
D H & = (1-u^2)UTP''+\big (2(1-u^2)UT'-3uUT\big)P'\\  &\quad +\big((1-u^2)UT''-3uUT'+\tfrac {1}{1-u^2}(C_0+C_1)UT\big)P\\
&=UT\Big( (1-u^2)P''+\big(2(1-u^2)T^{-1}T'-3u\big)P'\\
 &\quad \qquad +\Big( (1-u^2)T^{-1}T''-3uT^{-1}T' +\frac{1}{1-u^2}T^{-1} U^{-1}(C_0+C_1)UT \Big)P \Big).
\end{align*}
Since $T$ is a diagonal matrix, we easily compute
\begin{align*}
T^{-1}(u)T'(u) & = -\tfrac u {(1-u^2)} \sum_{j=0}^\ell j \, E_{jj},  \quad
  T^{-1}T''(u) = \tfrac 1{(1-u^2)^2} \sum_{j=0}^\ell j ((j-1)u^2-1)\, E_{jj}.
\end{align*}
 Also, from \eqref{UU-1} we have that
$U^{-1}(C_0+C_1)U = -V_0$. Since $V_0$ is a diagonal matrix, it commutes with $T$ and we get
\begin{align*}
(1-u^2)T^{-1} & T''-3uT^{-1}T'  +\frac{1}{1-u^2}T^{-1}U^{-1}(C_0+C_1)UT  \\
&= \frac 1{(1-u^2)} \sum_{j=0}^\ell \big ( j(j-1)u^2-j +3ju^2-j(j+1)\big ) E_{jj} =-V.
 \end{align*}

\smallskip
Now, for the differential operator $E$ introduced in \eqref{opEH}, we compute $E(H)$ with $H(u)=U T(u) P(u)$.
\begin{align*}
 EH &= \frac i2
\sqrt{1-u^2}A_0 UTP'\\
&\quad + \bigg(\frac i2 \sqrt{1-u^2}A_0 UT'+\frac i2 \frac u{\sqrt{1-u^2}}
(C_1-C_0)UT+\frac 12 (C_0+C_1)UT\bigg) P \displaybreak[0]\\
&=UT\bigg
( \frac i2 \sqrt{1-u^2}T^{-1}U^{-1}A_0 UTP'+ \bigg(\frac i2 \sqrt{1-u^2}T^{-1}U^{-1}A_0 UT'\\
&\quad +\frac i2 \frac u{\sqrt{1-u^2}} T^{-1}U^{-1}(C_1-C_0)UT+\frac 12
T^{-1}U^{-1}(C_0+C_1)UT\bigg)P\bigg).
\end{align*}

From Lemma \ref{Hahnproperties} above we have that
$U^{-1}A_0 U=Q_0+Q_1$.
 By using $T=\sum_{j=0}^\ell (1-u^2)^{j/2} E_{jj}$,
 we get
\begin{align*}
\sqrt{1-u^2} \,T^{-1}U^{-1}A_0 UT&=\sqrt{1-u^2}\, T^{-1}(Q_0+Q_1)T=(1-u^2)Q_0+Q_1.
\end{align*}

\noindent From  \eqref{UU-1} and the fact that $T$ is diagonal,  we have that
$T^{-1}U^{-1}(C_0+C_1)UT = -V_0$.
Then, it only remains to prove that
\begin{equation}\label{ecuac}
\sqrt{1-u^2}T^{-1}U^{-1}A_0 UT'\\
+ \frac u{\sqrt{1-u^2}} T^{-1}U^{-1}(C_1-C_0)UT=-uM.
\end{equation}

\noindent Since $T'(u)=\tfrac{-u}{1-u^2}J T(u) $, where  $J=\sum_{j=0}^\ell j E_{jj}$, we have to prove that
\begin{equation}\label{aux}
T^{-1}\big (U^{-1}A_0 U J   - U^{-1}(C_1-C_0)U \big )T =\sqrt {1-u^2} M.
\end{equation}

\noindent From Lemma \ref{Hahnproperties} we have
\begin{align*}
U^{-1}A_0 U J  - U^{-1}(C_1-C_0)U & = (Q_1+Q_0)J - Q_1J+ Q_0(J+1) = Q_0(2J+1) \\
& = \sum_{j=0}^{\ell-1} (j+1)(\ell+j+2) E_{j,j+1}= M.
\end{align*}

\noindent Since $T= \sum_{j=0}^\ell (1-u^2)^{j/2} E_{jj} $, \eqref{aux} is satisfied and this completes the proof of the theorem.
\end{proof}

The function $P$ is an eigenfunction of the differential operator $\bar D$ if and only if the function $ H=U T(u) P(u)$
is an eigenfunction of the differential operator $D$.
From Theorem \ref{Dhyp0} we have the explicit expression of the function $P(u)=(p_0(u), \dots, p_\ell(u))^t$,
\begin{align*}
 p_j(u) =a_j\, {}_2\!F_1\left(\begin{smallmatrix}-n+j,n+j+2\\  j+3/2 \end{smallmatrix}; \tfrac {1-u}2\right)  + b_j \, (1-u^2)^{-(j+1/2)}
 {}_2\!F_1\left(\begin{smallmatrix} -n-1/2 ,n+ 3/2\\  -j+1/2 \end{smallmatrix}; \tfrac{1-u}2\right),
 \end{align*}
where $a_j$ and $b_j$ are in $\CC$, for $0\le j\leq \ell$.

Since we are interested in determining the irreducible spherical functions of the pair $(G,K)$,
we need to study the simultaneous eigenfunctions of $D$ and $E$ such that there exists a finite limit of the function $H$ when $u\to1^{-}$.

From Theorem \ref{Dhyp0} we have that
$\lim_{u\to 1^-} H(u)$ is finite if and only if
  $$\lim_{u\to 1^-} b_j \, (1-u^2)^{-(j+1)/2}
 {}_2\!F_1\left(\begin{smallmatrix} -n-1/2 ,n+ 3/2\\  -j+1/2 \end{smallmatrix}; \tfrac{1-u}2\right)$$
exists and is finite for all $0\leq j \leq \ell$. This is true if and only if $b_j=0$ for all $0\leq j\leq \ell$. Therefore, $\lim_{u\to 1^-} H(u)$ is finite if and only if
$\lim_{u\to 1^-} P(u)$ is finite.

From Corollary \ref{Dhyp} we know that an eigenfunction $P=P(u)$ of $\bar D$ in the interval $(0,1)$ has a finite limit as $u\to1^{-}$ if and only if $P$ is analytic at $u=1$.
Let us now consider the following vector space of functions into $\CC^{\ell+1}$,
$$W_\lambda=\vzm{P=P(u) \text{ analytic in }(0,1]}{\bar D P=\lambda P }.$$

A function $P\in W_\lambda$ is characterized by $P(1)= (a_0, \cdots , a_\ell)$. Thus, the dimension of $W_\lambda$ is $\ell+1$ and the isomorphism  $W_\lambda\simeq\CC^{\ell+1}$  is given by
$$\nu:W_\lambda\longrightarrow\CC^{\ell+1},\qquad P\mapsto P(1). $$

{  The differential operators $ \bar D$ and $\bar E$ commute
 because the differential operators $D$ and $E$ commute (see Remark \ref{DyEconmutan3}).}

\begin{prop}\label{Llambda} The linear space $W_\lambda$
is stable under the differential operator  $\bar E$ and it  restricts to a linear map on $W_\lambda$. Moreover,  the following is a commutative diagram
\begin{equation}\label{diagrama1}
\begin{CD}
W_\lambda @ >\bar E >>W_\lambda \\ @ V \nu VV @ VV \nu V \\
\CC^{\ell+1} @> L(\lambda) >> \CC^{\ell+1}
\end{CD}
\end{equation}
where $L(\lambda)$ is the $(\ell+1)\times (\ell+1)$ matrix
\begin{equation*}
  \begin{split}
    L(\lambda) & = -\tfrac i2 Q_1C^{-1}(V+\lambda)-\tfrac i2 M-\tfrac 12 V_0 \\
    &=- i \sum_{j=1}^\ell \tfrac{j(\ell-j+1)\big((j-1)(j+1)+\lambda \big)}{2(2j-1)(2j+1)} E_{j,j-1} -i \sum_{j=0}^{\ell-1}  \tfrac{(j+1)(\ell+j+2)}2 E_{j,j+1}\\
    & \quad - \sum_{j=0}^\ell   \tfrac{j(j+1) }2E_{jj}.
  \end{split}
\end{equation*}
\end{prop}
\begin{proof}
  The differential operator $\bar E$ takes analytic functions into analytic functions, because its coefficients are polynomials, see Theorem \ref{puntos}.  A function $P\in W_\lambda$ is  analytic, then $\lim_{u\to 1^-} \bar EP(u)$ is finite.
 On the other hand, since $\bar D$ and $\bar E$ commute, the differential operator $\bar E$ preserves the eigenspaces of $\bar D$. This proves that  $W_\lambda$ is stable under $\bar E$. In particular, $\bar E$ restricts to a linear map $L(\lambda)$ on $W_\lambda$, to be determined now.

  From
  Theorem \ref{puntos} we have
  $$\nu(\bar E(P))=(\bar E P )(1)=\tfrac i2 Q_1 P'(1)-\tfrac i2 MP(1)-\tfrac 12 V_0P(1).$$
   But we can obtain  $P'(1)$ in terms of $P(1)$. In fact, if we evaluate $\bar DP=\lambda P$ at $u=1$ we get $$P'(1)= -C^{-1} (V+\lambda) P(1).$$
   Notice that  $C$ is an invertible matrix.
Hence,
  \begin{align*}
  \nu(\bar E(P))&=-\tfrac i2 Q_1 C^{-1} (V+\lambda) P(1)-\tfrac i2 MP(1)-\tfrac 12 V_0P(1)\\
                &=L(\lambda) P(1)= L(\lambda) \,\nu(P).
  \end{align*}
This  completes the proof of the proposition.
\end{proof}

 \remark
If $\lambda=-n(n+2)$, with $n\in\CC$,  then we have
\begin{equation}\label{matrixL}
  \begin{split}
    L(\lambda)& = i \sum_{j=1}^\ell \tfrac{j(\ell-j+1)(n-j+1)(n+j+1)}{2(2j-1)(2j+1)} E_{j,j-1} -i \sum_{j=0}^{\ell-1}  \tfrac{(j+1)(\ell+j+2)}2 E_{j,j+1}
\\ & \quad
   -  \sum_{j=0}^\ell  \tfrac {j(j+1)}2 E_{jj}.
  \end{split}
\end{equation}

\begin{cor}\label{corttr}
  All eigenvalues $\mu$ of $L(\lambda)$  have geometric multiplicity one, that is, all eigenspaces are one dimensional.
\end{cor}
\begin{proof}
A vector $a=\left(a_0,a_1,\dots,a_\ell\right)^t$ is an eigenvector of
$L(\lambda)$ of eigenvalue $\mu$, if and only if
$\left\{a_j\right\}_{j=0}^\ell$ satisfies the following three-term
 recursion relation
\begin{equation}\label{threetermI}
i \, \tfrac{j(\ell-j+1)(n-j+1)(n+j+1)}{2(2j-1)(2j+1)}  \,\, a_{j-1} - \tfrac {j(j+1)}2 \,\,a_j
 - i   \tfrac{(j+1)(\ell+j+2)}2 \,\,a_{j+1}  = \mu \, a_j,
\end{equation}
for $j=0,\dots,\ell-1$ (where we interpret $a_{-1}=0$), and
\begin{equation}\label{threetermclosing}
  i \,\tfrac{\ell(n-\ell+1)(n+\ell+1)}{2(2\ell-1)(2\ell+1)} \,\, a_{\ell-1} - \tfrac {\ell(\ell+1)}2 \,\,a_\ell=\mu a_\ell.
\end{equation}

From these equations we see that the vector $a$ is determined by  $a_0$,
which proves that the geometric multiplicity of the eigenvalue $\mu$ of $L(\lambda)$ is one.
\end{proof}

\begin{remark}
The values of $\mu$ for which the equations \eqref{threetermI} and \eqref{threetermclosing} have a solution $\{a_j\}_{j=0}^\ell$ are exactly the eigenvalues of the matrix $L(\lambda)$.

The equations \eqref{threetermI}, for $j=0, \dots , \ell-1,$ are used to define $a_1, \dots , a_\ell$ starting with any $a_0\in \CC$. The equation \eqref{threetermclosing} is an extra condition (a ``closing equation") that the coefficients $a_j$ should satisfy  in order for $a=(a_0, \dots , a_\ell)$ to be an eigenvector of $L(\lambda)$ of eigenvalue $\mu$.
\end{remark}

Finally, we get the main result of this section which is the characterization of the simultaneous eigenfunctions $H$ of the differential operators $D$ and $E$ in $(0,1)$, which are continuous in $(0,1]$. Recall that the irreducible spherical functions of the pair $(G,K)$ give raise to such functions $H$.

\begin{cor}\label{DEhyp}
Let $H(u)$ be a simultaneous eigenfunction of $D$ and $E$ in $(0,1)$, continuous in $(0,1]$, with respective eigenvalues $\lambda=-n(n+2)$, $n\in\CC$, and $\mu$.  Thus, $H$ is of the form
  \begin{equation*}
    H(u)= UT(u)P(u)
  \end{equation*}
   with $U$ the matrix given in \eqref{Ucolumnas}, $  T(u)=\displaystyle \sum_{j=0}^\ell (1-u^2)^{j/2}E_{jj}, $ and
   $ P=(p_0, \dots, p_\ell)^t$  is the vector valued function given by
  \begin{align*}
  p_j(u)& = a_j \, {}_2\!F_1\left(\begin{smallmatrix}-n+j,n+j+2\\  j+3/2 \end{smallmatrix}; \tfrac{1-u}2\right)
\end{align*}
 where $\{a_j\}_{j=0}^\ell$ satisfies the recursion relations \eqref{threetermI} and  \eqref{threetermclosing}.
We also have that $H(1)=a_0(1,1,\ldots,1)^t$. In particular,  if $H(u)$ is associated to an irreducible spherical function we have that $a_0=1$.
\end{cor}

\begin{remark}
The condition $H(1)=  (1,\dots , 1)^t$  implies that $P(1)$ is a vector whose first entry is equal to 1.
\end{remark}

\medskip

In $S^3$, the set
$$\vzm{x_\theta=(\sqrt{1-\theta^2},0,0,\theta)}{\theta\in[-1,1]}$$ parameterizes all the $K$-orbits. Notice that for $\theta>0$ we have that $x_\theta\in (S³)^+$, and $p(x_\theta)=(\frac{\sqrt{1-\theta^2}}{\theta},0,0)$. Therefore, in terms of the variable $r\in[0,\infty)$ we have that $r=\frac{\sqrt{1-\theta^2}}{\theta}$, and then, in terms of the variable $u\in(0,1]$ we get $u=\frac1{\sqrt{1+r²}}= \theta$.
Hence, given an irreducible spherical function $\Phi$ of type $\pi\in\hat K$, if we consider the associated function $H:S^3\longrightarrow\End(V_\pi)$ defined by
$$H(g\,(0,0,0,1)^t)=\Phi(g)\Phi^{-1}_\pi(g), \qquad g\in G,$$
we have that
\begin{equation}\label {u}
 H(\sqrt{1-u^2},0,0,u)=\text{diag}\{H(u)\}=\text{diag}\{UT(u)P(u)\},
\end{equation}
where $H(u)$, $u\in(0,1]$, is the vector valued function given in Corollary \ref{DEhyp} and $\text{diag}\{H(u)\}$ means the diagonal matrix valued function whose $kk$-entry is equal to the $k$-th  entry of the vector valued function $H(u)$.

\section{Eigenvalues of the spherical functions}
\label{sec:RepresentationTheory}

The aim of this section is to use the representation theory of $G$ to compute the eigenvalues of an irreducible spherical function $\Phi$ corresponding to the
differential operators $\Delta_1$ and $\Delta_2$. From these eigenvalues we shall obtain the eigenvalues of the function $H$ as eigenfunctions of $D$ and $E$.

 As we described in Section \ref{sec:prelim},  there
exists a one to one correspondence between irreducible
spherical functions of $(G,K)$ of type $\delta \in \hat K $ and finite dimensional irreducible
representations of $G$ that contain the $K$-type $\delta$.
 In fact,   every irreducible spherical function $\Phi$ of type $\delta \in \hat K$ is of the form
\begin{equation}\label{sphasprojec}
  \Phi(g)v= P(\delta)\tau(g)v, \qquad \quad  g\in G, \qquad v\in P(\delta)V_\tau,
\end{equation}
 where $(\tau , V_\tau)$ is a finite
dimensional irreducible representation of $G$, which contains the $K$-type $\delta$, and $P(\delta)$ is the
projection of $V_\tau$ onto the $K$-isotypic component of type $\delta$.

The irreducible finite dimensional representations $\tau$ of $G=\SO(4)$ are parameterized by a pair of integers
$(m_1, m_2)$ such that $$m_1\geq |m_2|,$$ while the irreducible finite dimensional representations $\pi_\ell$ of $K=\SO(3)$ are parameterized by $\ell\in 2\NN_0$.

The  representations $\tau_{(m_1,m_2)}$ restricted to $\SO(3)$ contain the representation $\pi_\ell$  if and only if
{ $m_1\geq\ell/2\geq |m_2|.$}
Therefore, the equivalence classes of irreducible spherical functions of $(G,K)$ of
type $\pi_\ell$ are parameterized by  the
set of all pairs $(m_1,m_2)\in \ZZ^2$
{ such that  $$m_1\geq\tfrac\ell2\geq |m_2|.$$
We denote by
\begin{equation*}
\Phi_\ell^{(m_1,m_2)}, \qquad  \text{ with }  \quad  m_1\geq\tfrac\ell2\geq |m_2|,
\end{equation*}  the spherical function of type $\pi_\ell$
associated to the representation $\tau_{(m_1,m_2)}$ of $G$.}

\begin{thm}\label{param2}
The spherical functon $\Phi_\ell^{(m_1,m_2)}$ satisfies
\begin{align*}
\Delta_1\Phi_\ell^{(m_1,m_2)}&= \tfrac14(m_1-m_2)(m_1-m_2+2) \Phi_\ell^{(m_1,m_2)},\\
\Delta_2\Phi_\ell^{(m_1,m_2)}&= \tfrac14(m_1+m_2)(m_1+m_2+2) \Phi_\ell^{(m_1,m_2)}.
\end{align*}
\end{thm}
\begin{proof}
We start by observing that the eigenvalue of any irreducible spherical function $\Phi$ corresponding to a differential operator $\Delta \in D(G)^G$, given by
$[\Delta\Phi](e)$, is a scalar multiple of the identity.
Since $\Delta_1$ and $\Delta_2$ are in $D(G)^G$, we have that
$$[\Delta_1\Phi_\ell^{(m_1,m_2)}](e)= \dot \tau_{(m_1,m_2)}(\Delta_1) \quad \text{ and } \quad
  [\Delta_2\Phi_\ell^{(m_1,m_2)}](e)= \dot \tau_{(m_1,m_2)}(\Delta_2).$$

These scalars  can be computed by looking at the action of $\Delta_1$ and $\Delta_2$ on a highest weight vector $v$ of the representation $\tau_{(m_1,m_2)}$,
 whose highest weight is of the form $m_1\varepsilon_1+ m_2\varepsilon_2$.

Recall that \begin{align*}
\Delta_1= (iZ_6)^2+iZ_6-(Z_5+iZ_4)(Z_5-iZ_4),\\
\Delta_2= (iZ_3)^2+iZ_3-(Z_2+iZ_1)(Z_2-iZ_1).
\end{align*}
Since $(Z_5-iZ_4)$ and $(Z_2-iZ_1)$ are positive root vectors and $Z_6, Z_3\in \lieh_\CC$, we get
\begin{align*}
\dot\tau_{(m_1,m_2)}(\Delta_1)&v = \dot\tau_{(m_1,m_2)}(iZ_6)^2v+\,\dot\tau_{(m_1,m_2)} (iZ_6)v
 = \tfrac14(m_1-m_2)(m_1-m_2+2)v ,\\
\dot\tau_{(m_1,m_2)}(\Delta_2)&v =\dot\tau_{(m_1,m_2)}(iZ_3)^2v+\,\dot\tau_{(m_1,m_2)} (iZ_3)v
= \tfrac14(m_1+m_2)(m_1+m_2+2)v .
\end{align*}
This completes the proof of the theorem.
\end{proof}

Now we give the eigenvalues of the function $H$ associated to an irreducible spherical function,
 corresponding to the differential operators $D$ and $E$.

\begin{cor}\label{autovs} The function $H$ associated to the spherical function
 $\Phi_\ell^{(m_1,m_2)}$  satisfies $D H=\lambda H$ and $E H=\mu H$ with
\begin{equation*}
  \lambda= -(m_1-m_2)(m_1-m_2+2)  , \qquad
  \mu=-\tfrac {\ell(\ell+2)} 4 + (m_1+1)m_2.
\end{equation*}
\end{cor}
\begin{proof}
Let $\Phi=\Phi_\ell^{(m_1,m_2)}$.
From Proposition \ref{relacionautovalores} we have that $\Delta_1 \Phi=\widetilde \lambda \Phi$ and $\Delta_2 \Phi=\widetilde \mu \Phi$
if and only if $DH=\lambda H$ and $EH=\mu H$, where
the  relation between the eigenvalues of $H$ and  $\Phi$ is
  $$\lambda=-4\widetilde \lambda, \qquad \qquad \mu=-\tfrac 14 \ell(\ell+2)+\widetilde \mu-\widetilde \lambda.$$
Now the statement follows easily from Theorem \ref{param2}.
\end{proof}

{
\begin{cor}\label{autovsP}
  The function $P$ associated to the spherical function $\Phi^{(m_1,m_2)}_\ell$, defined by $H(u)= UT(u) P(u)$ (see \eqref{HP}), satisfies $\bar D P=\lambda P$ and $\bar E P=\mu P$ with
  \begin{equation*}
  \lambda= -(m_1-m_2)(m_1-m_2+2)  , \qquad
  \mu=-\tfrac {\ell(\ell+2)} 4 + (m_1+1)m_2.
\end{equation*}
\end{cor}

\begin{remark} \label{ninteger}
Notice that we have just proved that the eigenvalue $\lambda$ can be written in the form
\begin{equation*}
\lambda=-n(n+2), \qquad \quad \text{with }  \quad n\in\NN_{0}.
\end{equation*}
\end{remark}
}
\smallskip
\begin{prop}\label{parimpar}
 If $\Phi$ is an irreducible spherical function of
$\big(\SO(4),\SO(3)\big)$, then $\Phi(-e)=\pm I$. Moreover, if
$\Phi=\Phi_\ell^{(m_1,m_2)}$ and $g\in \SO(4)$, then
$$\Phi(-g)= \begin{cases}
      \Phi(g)&    \text{ if }\; m_1+m_2\equiv 0 \mod(2) \\
      -\Phi(g)&   \text{ if } \; m_1+m_2\equiv 1 \mod(2).
\end{cases}
$$
\end{prop}
\begin{proof}
As we mentioned in Section \ref{sec:prelim}, every irreducible spherical function of the pair $\big(\SO(4),\SO(3)\big)$ of type $\pi$, is of the form
$\Phi(g)= P_\pi \tau(g)$, where
$\tau \in \hat\SO(4)$  contains the $K$-type $\pi$  and  $P_\pi$ is the projection onto the $\pi$-isotypic component of $V_\tau$.

{ Let $\eta$ be the highest weight of $\tau=(m_1,m_2) \in \hat\SO(4)$, i.e. $\eta=m_1\varepsilon_1+m_2\varepsilon_2$ (see Subsection \ref{representations}).}
We have that  $$-e=\exp \left(\begin{smallmatrix} 0&\pi &0&0 \\-\pi& 0&0&0 \\0&0&0&\pi \\0&0&-\pi&0\end{smallmatrix}\right);$$
therefore, if  $v$ is a highest weight vector in $V_\tau$ of weight $\eta$, we have
$$\tau (-e)v
=e^{-\pi(m_1+m_2)i} v=\pm v.$$
Since  $\tau(-e)$ commutes with $\tau(g)$ for all $g\in \SO(4)$, by Schur's Lemma $\tau(-e)$ is a multiple of the identity . Thus,
$$\tau(-e)= \left\{
    \begin{array}{ll}
       I ,&   \hbox{if $m_1+m_2\equiv 0 \mod(2)$} \\
      -I ,&   \hbox{if $m_1+m_2\equiv 1 \mod(2)$}
    \end{array}.
  \right.
$$
Therefore,
\begin{align*}
\Phi(-g)&= P_\pi \tau(-g) =P_\pi \tau(-e)\tau(g) =\tau(-e)\Phi(g).
\end{align*}
Hence, the proposition is proved.
\end{proof}

\section{The function $P$ associated to a spherical function.}\label{ThefunctionP}
{
 In the previous sections we were interested in  studying the irreducible spherical functions $\Phi$ of a  $K$-type $\pi=\pi_\ell$.  This is accomplished by associating each function $\Phi$  to a $\CC^{\ell+1}$-valued function $H$, which is
  a simultaneous eigenfunction  of the differential operators $D$ and $E$ in $(0,1)$, given in \eqref{opDH} and \eqref{opEH}, continuous in $(0,1]$ and such
  that $H(1)= (1,\dots , 1)$.
  This function $H$
  is of the form
  \begin{equation*}
    H(u)= UT(u)P(u),
  \end{equation*}
where $U$ is the constant matrix given in \eqref{Ucolumnas}  and $  T(u)=\sum_{j=0}^\ell (1-u^2)^{j/2}E_{jj} $.

In this way  we have associated each irreducible spherical function $\Phi$ to a function $P(u)$, analytic in $(0,1]$, which is a simultaneous  eigenfunction
of the differential operators $\bar D$ and $ \bar E$, explicitly  given in Theorem \ref{puntos} by
\begin{align*}
  \bar D P &= (1-u^2) P'' -u CP'-V P, \mbox{} \\
  \bar EP & = \tfrac i 2\left( (1-u^2)Q_0+Q_1 \right) P'-\tfrac i2 u MP-\tfrac 12 V_0 P.
\end{align*}

From Corollary \ref{Dhyp} we have that
a vector valued eigenfunction $P=(p_0, \dots , p_\ell)^t$ of $\bar D$ with eigenvalue $\lambda=-n(n+2)$ and such that $\lim_{u\to 1^{-}} P(u)$ exists is given by
  \begin{align*}
  p_j(u)& = a_j \, {}_2\!F_1\left(\begin{smallmatrix}-n+j,n+j+2\\  j+3/2 \end{smallmatrix}; \tfrac{1-u}{2}\right)\, ,  \qquad  0\le j \le \ell,
\end{align*}
where  $a_j$ are arbitrary  complex numbers for $j=1,2,\ldots,\ell$.

\smallskip

Let us introduce the following vector space of functions into $\CC^{\ell+1}$, defined for $n\in \NN_0$ and $\mu\in \CC$,
\begin{equation}
{\mathcal  V}_{n, \mu} =\{P=P(u) \text{ analytic in }(0,1] \, : \, \bar D P=-n(n+2) P, \,\bar E P=\mu P \}.
\end{equation}

We observe that $\mathcal V_{n,\mu}\neq 0$
if and only if $\mu $ is an eigenvalue of the matrix $L(\lambda)$ given in \eqref{matrixL}, with $\lambda=-n(n+2)$.
We are interested in considering only the cases $\mathcal V_{n,\mu}\neq 0$.

\smallskip
From Corollary \ref{DEhyp} we have that a function $P\in \mathcal V_{n,\mu}$
is of the form $P=(p_0, \dots , p_\ell)^t$, where
  \begin{align}\label{formadeP}
  p_j(u)& = a_j \, {}_2\!F_1\left(\begin{smallmatrix}-n+j,n+j+2\\  j+3/2 \end{smallmatrix}; \tfrac{1-u}{2}\right)\, ,  \qquad  0\le j \le \ell,
\end{align}
and the coefficients
 $\{a_j\}_{j=0}^\ell$ satisfy the recursion relations \eqref{threetermI}.

We observe that the equation \eqref{threetermclosing} is  automatically  satisfied because    $\mu$ is an eigenvalue of  the matrix $L(\lambda)$ given in \eqref{matrixL}.

If the function $P$ is associated to an irreducible spherical function, then  we have $a_0=1$, because
the condition $H(1)= (1,\dots , 1)$ implies that $P(1)$  is a vector whose first entry is 1,  see \eqref{limiteH}.

\begin{prop}\label{losaj}
 If $P\in \mathcal V_{n,\mu}$ then $P$ is a polynomial function.
\end{prop}

\begin{proof}
Let $P(u)=(p_0(u),\dots,p_\ell(u))\in \CC^{\ell+1}$. From Corollary \ref{DEhyp} we have that the entries of the function
$P$  are given by
 \begin{equation}\label{geg}
   p_j(u)= a_j\,    {}_2\!F_1\left(\begin{smallmatrix}-n+j,n+j+2\\  j+3/2 \end{smallmatrix}; \tfrac{1-u}{2}\right),
\end{equation}
 where the coefficients
 $\{a_j\}_{j=0}^\ell$ satisfy the recursion relation \eqref{threetermI}, for some  eigenvalue $\mu$  of $L(\lambda)$.

For $0\leq j\leq n$,  the function $p_j(u)$  is a polynomial function, while for $n<j\leq \ell$ the series defining the hypergeometric function is not finite.
Hence, in this case we have that $p_j$ is a polynomial if and only if the coefficient $a_j$ is zero.

From the expression of $\bar E$ in Theorem \ref{puntos}, for $0\leq j\leq \ell$ we have
 \begin{equation}\label{Epj}
 \begin{split}
 \mu p_j   = \tfrac i 2  & \left( (1-u^2)\tfrac{(j+1)(\ell+j+2)}{2j+3} p'_{j+1}+\tfrac{j(\ell-j+1)}{2j-1}p'_{j-1} \right)\\
&-\tfrac i2 u (j+1)(\ell+j+2) p_{j+1}-\tfrac 12 j(j+1)p_j,
\end{split}
\end{equation}
where we interpret  $p_{-1}=p_{\ell+1}=0$.

By induction on $j$, suppose that  $p_j$ and $p_{j-1}$ are polynomial functions and let $p_{j+1}(u) = \sum _{k\geq 0} b_k u^k$.
Then,
\begin{align*}
  p(u)& =\tfrac{1}{2j+3} (1-u^2) p'_{j+1}(u)- u p_{j+1}(u)\\
  & = \tfrac 1{2j+3}\sum_{k\geq 0} \big( (k+1) b_{k+1} -  (k+2j+2)b_{k-1}\big) u^{k}
\end{align*}
is also a polynomial function in $u$, (where as usual we denote $b_{-1}=0$). Let $m=\deg(p)$.
Thus, for $k> m$
we have
$$b_{k+1}= \tfrac{(k+2j+1)}{k+2}\, b_{k-1},$$
then $|b_{k+1}|\ge|b_{k-1}|$ for $k> m$. Since $\lim_{u \to 1^- }{p_{j+1}(u)}$ exists, we have that $b_k=0$ for $k> m$.
Thus, $p_{j+1}$ is a polynomial.
Also, we conclude that if $n< \ell$ the $j$-th entry of $P$ is $p_j=0$ for $n<j\leq \ell$.

\end{proof}

\begin{cor}\label{Ppolynomial}
  The  function $P(u)$ associated to an irreducible spherical function $\Phi$ of type $\pi_\ell$ is a $\CC^{\ell+1}$-valued polynomial function.
\end{cor}

\begin{proof} We only have to recall that the function $P$ associated to the irreducible spherical function $\Phi_\ell^{(m_1,m_2)}$ of type $\pi_\ell$ belongs to $\mathcal V_{n,\mu}$, with
 $n=m_1-m_2\in \NN_0$ and $\mu=-\tfrac {\ell(\ell+2)} 4 + (m_1+1)m_2$ (see  Corollary \ref{autovsP}).
\end{proof}

\section{From $P$ to $\Phi$}\label{depafi}

{
\subsection{Correspondence between polynomials and spherical functions.}
\

In this subsection we will prove that, given $\pi_\ell\in\hat K$, there is a one to one}
correspondence between the vector valued polynomial eigenfunctions $P(u)$ of $\bar D$ and $ \bar E$ such that the first entry of the vector $P(1)$ is equal to $1$,
and the irreducible spherical functions $\Phi$ of type $\pi_\ell$.

\begin{thm}
There is a one to one correspondence between  the irreducible spherical  functions $\Phi$ of type $\pi_\ell\in\hat K$, $\ell\in2\NN_0,$
 and the functions $P$ in $\mathcal V_{n,\mu}\neq 0$ such that $P(1)=(1, a_1,\dots , a_\ell)$.
\end{thm}
\begin{proof} Given an irreducible spherical function of type $\pi_\ell$, we have already proved  that the function $P$ associated to it belongs to the space $\mathcal V_{n,\mu}$, and $P(1)$ has its first entry equal to one.

The equivalence classes of irreducible spherical functions of $(G,K)$ of
type $\pi_\ell$ are parameterized by  the
set of all pairs $(m_1,m_2)\in \ZZ^2$
such that
 $$m_1\geq\tfrac\ell2\geq |m_2|.$$

Every irreducible spherical function $\Phi^{(m_1,m_2)}_{\ell}$ corresponds to a vector valued eigenfunction $P_\ell^{(m_1,m_2)}$ of the operators $\bar D$ and $\bar E$ whose eigenvalues,
 according to Corollary \ref{autovs}, are  respectively
\begin{equation*}
  \begin{split}
  \lambda_\ell^{(m_1,m_2)}&= -(m_1-m_2)(m_1-m_2+2)  ,\\
  \mu_\ell^{(m_1,m_2)}&=-\tfrac {\ell(\ell+2)} 4 + (m_1+1)m_2.
  \end{split}
\end{equation*}

Easily one can see that for different pairs  $(m_1,m_2)$
the pairs of eigenvalues $(\lambda_\ell^{(m_1,m_2)}, \mu_\ell^{(m_1,m_2)})$ are different.
Thus, each eigenfunction $P^{(m_1,m_2)}_{\ell}$ is associated to a unique irreducible spherical function $\Phi^{(m_1,m_2)}_{\ell}$.

On the other hand, from  \eqref{formadeP} we know that $P\in \mathcal V_{n,\mu}$
if and only if $P(u)=(p_0(u),\dots,p_\ell(u))^t$ is of the form
\begin{align*}
  p_j(u)& = a_j \, {}_2\!F_1\left(\begin{smallmatrix}-n+j,n+j+2\\  j+3/2 \end{smallmatrix}; \tfrac{1-u}2\right),\quad \text{ for all $0\leq j\leq \ell$},
\end{align*}
 where $\{a_j\}_{j=0}^\ell$ satisfies  \eqref{threetermI}.
Thus,  $a=(a_0,\dots,a_\ell)^t$ is an eigenvector of the matrix $L(\lambda)$ with eigenvalue $\mu$.
 In particular there are no more than $\ell+1$ linear independent eigenvectors.
If $P\in \mathcal V_{n,\mu}$  then  $P$ is a polynomial function; hence, when $n< \ell$ we have that
 $$a_j=0, \qquad \text{  for $n<j\leq \ell$} .$$
Thus, the eigenvectors of $L(\lambda)$  live in a subspace of dimension $n+1$ and, hence, there are at most $n+1$ linear independent eigenvectors.
From  Corollary \ref{corttr} we get that  every eigenspace of $L(\lambda)$ is one dimensional.
Therefore we conclude that, up to scalars, there are no more than $\min \{\ell+1, n+1\}$
 eigenvectors of $L(\lambda)$.

{ Hence, it is enough to prove that for each $\lambda=-n(n+2)$, with $n\in \NN_0$, there are exactly $\min\{\ell+1, n+1\}$ irreducible spherical functions of type $\pi_\ell\in K$}.

It is easy to verify (see Figure \ref{fig2}) that    there are exactly
 $\min\{\ell+1,n+1\}$ pairs $(m_1,m_2)\in \ZZ \times \ZZ$, satisfying
\begin{equation}\label{cond}
m_1\geq\tfrac\ell2\geq |m_2|\quad \text{  and }  \quad m_1-m_2=n.
\end{equation}

\begin{figure}
 \centering
\includegraphics{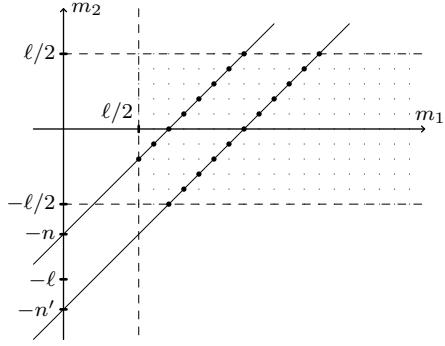}
 \caption[Fig]{Pairs $(m_1,m_2)$ that contain the $K$-type $\ell$.}%{Pairs $(m_1, m_2)$ satisfying \eqref{cond} for $n<\ell$ and for $n'\ge\ell$.}
\label{fig2}
\end{figure}

\noindent This concludes  the proof of the theorem.
\end{proof}

{  If we take $n=m_1-m_2$ and $k=\ell/2-m_2$ in Corollary \ref{autovs} we have that for an eigenfunction $P(u)$ of $D$ and $E$, associated }
to an irreducible spherical function of type $\pi_\ell\in\hat K$,
 the respective eigenvalues are of the form
\begin{equation*}
  \begin{split}
    \lambda&= -n(n+2)  ,\qquad  {  \mu=-\tfrac\ell2\left(\tfrac\ell2+1\right)+\left(n-k+\tfrac\ell2+1\right)\left(\tfrac\ell2-k\right)},
  \end{split}
\end{equation*}
\noindent with $0\le n$ and $0\le k\le\min(n,\ell)$.

\smallskip Now we can state the main theorem of this paper.

\begin{thm}\label{main}
There exists a one to one correspondence between the irreducible spherical functions of type $\pi_\ell\in \hat K$ and the
vector valued polynomial functions $P(u)=(p_0(u),\dots,p_\ell(u))^t$ with
  \begin{align*}
  p_j(u)& = a_j \, {}_2\!F_1\left(\begin{smallmatrix}-n+j,n+j+2\\  j+3/2 \end{smallmatrix}; \tfrac{1-u}2\right),
\end{align*}
 where $n\in{\NN_0}$,  $a_0=1$ and $\{a_j\}_{j=0}^\ell$ satisfies the recursion relation
\begin{equation*}
i \,\tfrac{j(\ell-j+1)(n-j+1)(n+j+1)}{2(2j-1)(2j+1)} \,\, a_{j-1} - \tfrac {j(j+1)}2 \,\,a_j
 - i   \tfrac{(j+1)(\ell+j+2)}2 \,\,a_{j+1} = \mu \, a_j,
\end{equation*}
for $0\leq j\leq \ell-1$, and  $\mu$ of the form
 $$\mu=-\tfrac\ell2\left(\tfrac\ell2+1\right)+\left(n+k-\tfrac\ell2+1\right)\left(k-\tfrac\ell2\right),$$
for $k\in\ZZ$, $0\le k\le \min\{n,\ell\}$.
\end{thm}

\subsection{Reconstruction of an irreducible spherical function.}\label{reconstruccion}
\

Fixed $\ell\in2{\NN_0}$ we know that a function $P=P(u)$ as in Theorem \ref{main} is associated to a unique irreducible spherical
function $\Phi$ of type $\pi_\ell\in\hat K$. Now we show how to explicitly construct the function $\Phi$ from such a $P$.
Recall that $P$ is a polynomial function.

Let us define the vector function $H(u)=UT(u)P(u)=(h_0(u),\dots,h_\ell(u))$, $u\in[-1,1]$, with $U$ and $T(u)$ as in Corollary \ref{DEhyp}
and let $\text{diag}\{H(u)\}$ denote the diagonal matrix valued function whose $kk$-entry is equal to the $k$-th  entry of the vector valued function $H(u)$.

  On the other hand, if we consider the function  $H:S^3\longrightarrow\End(V_\pi)$ associated to the irreducible spherical function
$\Phi$, from Corollary \ref{DEhyp} and (\ref{u}) we know that for $u\in(0,1)$
$$H(\sqrt{1-u²},0,0,u)=H(u).$$
Therefore, since both functions in the equality above are analytic in $(-1,1)$ and  continuous  in $[-1,1]$, we have that
$$H(\sqrt{1-u²},0,0,u)=H(u),$$
for all $u\in[-1,1]$.

Since $H(kx)=\pi_\ell(k)H(x)\pi_\ell^{-1}(k)$ for every $x\in S^3$ and $k\in K$, we have found the explicit values of the function $H$ on
the sphere $S^3$. Then, we can define the function $H:G\longrightarrow\End(V_{\pi_\ell})$ by $$H(g)=H(gK),\qquad g\in G.$$

\noindent Finally, we have that the irreducible spherical function $\Phi$ is of the form
$$\Phi(g)=H(g)\Phi_{\pi_\ell}(g),\qquad g\in G,$$
where $\Phi_{\pi_\ell}$ is the auxiliary spherical function introduced in Subsection \ref{auxiliar}.

\section{Hypergeometrization}\label{hyper}
{
In this section, for a fixed $\ell\in2{\NN_0}$ we shall construct a sequence of matrix valued polynomials $P_w$
closely related to irreducible spherical functions of type $\pi_\ell\in\hat K$.

 Given a nonnegative integer $w$ and $k =0,1,2,\dots ,\ell$,  the integers $m_1=w+ \ell/ 2$ and $m_2=-k+ \ell/ 2$ satisfy
 $$ w+\tfrac \ell 2 \geq \tfrac  \ell 2 \geq  \left| -k+\tfrac \ell 2\right| . $$
 Then, we can consider
 $$\Phi^{({w}+\ell/2,-{k}+\ell/2)}_{\ell},$$
 the spherical function of type $\pi_\ell\in\hat K$ associated to the $G$-representation $\tau_{(m_1, m_2)}$.

Also let us consider the matrix valued function $P_w= P_w(u)$,
whose $k$-th column  ($k =0,1,2,\dots ,\ell$) is given by the $\CC^{\ell+1}$-valued polynomial $P$ associated to
$\Phi^{({w}+\ell/2,-{k}+\ell/2)}_{\ell}$.

From Corollary \ref{autovsP}, we have that the $k$-th column of $P_w$ is an
eigenfunction of the operators $\bar D$ and $\bar E$ with eigenvalues
$\lambda_w(k)=-(w+k)(w+k+2)$ and $\mu_w(k)=w(\tfrac\ell2-k)-k(\tfrac\ell2+1)$
 respectively.

 Explicitly, we have  that the $jk$-entry of the matrix $P_w$ is given by
\begin{equation}\label{pes}
[ P_w(u)] _{jk}=a_j^{w,k} \,\, {}_2\!F_1\left(\begin{smallmatrix}-w-k+j,w+k+j+2\\  j+3/2 \end{smallmatrix}; (1-u)/2\right),
\end{equation}}
\noindent where $ {a}^{w,k}_{0}=1$ for all $k$ and $\{a_j^{w,k}\}_{j=0}^\ell$ satisfies
\begin{equation}
\begin{split}\label{ttr}
& i \tfrac{j(\ell-j+1)(w+k-j+1)(w+k+j+1)}{2(2j-1)(2j+1)} \,\, {a}^{w,k}_{{j-1}} - \tfrac {j(j+1)}2 \,\,{a}^{w,k}_{{{j}}}
 - i   \tfrac{(j+1)(\ell+j+2)}2 \,\,{a}^{w,k}_{{{j+1}}} \\
&= \left(w(\tfrac\ell2-k) -k(\tfrac\ell2+1)\right){a}^{w,k}_{{j}}.
\end{split}
\end{equation}

From Proposition \ref{Ppolynomial}, with $n=w+k$, we have that $[ P_w(u)] _{jk}$ is  polynomial on $u$. Therefore, we have the following results.

\begin{prop}\label{Pweigenfunction}
  The matrix valued polynomials $P_w$ defined above satisfy
  $$\bar D  P_w=P_w \Lambda_w \qquad \text{ and } \qquad \bar E P_w=P_w M_w,$$
  where $\Lambda_w=\sum_{k=0}^\ell \lambda_w(k) E_{kk}$, $M_w=\sum_{k=0}^\ell \mu_w(k) E_{kk}$, and
$$\lambda_w(k)=-(w+k)(w+k+2) \qquad \text{and}\qquad \mu_w(k)=w(\tfrac\ell2-k)-k(\tfrac\ell2+1).$$
\end{prop}

\smallskip
For the particular case  $w=0$, we have the following explicit formulas for $a_j^{0,k}$.

\begin{prop} We have
\begin{equation}\label{as}
\begin{split}
a_j^{0,k} &= \,\frac{(-2i)^{j}k!\,j!}{(k-j)!(2j)!}\quad\text{for } 0\leq j\leq k\leq \ell, \\
a_j^{0,k}&=0 \quad \text{for} \; 0\leq k< j\leq \ell.
\end{split}
\end{equation}
\end{prop}
\begin{proof}
 Clearly $a_0^{0,k}=1$; then, we only need  to check that these $a_j^{0,k}$ satisfy the following three-term recursive relation:
\begin{align}\label{qq}
 i \tfrac{j(\ell-j+1)(k-j+1)(k+j+1)}{2(2j-1)(2j+1)}  {a}^{0,k}_{{j-1}} - \tfrac {j(j+1)}2 {a}^{0,k}_{{{j}}}
 - i\tfrac{(j+1)(\ell+j+2)}2 {a}^{0,k}_{{{j+1}}} = -\tfrac{k(\ell+2)}{2}{a}^{0,k}_{{{j}}}.
 \end{align}
Notice that if the coefficients $a_j^{0,k}$ are given by (\ref{as}), for $ 0\leq j\leq k\leq \ell$ we have
\begin{align*}
ia_{j-1}^{0,k}=-\tfrac{2j-1}{k-j+1}a_j^{0,k}\qquad  \text{and}  \qquad ia_{j+1}^{0,k}=\tfrac{k-j}{2j+1}a_j^{0,k} .
 \end{align*}
Hence, for $ 0\leq j\leq k\leq \ell$ (\ref{qq}) is equivalent to
\begin{align*}
-  \tfrac{j(\ell-j+1)(k+j+1)}{2(2j+1)}  {a}^{0,k}_{{j}} - \tfrac {j(j+1)}2 {a}^{0,k}_{{{j}}}
 - \tfrac{(j+1)(\ell+j+2)(k-j)}{2(2j+1)} {a}^{0,k}_{{{j}}} = -\tfrac{k(\ell+2)}{2}{a}^{0,k}_{{{j}}},
 \end{align*}
which can be easily checked.

\noindent If $j=k+1$ we have
 \begin{align*}
 i \tfrac{(k+1)(\ell-(k+1)+1)(k-(k+1)+1)(k+(k+1)+1)}{2(2j-1)(2j+1)} \,\,  {a}^{0,k}_{k}  = 0,
 \end{align*}which is true. And if $j\geq k+2$ we just have $0=0$.
Therefore, the coefficients given by (\ref{as}) satisfy (\ref{ttr}) and the proof is finished.
\end{proof}

\subsection{The hypergeometric operators}
Now we introduce the matrix valued function $\Psi$ defined by the first ``package" of spherical functions $P_w$ with $w\ge0$, i.e.,
 \begin{equation*}
 \Psi(u) =  P_0(u).
\end{equation*}
{  From \eqref{pes} and \eqref{as} we observe that $\Psi(u)$ is an upper triangular matrix. Moreover, $\Psi(u)=(\Psi_{jk})_{jk}$ is the polynomial function given by \begin{equation}\label{psi}
\Psi _{jk}=\frac{(2j+1)(-2i)^j\, k!j!}{(k+j+1)!}C^{j+1}_{k-j}(u), \quad \text{ for $0\le j\leq k\le\ell$,}
\end{equation}}
{  where $C^{j+1}_{k-j}(u)$ is the Gegenbauer polynomial }
\begin{equation*}
C^{j+1}_{k-j}(u)=\binom{k+j+1}{k-j}\,{}_2\!F_1\left(\begin{smallmatrix}-k+j\,,\, k+j+2\\  j+3/2 \end{smallmatrix}; (1-u)/2\right).
\end{equation*}

Since the $k$-th column of $\Psi$ is an eigenfunction of $\bar D$ and $\bar E$ with eigenvalues $\lambda_0(k)=-k(k+2)$ and $\mu_0(k)=-k(\tfrac\ell2+1)$ respectively, the function $\Psi$ satisfies
\begin{equation}\label{Psiautofuncion}
  \bar D \Psi=\Psi \Lambda_0 \quad \text{ and } \quad   \bar E\Psi= \Psi M_0,
\end{equation}
where $\Lambda_0=\sum_{k=0}^\ell \lambda_0(k) E_{kk}$ and $M_0=\sum_{k=0}^\ell \mu_0(k) E_{kk}$.
{
\remark\label{psi-1} The entries of the diagonal of $\Psi(u)$ are nonzero constant polynomials, thus we have that $\Psi(u)$ is invertible. \\ Moreover, the inverse $\Psi(u)^{-1}$
is also an upper triangular matrix polynomial. This can be easily checked, for instance, using  Cramer's rule, because}
the determinant of $\Psi(u)$ is a nonzero constant.

\begin{thm}\label{hyp}
Let $\bar D$ and $\bar E$ be the differential operators defined in
Theorem \ref{puntos}, and let $\Psi$ the matrix valued function whose entries are given by  \eqref{psi}. Let $\widetilde
D = \Psi^{-1} \bar D \Psi$ and $\widetilde E = \Psi^{-1} \bar E
\Psi$, then
\begin{align*}
\widetilde D F=&(1-u^2) F'' + (-u C + S_1)  F'+  \Lambda_0  F,\\\widetilde E F=& (uR_2 + R_1)  F'+ M_0 F,
\end{align*}
for any $C^\infty$-function $F$ on $(0,1)$ with values in $\mathbb C^{\ell+1}$,
where\begin{align*}
C &= \sum_{j=0}^\ell (2j+3)E_{jj}, & S_1 &=  {  \sum_{j=0}^{\ell-1} 2(j+1)E_{j,j+1},}\\
R_1 &= \sum_{j=0}^{\ell-1} \tfrac{(j+1)}{2} E_{j,j+1}-\sum_{j=0}^{\ell-1} \tfrac{(\ell-j)}{2} E_{j+1,j},
& R_2 &= \sum_{j=0}^{\ell} (\tfrac \ell2-j)E_{j,j},\\
 \Lambda_0 &= \sum_{j=0}^{\ell}  -j(j+2)E_{j,j}, & \qquad M_0&= {  \sum_{j=0}^{\ell}-j(\tfrac\ell2+1)E_{j,j}}.
\end{align*}
\end{thm}

\begin{proof}
By definition we have
{  \begin{align*}
\widetilde D\widetilde F&=(1-u^2)\widetilde F'' + \Psi^{-1}[2(1-u^2)\Psi ' - u C \Psi ] \widetilde F'\\
&\quad + \Psi^{-1} \left[(1-u^2) \Psi'' -u C\Psi'-V \Psi\right] \widetilde F,\\
\widetilde E\widetilde F&=\tfrac i2 \Psi^{-1}[(1-u^2)Q_0+Q_1]\Psi
\widetilde F'\\
& \quad + \Psi^{-1}\left[\tfrac i 2\left( (1-u^2)Q_0+Q_1 \right) \Psi'-\tfrac i2
u M\Psi-\tfrac 12 V_0 \Psi\right] \widetilde F.
\end{align*}
By using \eqref{Psiautofuncion} we observe that
\begin{align*}
 (1-u^2) \Psi'' -u C\Psi'-V \Psi & =\bar D \Psi=\Psi\Lambda_0,\\
\tfrac i 2 \big( (1-u^2)Q_0+Q_1 \big) \Psi'-\tfrac i2 u M\Psi-\tfrac 12 V_0 \Psi & =\bar E \Psi=\Psi M_0.
\end{align*}

To complete the proof of this theorem, we use the following properties of the Gegenbauer polynomials (for the first three see \cite{KS} page 40, and for the last one see \cite{S}, page 83, equation (4.7.27))}
\begin{align}
\label{a}&\frac{d C_n^\lambda}{du}(u)=2\lambda C_{n-1}^{\lambda+1}(u),\\
\label{b}&2(n+\lambda)u C_n^\lambda (u)= (n+1)C_{n+1}^{\lambda}(u)+(n+2\lambda-1)C_{n-1}^{\lambda}(u),\\
\label{c}& (1-u^2)\frac{d C_n^\lambda}{du}(u) +(1-2\lambda)u C_{n}^{\lambda}(u) =-\frac{(n+1)(2\lambda+n-1)}{2(\lambda-1)}C_{n+1}^{\lambda-1}(u),\\
\label{d}&\frac{(n+2\lambda-1)}{2(\lambda-1)} C_{n+1}^{\lambda-1}(u)=  C_{n+1}^{\lambda}(u) -  u C_{n}^{\lambda}(u).
\end{align}

We need to establish the following identities
\begin{align}
 \label{paraD}[2(1-u^2)\Psi ' - u C \Psi ] &=\Psi (-u C + S_1),\\
 \label{paraE}\tfrac i2[(1-u^2)Q_0+Q_1]\Psi  &= \Psi (uR_2 + R_1).
 \end{align}
Since (\ref{paraD}) is a matrix identity, by looking at the  $jk$-place
we have
\begin{align*}
2(1-u^2)\Psi'_{jk} - u C_{jj} \Psi_{jk}  =-\Psi_{jk} u C_{kk} &+ \Psi_{j,k-1} (S_1)_{k-1,k}.
\end{align*}
Multiplying both sides by $\frac{(k+j+1)!}{(2j+1)(-2i)^j\, k!j!}$ and using \eqref{psi} we have
\begin{align*}
2(1-u^2)\frac{d }{du}C_{k-j}^{j+1} - u (2j+3)C_{k-j}^{j+1} = -u(2k+3)C_{k-j}^{j+1} +2(k+j+1)C_{k-j-1}^{j+1}.
\end{align*}
and by setting $\lambda=j+1$ and $n=k-j$ we get
\begin{align*}
2(1-u^2)\frac{d C_n^\lambda}{du} - u (2\lambda+1)C_n^\lambda = -u(2(n+\lambda)+1)C_n^\lambda+2(n+2\lambda-1)C_{n-1}^\lambda.
\end{align*}
{ To see that this identity holds, we use \eqref{c} to write $\frac{d C_n^\lambda}{du}$ in terms of $C_{n}^{\lambda} $ and $C_{n+1}^{\lambda-1}$,
\eqref{b} to express $C_{n-1}^{\lambda}$ in terms of $C_{n}^{\lambda}$ and  $C_{n+1}^{\lambda}$, and then
 we recognize the identity  \eqref{d}. Thus, we have proved \eqref{paraD}.}

\smallskip
Now we need to verify the matrix identity  \eqref{paraE}. The $jk$-entry is given by (see Theorem \ref{puntos}
for the definition of the matrices  $Q_0$ and $Q_1$)
\begin{align*}
\tfrac i2(1-u^2)(Q_0)_{j,j+1}\Psi_{j+1,k}  +\tfrac i2(Q_1)_{j,j-1} \Psi_{j-1,k} &=\\
 u \Psi_{jk} (R_2)_{kk}+\Psi_{j,k+1} (R_1)_{k+1,k}  &+ \Psi_{j,k-1} (R_1)_{k-1,k}.
 \end{align*}
Again, multiplying both sides by $\frac{(k+j+1)!}{(-2i)^j\, k!j!}$ and setting $\lambda=j+1$ and $n=k-j$, we obtain
\begin{align*}
(1&-u^2)\frac{\lambda^2(\ell+\lambda+1)}{n+2\lambda} C_{n-1}^{\lambda+1}- \frac{(\ell-\lambda+2)(n+2\lambda-1)}{4}C_{n+1}^{\lambda-1}=\\
&u\left(\tfrac\ell2-(n+\lambda-1)\right)(2\lambda-1)C_{n}^{\lambda}
-\frac{(\ell-n-\lambda+1)(2\lambda-1)(n+\lambda)}{2(n+2\lambda)}C_{n+1}^{\lambda}\\
&+\frac{(2\lambda-1)(n+2\lambda-1)}{2}C_{n-1}^{\lambda}.
\end{align*}

\noindent Now we  firstly  use (\ref {c}) combined with \eqref{a} to write $C_{n-1}^{\lambda+1}$ in terms of $C_{n}^{\lambda}$ and $C_{n+1}^{\lambda-1}$, and then we use \eqref{b} to express $C_{n-1}^{\lambda}$ in terms of $C_{n}^{\lambda}$ and  $C_{n+1}^{\lambda}$. Then we get
\begin{align*}
&\tfrac{(\lambda^2-\ell\lambda-3\lambda-\ell n-2n)(2\lambda-1)}{2(n+2\lambda)}\left(u C_n^\lambda(u) + \frac{(n+2\lambda-1)}{2(\lambda-1)} C_{n+1}^{\lambda-1}(u) -C_{n+1}^\lambda(u)\right)=0,
\end{align*}
which is true by (\ref{d}).
\end{proof}

{ For each $w\in\NN_0$  let us introduced the matrix valued function
\begin{equation}\label{Pwtilde}
  \widetilde P_w= \Psi^{-1} P_w,
\end{equation}
 where  $P_w$ is the matrix valued polynomials introduced in \eqref{pes} and
$\Psi$ the upper triangular matrix function given in \eqref{psi}. We recall
that the function $\Psi^{-1}$ is a polynomial function, as we observed in Remark \ref{psi-1}. Therefore,  $\widetilde P_w$ is also a polynomial function.
The following result is a direct consequence of Proposition \ref{Pweigenfunction} and Theorem \ref{hyp}.
\begin{cor}  \label{tildePweigenfunction}
  The matrix valued polynomials $\widetilde P_w= \Psi^{-1} P_w$  satisfy
  $$\widetilde D  \widetilde P_w=\widetilde P_w \Lambda_w \qquad \text{ and } \qquad \widetilde  E \widetilde P_w=\widetilde P_w M_w,$$
  where $\Lambda_w=\sum_{k=0}^\ell \lambda_w(k) E_{kk}$, $M_w=\sum_{k=0}^\ell \mu_w(k) E_{kk}$, and
  $$\lambda_w(k)=-(w+k)(w+k+2) \qquad \text{and}\qquad \mu_w(k)=w(\tfrac\ell2-k)-k(\tfrac\ell2+1).$$
\end{cor}

\subsection{The explicit expression of the coefficients $a_j^{w,k}$}\

\color{black}

In this subsection we give the expression of the coefficients $a_j^{w,k}$, defined by the relation \eqref{ttr}, in terms of the Racah polynomials.
We recall that the Racah polynomials are defined (see for example (1.2.1) in \cite{KS}) by
\begin{equation}\label{Racah}
  R_k(\lambda(x);\alpha,\beta,\gamma, \delta)={}_4F_{3}\left( \begin{matrix}
    -k,k+\alpha+\beta+1, -x,x+\gamma+\delta+1\\ \alpha+1,\beta+\delta+1,\gamma+1
  \end{matrix}\, ; 1\right)
\end{equation}
for $n=0,1, \dots N$, where $\lambda(x)=x(x+\gamma+\delta+1)$ and one of the numbers $\alpha+1, \beta+\delta+1$ or $\gamma+1$ are equal $-N$, with $N$ a nonnegative integer.

If we take
$$\alpha= -(\ell+1), \quad \beta=-(w+k+1), \quad \gamma=0,\quad  \delta= 0, \quad N=\ell, \quad  x=j,  $$
then the Racah poynomials
\begin{align*} R_k(\lambda(j))& =R_k(\lambda(j);-\ell-1,-w-k-1,0,0)\\
& ={}_4F_{3}\left( \begin{matrix}
    -k,-\ell-w-1, -j,j+1\\ -\ell,-k-w,1
  \end{matrix}\, ; 1\right)
  \end{align*}
satisfy the difference equation (see \cite{KS}, equation (1.2.5))
\begin{equation}\label{diffequ}
  \begin{split}
     \tfrac{(j-\ell)(j-k-w)(j+1)}{(2j+1)}  \, R_k(\lambda(j+1)) + \tfrac{j (j+\ell+1)(j+k+w+1)}{(2j+1)} \, R_k(\lambda(j-1))&
   \\
   + \big (2k(w+\ell+1)-j(j+1)-\ell(w+k)\big ) R_k(\lambda(j))&= 0.
  \end{split}
\end{equation}

\begin{prop} \label{ajwkRacah}
For $w\in \NN_0$ and $0\leq j,k\leq \ell$ let %The  sequence % coefficients defined by  are given by
$$a_j^{w,k}=(-2i)^j (-w-k)_j \frac{j!}{(2j)!} \textstyle\binom{\ell}{j}\binom{\ell+j+1}{j}^{-1} \, R_k(\lambda(j); -\ell-1,-w-k-1,0,0).$$
Then $a_0^{w,k}=1$ and the sequence $\{a_j^{k,w}\}$ satisfies the recursion  relation \eqref{ttr}.
  \end{prop}
\begin{proof}
It is easy to check  that $a_0^{w,k}=1$. Thus we start by observing that
\begin{align*}
  i\,a_{j+1}^{w,k} &= (-2i)^j (-w-k)_j \tfrac{j!}{(2j)!} \textstyle\binom{\ell}{j}\binom{\ell+j+1}{j}^{-1} \frac{(-w-k+j)(\ell-j)}{(2j+1)(\ell+j+2)} \, R_k(\lambda(j+1)) ,\\
   i\,a_{j-1}^{w,k} &= (-2i)^j (-w-k)_j \tfrac{j!}{(2j)!} \textstyle\binom{\ell}{j}\binom{\ell+j+1}{j}^{-1} \frac{(2j-1)(\ell+j+1)}{(w+k-j+1)(\ell-j+1)} \, R_k(\lambda(j-1)).
\end{align*}

\smallskip

\noindent   The left-hand side of \eqref{ttr} becomes
  \begin{align*}
 & i \tfrac{j(\ell-j+1)(w+k-j+1)(w+k+j+1)}{(2j-1)(2j+1)} \,\, {a}^{w,k}_{{j-1}} - {j(j+1)} \,\,{a}^{w,k}_{{{j}}}
 - i   (j+1)(\ell+j+2) \,\,{a}^{w,k}_{{{j+1}}} \\
 & = (-2i)^j (-w-k)_j \tfrac{j!}{(2j)!} \textstyle\binom{\ell}{j}\binom{\ell+j+1}{j}^{-1} \Big( \tfrac{j(w+k+j+1)(\ell+j+1)}{2j+1}\, R_k(\lambda(j-1)) \\
 & \qquad \qquad  -j(j+1) R_k(\lambda(j))+ \tfrac{(j+1)(\ell-j)(w+k-j)}{2j+1}R_k(\lambda(j+1) \Big).
 \intertext{By using the difference equation \eqref{diffequ} we have that the expression above is equal to}
 &   (-2i)^j (-w-k)_j \tfrac{j!}{(2j)!} \textstyle\binom{\ell}{j}\binom{\ell+j+1}{j}^{-1} \big( \ell(w+k)-2k(w+\ell+1) \big) \,R_k(\lambda(j+1) \\
&= \big( w(\ell-2k) -k(\ell+2)\big) {a}^{w,k}_{{j}}.
\end{align*}
This turns out to be the right-hand side of \eqref{ttr} and the proof is complete.
\end{proof}

\begin{cor} For  $w=0$ we obtain
  $$a_{j}^{0,k}=(2i)^j (-k)_j \tfrac{j!}{(2j)!}.$$
\end{cor}
\begin{proof} We start by observing that for $w=0$ the Racah polynomial involved in the expression of $a_j^{0,k}$ can be written as a ${}_3\!F_2$ function
\begin{align*}
   R_k(\lambda(j);-\ell-1,-k-1,0,0) & ={}_4\!F_{3}\left( \begin{smallmatrix}
    -k,-\ell-1, -j,j+1\\ -\ell,-k,1   \end{smallmatrix}\, ; 1\right)  =
    {}_3\!F_{2}\left( \begin{smallmatrix}
    -j,j+1,-\ell-1\\ 1, -\ell   \end{smallmatrix}\, ; 1\right) .
\end{align*}

By Pfaff-Saalsch\"utz identity (see for example \cite{AAR}, Theorem 2.2.6) we get
$${}_3\!F_{2}\left( \begin{smallmatrix}
    -j,j+1,-\ell-1\\ 1, -\ell   \end{smallmatrix}\, ; 1\right) = \frac{(-j)_j \, (\ell+2)_j}{(1)_j\,(\ell-j+1)_j }= (-1)^j \textstyle\binom{\ell+j+1}{j} \binom{\ell}{j}^{-1}.$$
Now the corollary follows directly from Proposition \ref{ajwkRacah}.
\end{proof}
\remark The expression of $a_{j}^{0,k}$ in Corollary \ref{ajwkRacah} coincides with the  result obtained  in \eqref{as}.

\begin{cor}For $0\leq j,k\leq \ell$ the $jk$-entry of the polynomial $P_w$ is given by
\begin{align*}
 [P_w(u)]_{jk}& = (-2i)^j (-w-k)_j \frac{j!}{(2j)!}  \textstyle\binom{\ell}{j}\binom{\ell+j+1}{j}^{-1} \, {}_4\!F_{3}\left( \begin{smallmatrix}
    -k,-\ell-w-1, -j,j+1\\ -\ell,-k-w,1
  \end{smallmatrix}\, ; 1\right)\\
 &\quad\times  {}_2\!F_1\left(\begin{smallmatrix}-w-k+j,w+k+j+2\\  j+3/2 \end{smallmatrix}; (1-u)/2\right).
 \end{align*}
\end{cor}

\color{black}

\section{Orthogonal Polynomials}\label{Matrix Orthogonal Polynomials}

The aim of this section is to build classical sequences of matrix valued orthogonal polynomials from our previous work. This means to exhibit
a weight matrix $W$ supported on the real line, a sequence $(\widetilde P_w)_{w\ge0}$ of matrix polynomials such that $\deg(\widetilde P_w)=w$ with the leading coefficient of
$\widetilde P_w$ nonsingular, orthogonal with respect to $W$, and a second order (symmetric) differential operator $\widetilde D$ such that $\widetilde D\widetilde P_w=\widetilde P_w\Lambda_w$ where $\Lambda_w$ is a real diagonal matrix. Moreover, we point out that we also have a first order (symmetric) differential operator $\widetilde  E$ such that $\widetilde E\widetilde P_w=\widetilde P_w M_w$, where $M_w$ is a real diagonal matrix.

From $\widetilde D$ (see Theorem \ref{hyp}) we obtain a new differential operator $D$ by making the change of variables $s=(1-u)/2$. Thus,
\begin{align*}
 D F=&s(1-s) F'' - \left(\frac{S_1-C}{2}+sC\right)  F'+  \Lambda_0 F.
\end{align*}

\subsection{Polynomial solutions of $DF=\lambda F$}
\

We are interested in studying the vector valued polynomial solutions of the equation $DF=\lambda F$; in particular we want to know when polynomial solutions exist.
We start with
\begin{align}\label{ocho}
s(1-s) F'' + \left(B-sC\right)  F'+  \left(\Lambda_0 -\lambda\right)F=0,
\end{align}
where
\begin{align*}
B& =\frac{C-S_1}{2}=  \sum_{j=0}^\ell (j+\tfrac 32 )E_{jj} - \sum_{j=0}^{\ell-1} (j+1)E_{j,j+1}, \\
C&= \sum_{j=0}^\ell (2j+ 3 )E_{jj},   \qquad S_1=\sum_{j=0}^{\ell-1} (j+1)E_{j,j+1}, \qquad \Lambda_0  =- \sum_{j=0}^\ell j(j+2 )E_{jj}.
\end{align*}

 This equation is an instance of a matrix hypergeometric differential equation studied in \cite{T4}. Since the eigenvalues of $B$ are not in $-\NN_0$, the function $F$ is determined by $F_0=F(0)$. For $|s|<1$ it is given by
 \begin{align*}
F(s)={}_2\!H_1\left(\begin{smallmatrix}C,-\Lambda_0+\lambda\\  B \end{smallmatrix}; s\right)F_0=\sum_{j=0}^{\infty}\frac{s^j}{j!} [B;C;-\Lambda_0+\lambda]_j F_0, \qquad F_0\in \CC^{\ell+1},
\end{align*}
where the symbol $[B;C;-\Lambda_0+\lambda]_j$ is inductively defined by
\begin{align*}
[B;C;-\Lambda_0+\lambda]_0   &=1,\\
[B;C;-\Lambda_0+\lambda]_{j+1} &=
\left(B+j\right)^{-1}(j(C+j-1)-\Lambda_0+\lambda)[B;C;-\Lambda_0+\lambda]_j ,
\end{align*}
for all $j\geq 0$.

Therefore, there exists a polynomial solution of \eqref{ocho} if and only if the coefficient
$[B;C;-\Lambda_0+\lambda]_j$ is a  singular matrix
for some $j\in {\NN_0}$.
Moreover, we have that there is a polynomial solution of degree
$w$ of (\ref{ocho}) if and only if  there exists $F_0\in\CC^{\ell+1}$ such that
$[B;C;-\Lambda_0+\lambda]_{w}F_0\neq 0$  and
$$(w(C+w-1)-\Lambda_0+\lambda) F_w=0,\quad \text { where } \quad F_w= [B;C;-\Lambda_0+\lambda]_{w}F_0.$$
The matrix
\begin{equation}\label{Mw}
  M_w=w(C+w-1)-\Lambda_0+\lambda=\sum_{j=0}^{\ell}((j+w)(j+w+2)+\lambda)E_{jj}
\end{equation}
is diagonal. Then, it is a singular matrix if and only if $\lambda$ is of the form
$$\lambda_w(k)=-(k+w)(k+w+2),$$
for $0\leq k\leq\ell.$ We get the following result.
\begin{prop}\label{polynomialsol}
Given $\lambda\in \CC$, the equation $DF=\lambda F$ has a polynomial solution if and only if $\lambda$ is of the form $-n(n+2)$ for
$n\in\NN_0$.
\end{prop}

\begin{remark}\label{igualdadautoval}
  Let $w\in \NN_0$, $0\leq k \leq \ell$. The eigenvalue $\lambda_w(k)$ satisfies
$\lambda_w(k)  =-n(n+2)$ with $n\in \NN_0$ if and only if  $n=w+k$.  In particular,
$$\lambda_w(k)=\lambda_{w'}(k') \quad \text{if and only if } \quad w+k=w'+k'.$$
\end{remark}

Now we want to study  in more detail the polynomial solutions of $DF=\lambda F$. Let us assume that $\lambda=-n(n+2)$ with $n\in \NN_0$.
Let  $$F(s)=\sum_{i=0}^w F_i s^i$$ be a polynomial solution of degree $w$ of the equation $DF=\lambda F$.
We have that the coefficients $F_i$ are recursively defined by
$$F_{i+1}= (B+i)^{-1} M_{i}F_i= [B;C;-\Lambda_0+\lambda]_{i}F_0,$$ where
$M_i$ is the matrix defined in \eqref{Mw}.

The function $F$ is a polynomial of degree $w$ if and only if  there exists $F_0\in\CC^{\ell+1}$ such that
\begin{equation}\label{condition}
  F_w = [B;C;-\Lambda_0+\lambda]_{w}F_0 \neq 0 \quad \text{ and } \quad
M_w F_w =0.
\end{equation}
 As we said, the matrix $M_w$ is singular if and only if $\lambda=\lambda_w(k)$ for some $k$ such that $0\leq k\leq \ell$, and therefore, we have
\begin{equation}\label{w=n-k}
  w=n-k.
\end{equation}
Also, we observe that $M_w F_w =0$ if and only if $F_w$ is in the
subspace  generated by $e_{k}$ (the $k$-th vector of the canonical basis of $\CC^{\ell+1}$).

Now we want to prove that it is always possible to choose a vector $F_0\in \CC^{\ell+1}$ such that $[B;C;-\Lambda_0+\lambda]_w F_0=e_k$.
Recall that
$$[B;C;-\Lambda_0+\lambda]_w F_0=(B+w-1)^{-1} M_{w-1}\dots M_1B^{-1}M_0 F_0,$$
 and that, for $0\le i\le w$, the matrices $M_{w-i}$ are defined by
$$M_r= \sum_{j=0}^{\ell}(\lambda_w(k)-\lambda_{w-i}(j))E_{jj}.$$
In particular the kernel of the matrix $M_{w-i}$ is  $\CC e_{k+i}$ for $0\leq i\leq \min\{w,\ell-k \}$, because  $\lambda_w(k)-\lambda_{w-i}(j)=0$
if and only if $j-i=k$ (see Remark \ref{igualdadautoval}).

Let $W_k$ be the subspace in $\CC^{\ell+1}$ generated by $\{e_0,e_1,\dots,e_{k}\}$.
We observe that for every $j\in\NN_0$ we have that $W_k$ is invariant by
$\left(B+j\right)^{-1}$ because it is an upper triangular matrix.
For $j<w$, $M_j$ is a diagonal matrix whose first $k+1$ entries are not zero, thus the restriction of $M_j$ to $W_k$ is invertible.
Therefore, there exists  $F_0$   such that $[B;C;-\Lambda_0+\lambda]_w F_0=e_k$.
Then
$$F(u)={}_2\!H_1\left(\begin{smallmatrix}C,-\Lambda_0+\lambda\\B \end{smallmatrix}; u\right)F_0$$ is a vector polynomial of
degree $w$.
We observe that $F_0$ is unique in $ W_k $, but not in $\CC^{\ell+1}$. Anyhow,
the $k$-th entry of $F$ is a polynomial of degree $w$, and all the other entries are of
 lower degrees because the leading coefficient $F_w$ is always a multiple of $e_k$ .

\smallskip
 In this way, we have obtained the following results. In the first one we fix the eigenvalue $\lambda=-n(n+2)$ with $n\in N_0$ while, in the second one we fix the degree $w$ of the polynomial $F$.

\begin{prop}\label{grado}
Let $n\in \NN_0$ and $\lambda=-n(n+2)$.
If $P$ is a polynomial solution of $DF= \lambda F$ of degree $w$, then $n-\ell \leq w \leq n$.\\
Conversely, for every $w\in \NN_0$ such that $n-\ell \leq w \leq n$, the equation $DF=\lambda F$ has a polynomial
solution of degree $w$.  Moreover, if $w=n-k$, $0\leq k \leq \ell$,
the leading coefficient of any polynomial solution of $DF=\lambda F$ is a multiple of $e_k$.
\end{prop}
\begin{proof}
From \eqref{w=n-k} we have that there exists a polynomial solution of degree $w$ if and only if $w=n-k$, with $0\leq k\leq \ell$.
In such a case, we have proved that there exists  $F_0\in \CC^{\ell+1}$ such that  \eqref{condition} holds and we have that $F_w$ is a multiple of $e_k$.
\end{proof}

\begin{prop}\label{corgrado}
Given $w \in \NN_0$ there exist exactly $\ell +1$ values of $\lambda$
such that $DF=\lambda F$ has a polynomial solution of degree $w$, more precisely
$$\lambda= \lambda_w(k)=-(k+w)(k+w+2), \quad 0\leq k\leq \ell.$$
For each $k$ the leading coefficient of any polynomial solution of $DF=\lambda_w(k) F$ is a multiple of $e_k$, the $k$-th vector in the canonical basis of $\CC^{\ell+1}$.

\end{prop}

\smallskip

\subsection{Our sequence of matrix orthogonal polynomials.}\label{MVOP}
\

The matrix polynomials  $$\widetilde P_w(u)=\Psi(u)^{-1} P_w(u)$$ were introduced in \eqref{Pwtilde}.

\begin{prop}\label{columns}
The columns $\{\widetilde P_w^k\}_{k=0,\dots,\ell}$ of $\widetilde
P_w$ are  polynomials of degree $w$. Moreover,
\begin{align*}
  \deg \left(\widetilde P_w^k\right)_k = w\quad \text{ and } \quad
  \deg \left(\widetilde P_w^k\right)_j< w, \text{ for }  j\neq k.
\end{align*}
\end{prop}

\begin{proof}
The $k$-th column  of the matrix $\widetilde P_w=\Psi^{-1}P_w$ is the vector $\widetilde P_w^k=\Psi^{-1}P_w^k$, where $P_w^k$ is the $k$-th column of $P_w$.
From Corollary \ref{tildePweigenfunction} we have that $\widetilde P_w^k$  is a polynomial function that satisfies
$$\widetilde D \widetilde P_w^k =\lambda_w(k) \widetilde P_w^k, \qquad \widetilde E \widetilde P_w^k =\mu_w(k) \widetilde P_w^k,$$ for
$\lambda_w(k)=-(w+k)(w+k+2)$ and $\mu_w(k)=w(\tfrac\ell2-k)-k(\tfrac\ell2+1).$

If $w'$ denotes the degree of $\widetilde P_w^k $,  then
we have that $w+k-\ell\leq w'\leq w+k$ (see  Proposition \ref{grado}).
Hence we  write $$\widetilde P_{w}^{k}=\sum_{j=0}^{w'}  A_j u^j \qquad \text{ with } A_j\in\CC^{\ell+1}.$$

Moreover, from Proposition \ref{corgrado} we have that the corresponding eigenvalue of $D$ should be equal to $\lambda_{w'}(k') = -(w'+k')(w'+k'+2) $, with $0\leq k' \leq \ell$, and the  leading coefficient $A_{w'}$ has all its entries equal to zero, except for the $k'$-th one.
From Remark \ref{igualdadautoval} we obtain that
$$w-w'= k'-k.$$

On the other hand, $\widetilde P_w^k$ satisfies $\widetilde E \widetilde P_w^k =\mu_w(k) \widetilde P_w^k$, where
  $$\widetilde E F= (uR_2 + R_1)  F'+ M_0 F $$
  is the differential operator given in Theorem \ref{hyp}.
Then, the coefficients of the polynomials $\widetilde P_w^k$ satisfy
$$\big (j R_2 + M_0-\mu_w(k) \big) A_j  +(j+1) R_1 A_{j+1} =0,\qquad
\text{ for } 0\leq j \leq w',
$$
denoting   $A_{w'+1}=0$. In particular, for $j=w'$ we have
\begin{equation}\label{aux2}
\big({w'} R_2  +M_0- \mu_w (k)\big ) A_{w'}=0.
\end{equation}
From Theorem \ref{hyp} we have
\begin{align*}
{w'} R_2  +M_0- \mu_w (k)I &=\sum_{j=0}^\ell \big( w'(\tfrac \ell 2 -j)-j(\tfrac \ell 2 +1)-\mu_w(k) \big) E_{jj}\\
& = \sum_{j=0}^\ell \big( w'(\tfrac \ell 2 -j)- w(\tfrac \ell 2 -k)+(k-j)(\tfrac \ell 2 +1) \big) E_{jj}.
\end{align*}
 From equation \eqref{aux2} we have that the $k'$-th entry of the matrix ${w'} R_2  +M_0- \mu_w (k) I$ must be zero, then
$$0= w'(\tfrac \ell 2 -k')- w(\tfrac \ell 2 -k)+(k-k')(\tfrac \ell 2 +1).$$
Since $w-w'= k'-k$, we have
$ 0=(w-w')(1+k+w)$,
which implies that $w'=w$ and $k'=k$.

Therefore, $\widetilde P_w^k$ is a polynomial of degree $w$ and the only non zero entry of the leading coefficient of $\widetilde P_w^k$ is the $k$-th one.
\end{proof}

\subsection{The Inner Product} \label{the inner product}
\

Given a finite dimensional irreducible representation
$\pi=\pi_{\ell}$ of $K$ in the vector space $V_\pi$, let
$(C(G)\otimes \End (V_\pi))^{K\times K}$ be the space of all
continuous functions $\Phi:G\longrightarrow \End(V_\pi)$  such that
$\Phi(k_1gk_2)=\pi(k_1)\Phi(g)\pi(k_2)$ for all $g\in G$,
$k_1,k_2\in K$. Let us equip $V_\pi$ with an inner product such that
$\pi(k)$ becomes unitary for all $k\in K$. Then, we introduce an
inner product in the vector space $(C(G)\otimes \End
(V_\pi))^{K\times K}$ by defining
\begin{equation}\label{pi}
\langle \Phi_1,\Phi_2 \rangle =\int_G \tr ( \Phi_1(g)\Phi_2(g)^*)\, dg\, ,
\end{equation}
where $dg$ denotes the Haar measure of $G$ normalized by $\int_G
dg=1$, and $\Phi_2(g)^*$ denotes the adjoint of $\Phi_2(g)$ with
respect to the inner product in $V_\pi$.

By using Schur's orthogonality relations for the unitary irreducible representations
of $G$, it follows that if $\Phi_1$ and $\Phi_2$ are non equivalent irreducible spherical functions, then
they are orthogonal with respect to the inner product $\langle\cdot ,\cdot\rangle$, i.e. $$\langle \Phi_1,\Phi_2 \rangle =0.$$
In particular, if $\Phi_1$ and $\Phi_2$ are two irreducible spherical
functions of type $\pi=\pi_\ell$, we write as above (see \eqref{defHg})  $\Phi_1=H_1\Phi_\pi$ and
$\Phi_2=H_2\Phi_\pi$
 and put $$H_1(u)=(h_0(u),\cdots, h_\ell(u))^t, \qquad H_2(u)=(f_0(u),\cdots, f_\ell(u))^t,$$
 as we did in Subsection \ref{reconstruccion}.

\begin{prop}\label{prodint}  If $\Phi_1, \Phi_2\in \left(C(G)\otimes \End (V_\pi)\right)^{K\times K}$ then
\begin{equation*}
\langle \Phi_1,\Phi_2 \rangle = \frac{2}{\pi}\int_{-1}^1
\sqrt{1-u^2}\sum_{j=0}^\ell {h_j(u)}\overline{f_j(u)}\, du=  \frac{2}{\pi}\int_{-1}^1
\sqrt{1-u^2} H_2^*(u) H_1(u) \, du.
\end{equation*}
\end{prop}
\begin{proof}
Let us consider the element $E_1=E_{14}-E_{41}\in \lieg$. Then, as
$\so(4)_\CC\simeq \mathfrak{sl}(2,\CC)\oplus\mathfrak{sl}(2,\CC)$,
$\text{ad }E_1$ has $0$ and $\pm i$ as eigenvalues with multiplicity 2.

Let $A=\exp \RR E_1$ be the Lie subgroup of $G$ of all elements of
the form	
$$a(t)= \exp tE_1=\left(\begin{matrix} \cos t&0& 0&
\sin t\\ 0&1&0&0	\\ 0&0&1&0\\ -\sin t&0& 0&\cos
t\end{matrix}\right)\, ,	\qquad t\in \RR.$$

Now Theorem 5.10, page 190 in \cite{He} establishes that for every
$f\in C(G/K)$ and a suitable $c_*$
$$\int_{G/K} f(gK)\,dg_K=c_*\int_{K/M}\Big(\int_{-\pi}^{\pi}
\delta_*(a(t))f(ka(t)K)\,dt\Big)\,dk_M\,,$$ where the function $\delta_*:A\longrightarrow \RR $ is defined by
$$\delta_*(a(t))=\prod_{\nu\in\Sigma^+} |\sin it \nu(E_1)|,$$ and  $dg_K$ and $dk_M$
are respectively the left invariant measures on $G/K$ and
 $K/M$ normalized by $\int_{G/K} dg_K=\int_{K/M} dk_M=1$.
 Recall that $M$ was introduced in \eqref{Msubgrupo} and coincides with  the centralizer of $A$ in $K$.
  In our case we have $\delta_*(a(t))=\sin^2t $.

Since the function $g\mapsto \tr(\Phi_1(g)\Phi_2(g)^*)$ is invariant
under left and right multiplication by elements in $K$, we have
\begin{align}\label{pi2}
\langle \Phi_1,\Phi_2\rangle = c_* \int_{-\pi}^{\pi} \sin ^2t
\,\tr\left( \Phi_1(a(t))\Phi_2(a(t))^*\right)\,dt.
\end{align}
Also, for each
$t\in [-\pi,0]$, we have that $(I-2(E_{11}+E_{22}))a(t)(I-2(E_{11}+E_{22}))=a(-t)$, with
$I-2(E_{11}+E_{22})$ in $K$. Then we have
\begin{align*}
\langle \Phi_1,\Phi_2\rangle = 2c_* \int_{0}^{\pi} \sin ^2t
\,\tr\left( \Phi_1(a(t))\Phi_2(a(t))^*\right)\,dt.
\end{align*}

By the definition of the auxiliary function $\Phi_\pi(g)$ (see Subsection \ref{auxiliar}), we have that
$\Phi_1(a(t))\Phi_2(a(t))^*=H_1(a(t))H_2(a(t))^*$. Therefore, making the change of variables $\cos(t)=u$, we have
$$\langle \Phi_1,\Phi_2\rangle = 2c_* \int_{-1}^{1} \sqrt{1-u²}\sum_{j=0}^\ell
 {h_j}(u)\overline{{f_j(}u)}du.$$

\noindent To find the value of $c_*$ we consider the trivial case $\Phi_1=\Phi_2=I$ in (\ref{pi}) and (\ref{pi2}). Therefore, we obtain
$$\ell+1= c_* \int_{-\pi}^{\pi} \sin ^2t
\,(\ell+1)\,dt.$$
Then, we get $c_*=\pi^{-1}$ and the proposition follows.
\end{proof}

In Theorems \ref{puntos} and \ref{hyp} we conjugate the differential operators $D$ and $E$ to  hypergeometric operators $\widetilde D$ and $\widetilde E$ given by
$$ \widetilde D=(UT(u)\Psi(u))^{-1} D  (UT(u)\Psi(u))\quad \text{and} \quad \widetilde E=(UT(u)\Psi(u))^{-1} E  (UT(u)\Psi(u)).$$

\noindent Therefore, in terms of the functions
$$\widetilde P_1=(U T(u) \Psi(u))^{-1}H_1 \text{ and } \widetilde P_2=(U T(u) \Psi(u))^{-1}H_2,$$
we have
\begin{equation*}
\langle \widetilde P_1,\widetilde P_2\rangle_W=\int_{-1}^1 {\widetilde P}_2(u)^*\,W(u)\widetilde P_1(u)\,du,
\end{equation*}
where the weight matrix $W(u)$ is given by
\begin{equation}\label{LDU}
       W(u)=\frac2\pi\sqrt{1-u^2}\Psi^*(u)T^*(u)U^*U T(u) \Psi(u).
      \end{equation}
\color{black} From \eqref{Hahnortogon} we notice that $U^*U$ is a diagonal matrix by the orthogonality of the Hahn polynomials, %\cite{KS}, equation (1.5.2)),
\color{black} precisely
$$U^*U=\sum_{j=0}^\ell  \frac{(j+\ell+1)!(\ell-j)!}{(2j+1)\,\ell!\, \color{black}\ell!}\, E_{jj}.$$
\color {blue} Since $T(u)=\sum_{j=0}^\ell (1-u^2)^{j/2}$ we have
 $$T^*(u)U^*U T(u)=\sum_{j=0}^\ell  \frac{(j+\ell+1)!(\ell-j)!}{(2j+1)\,\ell!\, \ell!} (1-u^2)^j\, E_{jj}.$$

\color{black}
\remark  Since $\Psi(u)$ is an upper triangular matrix, see \eqref{psi}, we observe that the decomposition
 \eqref{LDU} easily leads to the LDU-decomposition of the weight matrix $W(u)$.

 % Therefore, from \eqref{LDU} we can obtain the LDU-decomposition of the weight matrix $W(u)$ because $\Psi(u)$ is upper triangular, see \eqref{psi}.

  %\color{black}

\smallskip
 We recall that $\Psi(u)$ is polynomial in $u$, that $U$ is a constant matrix and that $T(u)=\sum_{j=0}^\ell (1-u^2)^{j/2} E_{jj}$. Then  it follows that $W(u)$ is a continuous function on the closed interval $[-1,1]$. Thus, $W$ is a weight matrix on $[-1,1]$ with finite moments of all orders.

\color{black}
One may be interested in the reducibility of the weight:
\begin{definition}
A $n\times n$  matrix weight function $W$ supported on the interval $(a,b)\subset \RR$, reduces to a smaller size if there exists a $n\times n$ matrix $R$ such that
$$W(u)=R\left(\begin{matrix}W_1&0\\0&W_2
\end{matrix} \right)R^*,\qquad \text{for all } u\in (a,b),$$
with $W_1$ and $W_2$ weight matrices of lower size.
\end{definition}
\begin{prop}
A $n\times n$ matrix weight function $W$ supported on the interval $(a,b)\subset \RR$ reduces to a smaller size if the commutant
$$\left\{A\in M_{n\times n}: AW(u)=W(u)A, \quad \text{\rm for all } u\in(a,b) \right\}$$
is not trivial.
\end{prop}
\begin{proof}
%If $W$ reduces then $W(u)=R\left(W_1(u)\oplus W_2(u)\right)R^*$, for all $u\in (a,b)$; therefore the projection from $\CC^n$ onto $\CC^{n_1}$ is in the commutant.
If $A$ is in the commutant also is $A^*$, assume that $A$ is not a scalar multiple of the identity $I$. Then if $A+A^*=cI$, for any $c\in \RR$, we take $B=iA+ic/2\, I$, otherwise we take $B=A+A^*$, having then that $B$ is not a scalar multiple of the identity in the commutant with $B=B^*$. Hence, by the spectral theorem we have a projection $P\neq I$ which is a polynomial of $B$ and then $P$ and $Q=I-P$ are in the commutant. Therefore $W(u)=(P+ Q)W(u)(P+ Q)=PW(u)P+ QW(u)Q$, for all $u\in(a,b)$.

Let us define the matrix $R$ as the matrix whose first columns are the vectors of an orthogonal basis of $P(\CC^n)$ and the last  columns are the vectors of an orthogonal basis of $Q(\CC^n)$. Therefore, for all $u\in (a,b)$ we have
$$R^*W(u)R=R^*PW(u)PR+ R^*QW(u)QR=\left(\begin{smallmatrix}I&0\\0&0
\end{smallmatrix} \right)
W(u)
\left(\begin{smallmatrix}I&0\\0&0
\end{smallmatrix} \right)
+
\left(\begin{smallmatrix}0&0\\0&I
\end{smallmatrix} \right)
W(u)\left(\begin{smallmatrix}0&0\\0&I
\end{smallmatrix} \right).$$
Hence $W$ reduces to a smaller size.
\end{proof}

In \cite{KPR12a} it is proved that the commutant of our weight is a $2$-dimensional vector space (see Proposition 5.5), therefore $W$ reduces to a smaller size. See Theorem 6.5 in \cite{KPR12a}.

\color{black}
\smallskip
Consider now the sequence of matrix polynomials $(\widetilde P_w(u))_{w\geq0}$ introduced in \eqref{Pwtilde}.
The $k$-th column of  $\widetilde P_w(u)$  is given by a vector
$\widetilde P_w^k(u)$ associated to the irreducible spherical function of type $\pi_\ell$
$$\Phi^{({w}+\ell/2,-{k}+\ell/2)}_{\ell}.$$

\noindent Therefore
$\widetilde P_w^k$ and $\widetilde P_{w'}^{k'}$ are orthogonal with respect to $W$, i.e.
\begin{equation}\label{columnortog}
  \langle \widetilde P_w^k, \widetilde P_{w'}^{k'} \rangle_W=0 \qquad  \text{if  $(w,k)\neq(w',k')$}.
\end{equation}
In other words, this sequence of matrix-valued polynomials squarely fits within Krein's theory, and we obtain the following theorem.

\begin{thm}
 The matrix polynomial functions  $\widetilde P_w$, $ w\geq0$, form a sequence of orthogonal
polynomials with respect to $ W$, which are eigenfunctions of the symmetric differential operators $\widetilde D$ and $\widetilde E$ appearing in Theorem \ref{hyp}.
Moreover,
 $$ \widetilde D \widetilde P_w =\widetilde P_w \Lambda_w \qquad \text{and }\qquad
\widetilde E \widetilde P_w =\widetilde P_w M_w,$$
where $\Lambda_w= \sum_{k=0}^\ell\lambda_w(k)E_{kk}$, and  $M_w= \sum_{k=0}^\ell \mu_w(k)E_{kk}$, with
\begin{align*}
\lambda_w(k)=-(w+k)(w+k+2) \qquad \text{and}\qquad \mu_w(k)=w(\tfrac\ell2-k)-k(\tfrac\ell2+1).
\end{align*}
\end{thm}

\begin{proof}
From Proposition \ref{columns} we obtain that each  column of $\widetilde P_w$ is a polynomial function of degree $w$.
Moreover, $\widetilde P_w$ is a polynomial  whose  leading coefficient is
a  nonsingular diagonal matrix.

Given $w$ and $w'$, non negative integers, by using \eqref{columnortog} we have
\begin{align*}
\langle \widetilde P_w,\widetilde P_{w'} \rangle _W&=
\int_{-1}^1 \widetilde P_w(u)^*W(u) \widetilde P_{w'}(u) \, du\\
& =\sum_{k,k'=0}^\ell \int_{-1}^1 \Big( \widetilde P_w^k(u)^*W(u)
\widetilde P_{w'}^{k'}(u) \, du \Big)\, E_{k,k'}\\
& =\sum_{k,k'=0}^\ell   \delta_{w,w'}\delta_{k,k'}   \Big(\int_{-1}^1 \widetilde P_w^k(u)^*W(u) \widetilde P_{w'}^{k'}(u)
\, du \Big)\, E_{k,k'}\\
& =\delta_{w,w'}\sum_{k=0}^\ell   \int_{-1}^1 \Big(\widetilde P_w^k(u)^*W(u) \widetilde P_{w'}^{k}(u) \, du,\Big) \, E_{k,k},
\end{align*}
which proves the orthogonality. Even more, it also shows us that $\langle \widetilde P_w,\widetilde P_{w} \rangle _W$ is a diagonal matrix. Now, thanks to Corollary \ref{tildePweigenfunction} it only remains to prove that the operators $\widetilde D$ and $\widetilde E$ are symmetric with respect to $W$.

Making a few simple computations we have that
$$\langle\widetilde D \widetilde P_w,\widetilde P_{w'}\rangle=\delta_{w,w'}\langle \widetilde P_w,\widetilde P_{w'}\rangle \Lambda_w
  =\delta_{w,w'} \Lambda_w^* \langle \widetilde P_w,\widetilde P_{w'}\rangle =\langle \widetilde P_w,\widetilde D\widetilde P_{w'}\rangle,$$
for every $w,w'\in\NN_0$, because $\Lambda_w$ is real and diagonal.
This concludes the proof of the theorem.
\end{proof}

%\color{black}

%For the benefit of the reader, and suggested by the referee, we observe some comparisons between last sections and \cite{KPR12b}: the function $L$ in Section 2 of that work is, up to a diagonal matrix, equal to the transposed of our function $\Psi$ (see \eqref{psi}), and our weight $W$ is equal to the weight $W$ in that paper.

\color{black}
\begin{remark}\label{comparacion}
Following the  suggestion of the referee, we include here some comparisons between the results in the last sections of this paper  and those in \cite{KPR12b}.

%First of all we point out that our parameter $\ell$ corresponds to $2\ell$ in \cite{KPR12b}. To avoid confusions we  denote their $\ell$ by $\tilde \ell$, i.e.
%$\ell=2\tilde \ell$.
%Our variable $u\in [-1,1]$ corresponds to variable $x$ in \cite{KPR12b}.
% Our variable $s$, occasionally used in Section \ref{DH}, is the variable $u\in[0,1]$  in \cite{KPR12b}.
 Our differential operators $\widetilde D$ and $\widetilde  E$ introduced in Theorem \ref{hyp} are related to the operators $\tilde D$ and $\tilde E$ introduced in Theorem 3.1 of \cite{KPR12b}.
 Unfortunately the expression of  $\tilde E$ in Theorem 3.1 of \cite{KPR12b} is wrong. By starting from the differential operators $D$ and $E$ in Theorem 4.1 of \cite{KPR12b} and changing the variables we obtain the following relations
$$\widetilde D-2\widetilde E= \big ( \tilde D\big) ^t, \qquad -2\widetilde E=  \ell \big (\tilde E\big )^t + \ell( \ell+2)I. $$

On the other hand the operators $\mathcal D$ and $\mathcal E$, defined respectively  in Proposition 6.1 and (6.13) of \cite{KPR12b},  are related to our operators $\bar D$ and $\bar E$ defined in \eqref{HP} by
$$\bar D= \big(\mathcal D\big)^t-\ell(\ell+2)I, \qquad -2\bar E= C\big (\ell \mathcal E + \ell(\ell+2)I\big )^t C^{-1},$$
where $C= \sum_{j=0}^\ell \tfrac{(-2i)^j j! j!}{(2j)!} E_{jj}.$ % and $\alpha_j= \frac{(-2i)^j j! j!}{(2j)!}$.

Our weight $W(u)$ introduced in \eqref{LDU} is equal to the weight $W(x)$ in  Theorem 2.1 of  \cite{KPR12b}.
The function $L(x)$ given there is related to our function $\Psi(u)$ (see \eqref{psi}) by
 %up to a diagonal matrix, equal to the transposed of our function $\Psi$ (see \eqref{psi}),
$$\Psi(u)= C L(u)^t.$$
Finally the matrix polynomial function $P_w(u)$ given in \eqref{pes} are related to the polynomial $\mathcal R_n$ in \cite{KPR12b} by
$$P_w(u)=C\mathcal R_w^t(\tfrac{1-u}{2}).$$
\end{remark}

\color{black}

\color{black}

\section{Appendix}
The purpose of this section is to give the proofs of Propositions \ref{D3} and \ref{E3}.
\color{black}

\begin{propD3}
For any $H\in C^\infty( \RR^3)\otimes
\End(V_\pi)$ we have
\begin{align*}
D(H)& (y) =  (1+\left\|y\right\|^2)\Big((y_1^2+1) H_{y_1y_1}+(y_2^2+1) H_{y_2y_2}+(y_3^2+1) H_{y_3y_3}\\
 &+ 2(y_1y_2 H_{y_1y_2}+y_2y_3 H_{y_2y_3}+y_1y_3 H_{y_1y_3}) + 2(y_1 H_{y_1}+y_2 H_{y_2}+y_3 H_{y_3})
\Big).
\end{align*}
\end{propD3}

\begin{proof}%[Proof of Proposition \ref{D3}]
From \eqref{Ddefuniv} we have $D(H)=Y_4^2(H)+Y_5^2(H)+Y_6^2(H).$
We need to give the expressions of this operators in the coordinate system
 $p:(S^3)^{+}\longrightarrow \RR^3$  %introduced in \eqref{pfunction} and given by}
\begin{equation*}
p(x)=\left(\frac{x_1}{x_4},\frac{x_2}{x_4},\frac{x_3}{x_4}\right)=(y_1,y_2,y_3).
 \end{equation*}
%we begin by calculating $Y_4^2(H)(y)$, for all $y\in {\RR^3}$.

\noindent We have
$$Y^2(H)(g)= \frac{d}{ds}\,\frac{d}{dt} H\Big(\tilde p\big(g(\exp (s+t)Y)\big)\Big) _{\mid_{s=t=0}},$$
where $\tilde p:G\longrightarrow \RR^3$, $\tilde p(g)=p(gK)$.

From Subsection 2.3 we have  $Y_4=E_{1,4}-E_{4,1}$, $Y_5=E_{2,4}-E_{4,2}$ and $Y_6=E_{3,4}-E_{4,3}$  then
\begin{align*}
\tilde p(g(\exp (s+t)Y_4))&=(u_1(s+t)\, , \, v_1(s+t)\, , \, w_1(s+t)),\\
\tilde p(g(\exp (s+t)Y_5))&=(u_2(s+t)\, , \, v_2(s+t)\, , \, w_2(s+t)),\\
\tilde p(g(\exp (s+t)Y_6))&=(u_3(s+t)\, , \, v_3(s+t)\, , \, w_3(s+t)).
\end{align*}
where
%\beg in{align*}
$$u_j(s+t)=\frac{g_{1j}\sin(s+t)+g_{14}\cos(s+t)}{g_{4j}\sin(s+t)+g_{44}\cos(s+t)},\,
v_j(s+t)=\frac{g_{2j}\sin(s+t)+g_{24}\cos(s+t)}{g_{4j}\sin(s+t)+g_{44}\cos(s+t)},$$
$$w_j(s+t)=\frac{g_{3j}\sin(s+t)+g_{34}\cos(s+t)}{g_{4j}\sin(s+t)+g_{44}\cos(s+t)},$$
for $j=1,2,3$.

\noindent
By using the chain rule we have %for each $k=4,5,6$.
%\begin{align*}
%Y_k^2(H)(g)=& H_{y_1y_1}\,\frac{\partial u_k}{\partial s}\,\frac{\partial u_k}{\partial t}+H_{y_2y_2}\,\frac{\partial u_k}{\partial s}\,\frac{\partial u_k}{\partial t}+H_{y_3y_3}\,\frac{\partial u_k}{\partial s}\,\frac{\partial u_k}{\partial t}\\
%&+2\left(H_{y_1y_2}\,\frac{\partial u_k}{\partial s}\,\frac{\partial v_k}{\partial t}+H_{y_1y_3}\,\frac{\partial u_k}{\partial s}\,\frac{\partial w_k}{\partial t}+H_{y_2y_3}\,\frac{\partial v_k}{\partial s}\,\frac{\partial w_k}{\partial t}\right)\\
%&+H_{y_1}\,\frac{\partial^2 w_k}{\partial s\partial t}+H_{y_2}\,\frac{\partial^2 w_k}{\partial s\partial t}+H_{y_3}\,\frac{\partial^2 w_k}{\partial s\partial t}.
%\end{align*} Then
\begin{align*}
 D(H)(g)&= Y_4^2(H)(g)+ Y_5^2(H)(g)+ Y_6^2(H)(g) \\
 &=\sum_{j=1}^3 \left[ H_{y_1y_1}\,\frac{\partial u_j}{\partial s}\,\frac{\partial u_j}{\partial t}
 +H_{y_2y_2}\,\frac{\partial v_j}{\partial s}\,\frac{\partial v_j}{\partial t}
 +H_{y_3y_3}\,\frac{\partial w_j}{\partial s}\,\frac{\partial w_j}{\partial t}\right.\\
 &\quad +H_{y_1y_2}\Big(\frac{\partial u_j}{\partial s}\,\frac{\partial v_j}{\partial t}+ \frac{\partial u_j}{\partial t}\,\frac{\partial v_j}{\partial s} \Big)
 +H_{y_1y_3}\Big( \frac{\partial u_j}{\partial s}\,\frac{\partial w_j}{\partial t}+ \frac{\partial u_j}{\partial t}\,\frac{\partial w_j}{\partial s}\Big) \\
 &
\quad  +H_{y_2y_3}\Big(\frac{\partial v_j}{\partial s}\,\frac{\partial w_j}{\partial t} + \frac{\partial v_j}{\partial t}\,\frac{\partial w_j}{\partial s}\Big) % \\&\quad
 \left.+ H_{y_1}\,\frac{\partial^2 u_j}{\partial s\partial t}+H_{y_2}\,\frac{\partial^2 v_j}{\partial s\partial t}+H_{y_3}\,\frac{\partial^2 w_j}{\partial s\partial t}\right].
\end{align*}

We observe that
\begin{align*}
  \frac{\partial u_j}{\partial s}& =\frac{\partial u_j}{\partial t}= \frac{g_{1j}g_{44}-g_{14}g_{4j}}{g_{44}^2} \, , \qquad \quad
  \frac{\partial^2 u_j}{\partial s\partial t}= \frac{2g_{4j}(g_{14}g_{4j}-g_{1j}g_{44})}{g_{44}^3},
  \\
  \frac{\partial v_j}{\partial s}& =\frac{\partial v_j}{\partial t}= \frac{g_{2j}g_{44}-g_{24}g_{4j}}{g_{44}^2} \, , \qquad \quad
  \frac{\partial^2 v_j}{\partial s\partial t}= \frac{2g_{4j}(g_{24}g_{4j}-g_{2j}g_{44})}{g_{44}^3},
  \\
  \frac{\partial w_j}{\partial s}& =\frac{\partial w_j}{\partial t}= \frac{g_{3j}g_{44}-g_{34}g_{4j}}{g_{44}^2}\, , \qquad \quad
  \frac{\partial^2 w_j}{\partial s\partial t}= \frac{2g_{4j}(g_{34}g_{4j}-g_{3j}g_{44})}{g_{44}^3}.
\end{align*}
Now we observe that $y=(y_1,y_2,y_3)=\big(\frac{g_{14}}{g_{44}},\frac{g_{24}}{g_{44}},\frac{g_{34}}{g_{44}}\big)$ and recall that
$g=(g_{jk})$ is a matrix in $\SO(4)$, therefore its rows are orthonormal vectors.
 In particular we   have
 $$\frac 1{g_{44}^2}=1+y_1^2+y_2^2+y_3^2=1+\| y\|^2.$$
  %for $0\leq j,k\leq4$ are the entries of a matrix in $\SO(4)$
  Now the proposition follows after some straightforward computations.
\end{proof}

\begin{propE3}
For any $H\in C^\infty( \RR^3)\otimes
\End(V_\pi)$
we have
\begin{align*}
&E(H)(y)=H_{y_1}\dot\pi \left(\begin{smallmatrix} 0 &-y_2-y_1y_3 &-y_3+y_1y_2 \\y_2+y_1y_3 &0 &-1-y_1^2 \\ y_3-y_1y_2&1+y_1^2 &0 \end{smallmatrix}\right)  \\
& \quad +H_{y_2}\dot\pi \left(\begin{smallmatrix} 0 &-y_2y_3+y_1 &1+y_2^2 \\y_2y_3-y_1 & 0 & -y_3-y_1 y_2 \\-1-y_2^2 & y_3+y_1y_2&0 \end{smallmatrix}\right) +H_{y_3}\dot\pi \left(\begin{smallmatrix} 0 &-1-y_3^2 &y_1+y_2y_3 \\1+y_3^2 & 0 & y_2-y_1 y_3 \\ -y_1-y_2y_3 & -y_2+y_1y_3 &0 \end{smallmatrix}\right).
\end{align*}
\end{propE3}
\begin{proof}%[Proof of Proposition \ref{E3}]
From \eqref{Edefuniv} we obtain
\begin{align*}
  E(H)&=\Big(-Y_4(H)Y_3(\Phi_\pi)+ Y_5(H)Y_2(\Phi_\pi)-Y_6(H)Y_1(\Phi_\pi) \Big)\Phi_\pi^{-1}\\
  & =\sum_{j=1}^3 (-1)^j\, Y_{3+j}(H) \, Y_{4-j}(\Phi_\pi) \Phi_\pi^{-1}.
\end{align*}
For $j=1,2,3$ we have %know that
\begin{align*}
Y_{3+j}(H)&=\frac{d}{ds}_{\mid {s=0}} H\Big(\tilde p(g(\exp s\,Y_{3+j}))\Big) %\\ &
=H_{y_1} \frac{\partial u_j}{\partial s}+H_{y_2}\frac{\partial v_j}{\partial s}+H_{y_3} \frac{\partial w_j}{\partial s},
\end{align*}
%for $k=4,5,6$;
where $u_j$, $v_j$ and $w_j$ are the functions introduced in the proof of Proposition \ref{D3}.

In the other hand we have that
\begin{align*}
(Y_k\Phi_\pi)(g)\Phi_\pi^{-1}(g)&= \frac{d}{dt}_{\mid t=0} \Phi_\pi(g(\exp tY_k))\Phi_\pi^{-1}(g)=\frac{d}{dt}_{\mid t=0}
\Phi_\pi(g(\exp tY_k)g^{-1})\\
& = \frac{d}{dt}_{\mid t=0} \pi\big ( a(g(\exp tY_k)g^{-1})\big )= \dot \pi \big( \dot a (gY_kg^t)\big) ,
\end{align*}
where $\dot a$ is the function introduced in \eqref{funcionq}. It is easy to check that %it is given by
$$\dot a (X)=\left(\begin{smallmatrix} 0 & -X_{12}-X_{34}&-X_{13}+X_{24} \\X_{12}+X_{34} &0 &-X_{23}-X_{14} \\ X_{13}-X_{24}&X_{23}+X_{14}&0 \end{smallmatrix}\right),\qquad \text{ for all } X\in \so(4).$$
Therefore, at $s=t=0$ we get
\begin{align*}
E&(H)(g) =\sum_{j=1}^3 \Big( H_{y_1}  \frac{\partial u_j}{\partial s} +H_{y_2}  \frac{\partial v_j}{\partial s} +H_{y_3} \frac{\partial w_j}{\partial s} \Big)
\dot \pi\big( \dot a (gY_3g^t)\big) \displaybreak[0]\\
&= H_{y_1} \dot \pi \Big (\dot a \big( \textstyle\sum_{j=1}^3(-1)^j \frac{\partial u_j}{\partial s}g \,Y_{4-j} \, g^t\big)\Big) %\\& \quad
+ H_{y_2} \dot \pi \Big (\dot a \big( \textstyle\sum_{j=1}^3(-1)^j \frac{\partial v_j}{\partial s}g\, Y_{4-j} \, g^t\big)\Big)\\
&\quad + H_{y_3} \dot \pi \Big (\dot a \big( \textstyle\sum_{j=1}^3(-1)^j \frac{\partial w_j}{\partial s}g\, Y_{4-j} \, g^t\big)\Big) \\
& = H_{y_1} \dot\pi (\dot a(B_1))+H_{y_2} \dot\pi (\dot a(B_2))+H_{y_3} \dot\pi (\dot a(B_3)).
\end{align*}

Now we recall that $Y_1=E_{12}-E_{21}$, $Y_2=E_{13}-E_{31}$, $Y_3=E_{23}-E_{32} $. It is easy to verify that
$$g(E_{ij}-E_{ji})g^t=\sum_{1\le k,r\le 4}g_{ki}g_{rj}\, E_{kr}.$$
By using the expressions of $\frac{\partial u_j}{\partial s}$,
$\frac{\partial v_j}{\partial s}$ and $\frac{\partial w_j}{\partial s}$ given in the proof of Proposition \ref{D3}, we obtain that % to compute
\begin{align*}
&B_1=\sum_{j=1}^3  (-1)^j \frac{\partial u_j}{\partial s}g \,Y_{4-j} \, g^t   =
\tfrac 1{g_{44}^2}  \sum_{1\le k,r\le4} \,\Big (  (-g_{11}g_{44}+g_{14}g_{41})(g_{k2}g_{r3}-g_{k3}g_{r2})\\
&\quad  + (g_{12}g_{44}-g_{14}g_{42})(g_{k1}g_{r3}-g_{k3}g_{r1})%\\ & \qquad \qquad \quad
+ (-g_{13}g_{44}+g_{14}g_{43})(g_{k1}g_{r2}-g_{k2}g_{r1})  \Big ) E_{kr}.
% \displaybreak[0]\\ & = \tfrac {g_{14}}{g_{44}^2}  \sum_{k,r} \Big (  (-g_{11}g_{44}+g_{14}g_{41})(g_{k2}g_{r3}-g_{k3}g_{r2})\\
\end{align*}
It is not difficult to verify that the $kr$-th entry of $B_1$ is equal to
\begin{align*}
(B_1)_{kr}= \frac 1{g_{44}^2}\Big (g_{44}\det \left( \begin{smallmatrix}
  g_{11}& g_{12}& g_{13}\\ g_{r1}& g_{r2}& g_{r3} \\g_{k1}& g_{k2}& g_{k3}
\end{smallmatrix}\right) -g_{14} \det \left( \begin{smallmatrix}
 g_{r1}& g_{r2}& g_{r3} \\g_{k1}& g_{k2}& g_{k3}\\  g_{41}& g_{42}& g_{43}
\end{smallmatrix}\right) \Big).
\end{align*}
Now we use the following fact: If $g$ is a matrix in $\SO(n)$  and   $g(i|j)$ denotes the matrix obtained from $g$ deleting the $i$-th row and the $j$-th column, then
\begin{equation}
  g_{ij}=(-1)^{i+j}\det \big(g(i|j)\big).
\end{equation}
Therefore
$$B_1= \frac 1{g_{44}^2}\left(\begin{smallmatrix}
  0& g_{14}g_{34} & -g_{14}g_{24} & 0 \\
  g_{14}g_{34} & 0 & g_{44}^2+g_{14}^2 & -g_{44}g_34\\
  -g_{14}g_{24} & g_{44}^2+g_{14}^2  &  0 & g_{44}g_{24}\\
  0 & -g_{44}g_{34} & -g_{44}g_{24}& 0
\end{smallmatrix}\right).
$$

\smallskip
\noindent We proceed in a similar way with $B_2=\sum_{j=1}^3  (-1)^j \frac{\partial v_j}{\partial s}g \,Y_{4-j} \, g^t $  and obtain
$$B_2= \frac 1{g_{44}^2}\left(\begin{smallmatrix}
  0& g_{24}g_{34} & -g_{44}^2-g_{24}^2 & g_{44}g_{34} \\
  -g_{24}g_{34} & 0 & g_{24}g_{14} & 0\\
  g_{44}^2+g_{24}^2 & -g_{24}g_{14}  &  0 & -g_{44}g_{14}\\
   -g_{44}g_{34} & 0 & g_{44}g_{14}& 0
\end{smallmatrix} \right).
$$

\noindent For
$B_3=\sum_{j=1}^3  (-1)^j \frac{\partial w_j}{\partial s}g \,Y_{4-j} \, g^t $ we get
$$B_3= \frac 1{g_{44}^2}\left(\begin{smallmatrix}
  0& g_{34}^2+g_{44}^2 & -g_{34}g_{24} & -g_{44}g_{24} \\
 -g_{34}^2-g_{44}^2 & 0 & g_{34}g_{14} & g_{44}g_{14}\\
  g_{34}g_{24} &- g_{34}g_{14}  &  0 & 0\\
  g_{44}g_{24} & -g_{44}g_{14} & 0& 0
\end{smallmatrix}\right).
$$

\noindent Therefore
\begin{align*}
E(H)
=&H_{y_1}\dot\pi \left(\begin{smallmatrix} 0 &-g_{24}g_{44}-g_{14}g_{34} &-g_{34}g_{44}+g_{14}g_{24} \\g_{24}g_{44}+g_{14}g_{34} &0 &-g_{44}^2-g_{14}^2 \\ g_{34}g_{44}-g_{24}g_{14}&g_{44}^2+g_{14}^2 &0 \end{smallmatrix}\right)\frac{1}{g_{44}^2}\\
&+H_{y_2}\dot\pi \left(\begin{smallmatrix} 0 &-g_{24}g_{34}+g_{14}g_{44} &g_{44}^2+g_{24}^2 \\g_{24}g_{34}-g_{14}g_{44} & 0 & -g_{34}g_{44}-g_{14} g_{24} \\-g_{44}^2-g_{24}^2 & g_{34}g_{44}+g_{14}g_{24}&0 \end{smallmatrix}\right)\frac{1}{g_{44}^2}\\ &+H_{y_3}\dot\pi \left(\begin{smallmatrix} 0 &-g_{44}^2-g_{34}^2 &g_{14}g_{44}+g_{24}g_{34} \\g_{44}^2+g_{34}^2 & 0 & g_{24}g_{44}-g_{14} g_{34} \\ -g_{14}g_{44}-g_{24}g_{34} & -g_{24}g_{44}+g_{14}g_{34} &0 \end{smallmatrix}\right)\frac{1}{g_{44}^2} .
\end{align*}

\smallskip
\noindent The proposition follows by observing that  $y_j=\frac {g_{j4}}{g_{44}}$, for $j=1,2,3$.
\end{proof}

\color{black}

\end{document}